\documentclass[reqno]{amsart}

\usepackage{latexsym,amssymb,amsmath,bm,bbm,graphicx,enumitem,verbatim,footmisc}
\usepackage{tikz}
\usepackage{hyperref}
\usetikzlibrary{positioning,decorations.pathmorphing,fit,petri,backgrounds,decorations.pathreplacing}
\usepackage[all]{xy}
\usepackage{float}

\usepackage[a4paper, margin=0.9in]{geometry}

\title{Small Models, Large Cardinals, and Induced Ideals}

\DeclareMathOperator{\IN}{\epsilon}

\DeclareMathOperator{\Lim}{Lim}

\DeclareMathOperator{\Ord}{Ord}

\DeclareMathOperator{\tcl}{\mathrm{tc}}

\DeclareMathOperator{\id}{id}

\DeclareMathOperator{\supp}{supp}

\DeclareMathOperator{\ZFC}{ZFC}

\newcommand{\crit}[1]{\mathrm{crit}({#1})}
\newcommand{\Set}[2]{\{#1~\vert~#2\}}

\newcommand{\pow}{\mathcal{P}}

\newcommand{\anf}[1]{{\text{``}\hspace{0.3ex}{#1}\hspace{0.3ex}\text{''}}}
\newcommand{\lanf}{{\text{``}\hspace{0.3ex}}}
\newcommand{\ranf}{{\hspace{0.02ex}\text{''}}}

\DeclareMathOperator{\calS}{\mathcal{S}}

\newcommand{\ran}[1]{\mathrm{ran}({#1})}
\newcommand{\map}[3]{{#1}:{#2}\longrightarrow{#3}}
\newcommand{\Map}[5]{{#1}:{#2}\longrightarrow{#3}; ~ {#4}\mapsto{#5}}
\newcommand{\PPP}{{\mathbb{P}}}

\newcommand{\RRR}{{\mathbb{R}}}
\newcommand{\SSS}{{\mathbb{S}}}
\newcommand{\LL}{{\rm{L}}}
\newcommand{\VV}{{\rm{V}}}
\newcommand{\II}{{\rm{I}}}
\newcommand{\NN}{{\rm{N}}}
\newcommand{\HH}[1]{{\rm{H}}(#1)}
\newcommand{\betrag}[1]{\vert{#1}\vert}
\newcommand{\otp}[1]{\rm{otp}({#1})}
\newcommand{\POT}[1]{{\mathcal{P}}({#1})}
\newcommand{\seq}[2]{\langle{#1}~\vert~{#2}\rangle}
\newcommand{\Add}[2]{{\rm{Add}}({#1},{#2})}
\newcommand{\Ult}[2]{{\mathrm{Ult}}({#1},{#2})}
\newcommand{\goedel}[2]{{\prec}{#1},{#2}{\succ}}
\newcommand{\delt}{\Delta}
\newcommand{\infi}{\boldsymbol{\infty}}
\newcommand{\cc}{\mathbf{cc}}
\newcommand{\scc}{\mathbf{sc}}
\newcommand{\triv}{\mathbf{T}}
\newcommand{\wf}{\mathbf{wf}}

\author{Peter Holy}
\address{University of Udine, Via delle Scienze 206, 33100 Udine, Italy}
\email{pholy@math.uni-bonn.de}

\author{Philipp L\"ucke}
\address{Institut de Matem\`{a}tica, Universitat de Barcelona, 
Gran via de les Corts Catalanes 585,
08007 Barcelona, Spain}
\email{pluecke@math.uni-bonn.de}


\thanks{The first author would like to thank Niels Ranosch and Philipp Schlicht for discussing material related to some of the contents of this paper at an early stage. 
 Both authors would like to thank Victoria Gitman for some helpful comments on an early version of this paper. 
  The authors would also like to thank the anonymous referee for a number of helpful comments on the paper. 
  During the revision of this paper, the research of the first author was supported by the Italian PRIN 2017 Grant \emph{Mathematical Logic: models, sets, computability}. 
  This project has received funding from the European Union’s Horizon 2020 research and innovation programme under the Marie Sk{\l}odowska-Curie grant agreement No 842082 of the second author (Project \emph{SAIFIA: Strong Axioms of
Infinity -- Frameworks, Interactions and Applications}).}

\begin{document}

\begin{abstract}
  We show that many large cardinal notions up to measurability can be characterized through the existence of certain %elementary embeddings between, or ultrafilters on, 
 filters for small models of set theory. 
 This correspondence will allow us to obtain a canonical way in which to assign ideals to many large cardinal notions. 
 This assignment coincides with classical large cardinal ideals whenever such ideals had been defined before. 
 Moreover, in many important cases, relations between these ideals reflect the ordering of the corresponding large cardinal properties both under direct implication and consistency strength. 
\end{abstract}

\maketitle

\theoremstyle{definition}
\newtheorem{definition}{Definition}[section]
\newtheorem{scheme}{Scheme}
\newtheorem{remark}[definition]{Remark}
\theoremstyle{plain}
\newtheorem{fact}[definition]{Fact}
\newtheorem{lemma}[definition]{Lemma}
\newtheorem{theorem}[definition]{Theorem}
\newtheorem{corollary}[definition]{Corollary}
\newtheorem{claim}[definition]{Claim}
\newtheorem*{claim*}{Claim}
\newtheorem{subclaim}[definition]{Subclaim}
\newtheorem{conjecture}[definition]{Conjecture}
\newtheorem{question}[definition]{Question}
\newtheorem{observation}[definition]{Observation}
\newtheorem{proposition}[definition]{Proposition}

\renewcommand*{\thescheme}{\Alph{scheme}}

%%%%%%%%%%%%%%%%%%%%%%%%%%%%%
%%%%%%%%%%%%%%%%%%%%%%%%%%%%%

\section{Introduction}

The work presented in this paper is motivated by the aim to develop general frameworks for large cardinal properties and their ordering under both direct implication and consistency strength. 
We develop such a framework for large cardinal notions up to measurability, that is based on the existence of set-sized models and ultrafilters for these models satisfying certain degrees of amenability and normality. 
 This will cover several classical large cardinal concepts like inaccessibility, weak compactness, ineffability, Ramseyness and measurability, and also many of the \emph{Ramsey-like} cardinals that are an increasingly popular subject of recent set-theoretic research (see, for example, \cite{mitchellforramsey}, \cite{brent}, \cite{MR3348040}, \cite{MR2830415}, \cite{MR2830435}, \cite{MR3800756}, \cite{nielsen-welch},  and \cite{MR2817562}), but in addition, it canonically yields a number of new large cardinal notions.
 %
% By considering the collections of subsets of the given large cardinal that are not contained in any of the relevant ultrafilters, 
We then use these large cardinal characterizations to canonically assign ideals to large cardinal notions, in a way that generalizes all such assignments previously considered in the set-theoretic literature, like the classical definition of the \emph{weakly compact ideal}, the \emph{ineffable ideal}, the \emph{completely ineffable ideal} and the \emph{Ramsey ideal}. 
 In a great number of cases, we can show that the ordering of these ideals under inclusion directly corresponds to the ordering of the underlying large cardinal notions under direct implication. 
Similarly, the ordering of these large cardinal notions under consistency strength can usually be read off these ideals %by considering whether the set of all smaller cardinals that are lacking the weaker large cardinal property is an element of the ideal associated to the stronger large cardinal property.
in a simple and canonical way.
 In combination, these results show that the framework developed in this paper provides a natural setting to study the lower reaches of the large cardinal hierarchy.

%%%%%%%%%%%%%%%%

\subsection{Characterization schemes}

Starting from measurability upwards, many important large cardinal notions are defined through the existence of certain ultrafilters that can be used in ultrapower constructions in order to produce elementary embeddings $\map{j}{\VV}{M}$ of the set-theoretic universe $\VV$ into some transitive class $M$ with the large cardinal in question as their critical point.  
A great variety of results shows that many prominent large cardinal properties below measurability can be characterized through the existence of filters that only measure sets contained in set-sized models $M$ of set theory.\footnote{In the following, we identify (not necessarily transitive) classes $M$ with the $\epsilon$-structure $\langle M,\in\rangle$. In particular, given some theory $T$ in the language of set theory, we say that a class $M$ is a model of $T$ (and write $M\models T$) if and only if $\langle M,\in\rangle\models T$.}
 For example, the equivalence of weak compactness to the \emph{filter property} (see \cite[Theorem 1.1.3]{MR0460120}) implies that an uncountable cardinal $\kappa$ is weakly compact if and only if for every model $M$ of $\ZFC^-$ of cardinality at most $\kappa$ that contains $\kappa$, there exists an uniform\footnote{All relevant definitions can be found in Section \ref{generalizations}.} $M$-ultrafilter $U$ on $\kappa$ that is ${<}\kappa$-complete in $\VV$.\footnote{I.e., arbitrary intersections of less than $\kappa$-many elements of $U$ in $\VV$ are nonempty.}  
  Isolating what was implicit in folklore results (see, for example, \cite{MR534574}), Gitman, Sharpe and Welch showed that Ramseyness can be characterized through the existence of countably complete ultrafilters for transitive $\ZFC^-$-models of cardinality $\kappa$ (see \cite[Theorem 1.3]{MR2830415} or \cite[Theorem 1.5]{MR2817562}). 
 More examples of such characterizations are provided by results of Kunen \cite{MR2617841}, Kleinberg \cite{MR513844} and Abramson--Harrington--Kleinberg--Zwicker \cite{MR0460120}. 
 Their characterizations can be formulated through the following scheme, which is hinted at in the paragraph before \cite[Lemma 3.5.1]{MR0460120}:  
 \emph{An uncountable cardinal $\kappa$ has the large cardinal property $\Phi(\kappa)$ if and only if for some (equivalently, for all) sufficiently large regular cardinal(s) $\theta$ and for some (equivalently, for all) countable $M\prec\HH{\theta}$ with $\kappa\in M$, there exists a uniform $M$-ultrafilter $U$ on $\kappa$ with the property that the statement $\Psi(M,U)$ holds}.
Their results show that this scheme holds true in the following cases: 
 \begin{itemize} 
    \item {\cite{MR0460120}} $\Phi(\kappa)\equiv\anf{\textit{$\kappa$ is inaccessible}}$ and $\Psi(M,U)\equiv\anf{\textit{$U$ is ${<}\kappa$-amenable and ${<}\kappa$-complete for $M$}}$.\footnote{Note that, here and in the following, in order to avoid mention of $\kappa$ within $\Psi(M,U)$, the cardinal $\kappa$ could be defined as being the least $M$-ordinal $\eta$ satisfying $\bigcup U\subseteq\eta$.}

    \item {\cite{MR2617841}} $\Phi(\kappa)\equiv\anf{\textit{$\kappa$ is weakly compact}}$ and $\Psi(M,U)\equiv\anf{\textit{$U$ is $\kappa$-amenable and ${<}\kappa$-complete for $M$}}$. 

    \item {\cite{MR513844}} $\Phi(\kappa)\equiv\anf{\textit{$\kappa$ is completely ineffable}}$ and $\Psi(M,U)\equiv\anf{\textit{$U$ is $\kappa$-amenable for $M$ and $M$-normal }}$. 
  \end{itemize}

%The work presented in our paper is motivated by the question  whether the consideration of models of greater cardinality and of stronger properties of filters leads to new characterizations of large cardinal properties. 
% %
%Since the original proofs of the above results %presented in \cite{MR0460120} 
%strongly use the existence of generic filters over countable models of set theory, such characterizations have to be based on very different proof techniques, which will be developed in the course of this paper. 

Generalizing the above scheme, our large cardinal characterizations will be based on  three schemes that are introduced below. 
 In order to phrase these schemes in a compact way, we introduce some terminology.  
As usual, we say that some statement $\varphi(\beta)$ holds for \emph{sufficiently large ordinals $\beta$} if there is an $\alpha\in\Ord$ such that $\varphi(\beta)$ holds for all $\alpha\leq\beta\in\Ord$.
Given infinite cardinals $\lambda\leq\kappa$, a $\ZFC^-$-model $M$ is a \emph{$(\lambda,\kappa)$-model} if it has cardinality $\lambda$ and $(\lambda+1)\cup\{\kappa\}\subseteq M$. A $(\kappa,\kappa)$-model is called a \emph{weak $\kappa$-model}.
A \emph{$\kappa$-model} is a weak $\kappa$-model that is closed under ${<}\kappa$-sequences.\footnote{Note that, unlike some authors, we do not require (weak) $\kappa$-models to be transitive.} 
Moreover, given an infinite cardinal $\theta$ and a class $\calS$ of elementary submodels of $\HH{\theta}$, we say that some statement $\psi(M)$ holds for \emph{many models in $\calS$} if for every $x\in\HH{\theta}$, there exists an $M$ in $\calS$ with $x\in M$ for which $\psi(M)$ holds. 
 Finally, we say that a statement $\psi(M)$ holds for \emph{many transitive weak $\kappa$-models} $M$ if for every $x\subseteq\kappa$, there exists a transitive  weak $\kappa$-model $M$ with $x\in M$ for which $\psi(M)$ holds.

\begin{scheme}\label{schemeSmall}
 An uncountable cardinal $\kappa$ has the large cardinal property $\Phi(\kappa)$ if and only if for all sufficiently large regular cardinals $\theta$ and all infinite cardinals $\lambda<\kappa$,  there are many $(\lambda,\kappa)$-models $M\prec\HH{\theta}$ for which there exists a uniform $M$-ultrafilter $U$ on $\kappa$ with $\Psi(M,U)$. 
\end{scheme}

\begin{scheme}\label{schemeNoel}
  An uncountable cardinal $\kappa$ has the large cardinal property $\Phi(\kappa)$ if and only if for many transitive weak $\kappa$-models $M$ there exists a uniform $M$-ultrafilter $U$ on $\kappa$ with $\Psi(M,U)$. 
\end{scheme}

\begin{scheme}\label{schemeKappa}
 An uncountable cardinal $\kappa$ has the large cardinal property $\Phi(\kappa)$ if and only if for all sufficiently large regular cardinals $\theta$, there are many weak $\kappa$-models $M\prec\HH{\theta}$ for which there exists a uniform $M$-ultrafilter $U$ on $\kappa$ with $\Psi(M,U)$. 
\end{scheme}

Trivial examples of valid instances of the Schemes \ref{schemeSmall} and \ref{schemeKappa} can be obtained by considering the properties $\Phi(\kappa)\equiv\Phi_{ms}(\kappa)\equiv\anf{\textit{$\kappa$ is measurable}}$ and $\Psi(M,U)\equiv\Psi_{ms}(M,U)\equiv$ \lanf \emph{$U$ is $M$-normal and $U=F\cap M$ for some $F\in M$}\ranf. 
 In contrast, Scheme \ref{schemeNoel} cannot provably hold true for $\Phi(\kappa)\equiv\Phi_{ms}(\kappa)$ and a property $\Psi(M,U)$ of models $M$ and $M$-ultrafilters $U$ whose restriction to $\kappa$-models and filters on $\kappa$ is definable by a $\Pi^2_1$-formula over $\VV_\kappa$, because the statement that for many transitive weak $\kappa$-models $M$ there exists a uniform $M$-ultrafilter $U$ on $\kappa$ with $\Psi(M,U)$ could then again be formulated by a $\Pi^2_1$-sentence over $\VV_\kappa$, and measurable cardinals are $\Pi^2_1$-indescribable (see \cite[Proposition 6.5]{MR1994835}). 
Since the measurability of $\kappa$ can be expressed by a $\Sigma^2_1$-formula over 
$\VV_\kappa$, this shows that there is no reasonable\footnote{Let us remark that any reasonable characterization of a large cardinal notion through one of the above schemes should make use of a formula $\Psi$ that is of lower complexity than the formula $\Phi$, for otherwise one can obtain trivial characterizations through any of the three schemes by setting $\Psi(M,U)\equiv\Phi(\bigcup\bigcup U)$.} characterization of measurability through Scheme \ref{schemeNoel}. 
 %
 %Nevertheless, the large cardinal property induced by the property $\Psi_{ms}(M,U)$ and the statement of Scheme \ref{schemeNoel} will turn out to be relevant for the results of this paper and will be discussed below. 
 %
 In order to have a trivial example for a valid instance of Scheme \ref{schemeNoel} available, we make the following definition: 
  %
  % Motivated by the impossibility to provide reasonable characterizations of measurability through Scheme \ref{schemeNoel}, we make the following definition: 

\begin{definition}
  An uncountable cardinal $\kappa$ is \emph{locally measurable} if and only if for many transitive weak $\kappa$-models $M$ there exists a uniform $M$-normal $M$-ultrafilter $U$ on $\kappa$ with $U\in M$.
\end{definition}

 By the transitivity of the models $M$ involved, Scheme \ref{schemeNoel} then holds true for the properties $\Phi(\kappa)\equiv\Phi_{lms}(\kappa)\equiv\anf{\textit{$\kappa$ is locally measurable}}$ and $\Psi(M,U)\equiv\Psi_{ms}(M,U)$. 
 Standard arguments show that measurable cardinals are locally measurable limits of locally measurable cardinals. 
 In addition, we will show that consistency-wise,  locally measurable cardinals are strictly above all other large cardinal notions discussed in this paper. 
 We will also show that locally measurable cardinals are Ramsey. 
 In contrast, they are not necessarily ineffable, for ineffable cardinals are known to be $\Pi^1_2$-indescribable, while local measurability is a $\Pi^1_2$-property of $\kappa$.

%%%%%%%%%%%%%%%%%

\subsection{Large cardinal characterizations}

 In combination with existing results, the work presented in this paper will yield a complete list of large cardinal properties that can be characterized through the above schemes by considering filters possessing various degrees of amenability and normality. 
In order to present these results in a compact way, we introduce abbreviations for the relevant properties of cardinals, models and filters. 
 All relevant definitions will be provided in the later sections of our paper.

\begin{itemize} 
  \item $\Phi_{ia}(\kappa)\equiv\anf{\textit{$\kappa$ is inaccessible}}$, $\Psi_{ia}(M,U)\equiv\anf{\textit{$U$ is ${<}\kappa$-amenable and ${<}\kappa$-complete for $M$}}$.  

  \item $\Phi_{wc}(\kappa)\equiv\anf{\textit{$\kappa$ is weakly compact}}$, $\Psi_{wc}(M,U)\equiv\anf{\textit{$U$ is $\kappa$-amenable and ${<}\kappa$-complete for $M$}}$.  

  \item $\Psi_{\delta}(M,U)\equiv\anf{\textit{$U$ is $M$-normal}}$. 
  
  \item $\Psi_{WC}(M,U)\equiv\anf{\textit{$U$ is ${<}\kappa$-amenable for $M$ and $M$-normal}}.$  

  \item $\Phi_{wie}(\kappa)\equiv\anf{\textit{$\kappa$ is weakly ineffable}}$, $\Psi_{wie}(M,U)\equiv\anf{\textit{$U$ is genuine}}$. 

  \item $\Phi_{ie}(\kappa)\equiv\anf{\textit{$\kappa$ is ineffable}}$, $\Psi_{ie}(M,U)\equiv\anf{\textit{$U$ is normal}}$.  

  \item $\Phi_{cie}(\kappa)\equiv\anf{\textit{$\kappa$ is completely ineffable}}$, $\Psi_{cie}(M,U)\equiv\anf{\textit{$U$ is $\kappa$-amenable for $M$ and $M$-normal}}$. 

  \item $\Phi_{\textup{w}R}(\kappa)\equiv\anf{\textit{$\kappa$ is weakly Ramsey}}$, \\ $\Psi_{\textup{w}R}(M,U)\equiv$\lanf{\it $U$ is $\kappa$-amenable for $M$, $M$-normal, and $\Ult{M}{U}$ is well-founded}\ranf. 
 \item $\Phi_{\omega R}(\kappa)\equiv\anf{\textit{$\kappa$ is $\omega$-Ramsey}}$. 

  \item $\Phi_R(\kappa)\equiv\anf{\textit{$\kappa$ is Ramsey}}$, \\ $\Psi_R(M,U)\equiv\anf{\textit{$U$ is $\kappa$-amenable for $M$, $M$-normal and countably complete}}$. 

  \item $\Phi_{iR}(\kappa)\equiv\anf{\textit{$\kappa$ is ineffably Ramsey}}$, \newline 
$\Psi_{iR}(M,U)\equiv$\lanf\emph{$U$ is $\kappa$-amenable for $M$, $M$-normal and stationary-complete}\ranf.  

  \item $\Phi_{nR}(\kappa)\equiv\anf{\textit{$\kappa$ is $\Delta_\omega^\forall$-Ramsey}}$, $\Psi_{nR}(M,U)\equiv\anf{\textit{$U$ is $\kappa$-amenable for $M$ and normal}}$. 

  \item $\Phi_{stR}(\kappa)\equiv\anf{\textit{$\kappa$ is strongly Ramsey}}$, \newline $\Psi_{stR}(M,U)\equiv\anf{\textit{$M$ is a $\kappa$-model, $U$ is $\kappa$-amenable for $M$ and $M$-normal}}$.  

   \item $\Phi_{suR}(\kappa)\equiv\anf{\textit{$\kappa$ is super Ramsey}}$, \newline $\Psi_{suR}(M,U)\equiv\anf{\textit{$M\prec\HH{\kappa^+}$ is a $\kappa$-model, $U$ is $\kappa$-amenable for $M$ and $M$-normal}}$. 
\end{itemize}

First, note that some of the large cardinal properties appearing in the above list are already defined through one of the above schemes, yielding the following trivial correspondences:

\begin{itemize} 
  \item Scheme \ref{schemeNoel} holds true in the following cases: 
   \begin{itemize} 
    \item $\Phi(\kappa)\equiv\Phi_{stR}(\kappa)$ and $\Psi(M,U)\equiv\Psi_{stR}(M,U)$.  

    \item $\Phi(\kappa)\equiv\Phi_{suR}(\kappa)$ and $\Psi(M,U)\equiv\Psi_{suR}(M,U)$.  
   \end{itemize}
   
 \item Scheme \ref{schemeKappa} holds true in the following cases: 
   \begin{itemize} 
    \item $\Phi(\kappa)\equiv\Phi_{\omega R}(\kappa)$ and $\Psi(M,U)\equiv\Psi_{\textup{w}R}(M,U)$. 
    
    \item $\Phi(\kappa)\equiv\Phi_{nR}(\kappa)$ and $\Psi(M,U)\equiv\Psi_{nR}(M,U)$. 
  \end{itemize}
\end{itemize}

The following theorem summarizes our results, together with a number of known results. Items (\ref{item:wca}), (\ref{item:schemeSmall:CI}) and (\ref{item:ina}) extend the classical results of Kunen, Kleinberg, and Abramson--Harrington--Kleinberg--Zwicker from \cite{MR0460120} mentioned above.\footnote{These could of course have been stated in terms of the existence of \emph{many} countable models $M\prec H(\theta)$ in \cite{MR0460120}. Note that the arguments of \cite{MR0460120} are strongly based on the existence of generic filters over countable models of set theory, hence we will need to follow a very different line of argument.} Item (\ref{item:ramsey}) is the result from Gitman, Sharpe and Welch mentioned above. Items 
%(\ref{item:WC}), 
(\ref{item:wie}) and (\ref{item:ie}) are due to Abramson, Harrington, Kleinberg and Zwicker in \cite[Theorem 1.1.3, Theorem 1.2.1 and Corollary 1.3.1]{MR0460120}. 
% Item (\ref{item:wcbim}) is essentially due to Baumgartner in \cite[Theorem 2.1]{MR0540770}.

\begin{theorem}\label{theorem:SchemesSummary}
 \begin{enumerate} 

  \item\label{item:wca} Schemes \ref{schemeSmall}, \ref{schemeNoel} and \ref{schemeKappa} hold true in case $\Phi(\kappa)\equiv\Phi_{wc}(\kappa)$ and $\Psi(M,U)\equiv\Psi_{wc}(M,U)$.  

  \item Schemes \ref{schemeSmall} and \ref{schemeKappa} both hold true in the following cases: 
    \begin{enumerate}
        \item\label{item:schemeSmall:CI} $\Phi(\kappa)\equiv\Phi_{cie}(\kappa)$ and $\Psi(M,U)\equiv\Psi_{cie}(M,U)$.  

        \item\label{item:schemeSmall:wR} $\Phi(\kappa)\equiv\Phi_{\omega R}(\kappa)$ and $\Psi(M,U)\equiv\Psi_{\textup{w}R}(M,U)$.  

       \item\label{item:schemeSmall:nR} $\Phi(\kappa)\equiv\Phi_{nR}(\kappa)$ and either
       \begin{enumerate}
        \item $\Psi(M,U)\equiv\Psi_{iR}(M,U)$,  or
        
        \item\label{item:schemeSmall:nRobvious} $\Psi(M,U)\equiv\Psi_{nR}(M,U)$.  
       \end{enumerate}

   \end{enumerate}

  \item Scheme \ref{schemeNoel} holds true in the following cases:
  \begin{enumerate}
    \item\label{item:ramsey} $\Phi(\kappa)\equiv\Phi_R(\kappa)$ and $\Psi(M,U)\equiv\Psi_R(M,U)$. 
    
    \item\label{item:CharIneffRamsey} $\Phi(\kappa)\equiv\Phi_{iR}(\kappa)$ and $\Psi(M,U)\equiv\Psi_{iR}(M,U)$.
  \end{enumerate}

  \item\label{item:a} Scheme \ref{schemeSmall} holds true in the following cases:  

      \begin{enumerate}
    \item $\Phi(\kappa)\equiv\anf{\textit{$\kappa$ is regular}}$ and either 

      \begin{enumerate}
        \item\label{item:rega} $\Psi(M,U)\equiv\anf{\textit{$U$ is ${<}\kappa$-complete for $M$}}$,   or 
        \item\label{item:regb} $\Psi(M,U)\equiv\Psi_{ie}(M,U)$. 
      \end{enumerate} 

    \item\label{item:schemeSmall:Inaccessible} $\Phi(\kappa)\equiv\Phi_{ia}(\kappa)$ and either 

      \begin{enumerate}
        \item\label{item:ina} $\Psi(M,U)\equiv\Psi_{ia}(M,U)$, or
        \item\label{item:inb} $\Psi(M,U)\equiv\anf{\textit{$U$ is ${<}\kappa$-amenable for $M$ and normal}}$.  
      \end{enumerate}
   \end{enumerate}

  \item\label{item:b} Schemes \ref{schemeNoel} and \ref{schemeKappa} hold true in the following cases: 

      \begin{enumerate}
        \item\label{item:wcb} $\Phi(\kappa)\equiv\Phi_{wc}(\kappa)$ and either 

        \begin{enumerate}
         \item\label{item:WCia} $\Psi(M,U)\equiv\Psi_{ia}(M,U)$, 
         
          \item\label{item:WCii} $\Psi(M,U)\equiv\anf{\textit{$U$ is stationary-complete, $M$-normal and ${<}\kappa$-amenable for $M$}}$.  
        \end{enumerate}

        \item\label{item:wie} $\Phi(\kappa)\equiv\Phi_{wie}(\kappa)$ and $\Psi(M,U)\equiv\Psi_{wie}(M,U)$. 

        \item\label{item:ie} $\Phi(\kappa)\equiv\Phi_{ie}(\kappa)$ and $\Psi(M,U)\equiv\Psi_{ie}(M,U)$.  
   \end{enumerate}
 \end{enumerate}
\end{theorem}

The above results are summarized in abbreviated form in Tables \ref{table:schemeSmall}, \ref{table:schemeIntermediate} and \ref{table:schemeKappa} below.\footnote{Let us remark that if we were to consider models of size less than $\kappa$ in Table \ref{table:schemeKappa} without elementarity, we would clearly only get regularity at all levels of normality
%, as in Table \ref{table:schemeSmall}
.}
  The meaning of the tables should be clear to the reader when compared with the results presented in Theorem \ref{theorem:SchemesSummary}. 
  Note that no normality properties weaker than genuineness appear in Table \ref{table:schemeSmall}, because it turns out that these properties alone cannot be used to characterize large cardinal properties (see Theorem \ref{theorem:mahlolike} and the discussion following Theorem \ref{theorem:weaklycompactideal}). 
All entries in Table \ref{table:schemeKappa} that are not mentioned within the statement of Theorem \ref{theorem:SchemesSummary} will be immediate consequences of the definitions of the large cardinal notions that will be introduced later in our paper. 
%Table \ref{table:schemeKappa} contains four entries ($\mathbf T_\omega^\kappa$-Ramsey, $\infty_\omega^\kappa$-Ramsey, $\Delta_\omega^\kappa$-Ramsey and $\prec$-Ramsey) that have not been mentioned so far -- they constitute large cardinal notions that will be introduced in the later sections of our paper. 
 %
 %The fact that they fit into our below tables will be immediate from their definitions, which is why they are not mentioned in Theorem \ref{theorem:SchemesSummary}.
Furthermore, let us remark that our results (some of which are mentioned already within Theorem \ref{theorem:SchemesSummary}) will show that the size of the models considered is irrelevant once we consider elementary submodels of (sufficiently large) $\HH{\theta}$'s in Table \ref{table:schemeKappa}.

%%%%%%%%%%%%%%%%%%%%%%%%%%%

\begin{table}[h]

\medskip

\begin{tabular}{| l | c | c |}
 \hline 
 & $|M|<\kappa$ & $|M|=\kappa$\\
%\hline
%${<}\kappa$-complete for $M$& regular & ??, implies tree property\\
%\hline
%$M$-normal & regular & ??, implies limit cardinal\\
%\hline
%$M$-normal and well-founded & regular & embedding property\\
%\hline
%$M$-normal and countably complete & regular & $\sigma$-complete filter property\\
%\hline
%$M$-normal and stationary-complete & regular & $\sigma$-complete filter property\\
\hline
genuine & regular & weakly ineffable\\
\hline
normal & regular & ineffable\\
\hline
\end{tabular}

  \medskip

 \caption{Characterizations without amenability}\label{table:schemeSmall}
\end{table}

%%%%%%%%%%%%%%%%%%%%%%%%%%%

\begin{table}[h]

\medskip

\begin{tabular}{| l | c | c | c |}
 \hline 
 ${<}\kappa$-amenable and \ldots & $|M|<\kappa$ & $|M|=\kappa$\\
\hline
${<}\kappa$-complete for $M$ & inaccessible & weakly compact\\
\hline
$M$-normal & inaccessible & weakly compact\\
\hline
$M$-normal and well-founded & inaccessible & weakly compact\\
\hline
$M$-normal and countably complete & inaccessible & weakly compact\\
\hline
$M$-normal and stationary-complete & inaccessible & weakly compact\\
\hline
genuine & inaccessible & weakly ineffable\\
\hline
normal & inaccessible & ineffable\\
\hline
\end{tabular}

  \medskip

 \caption{Characterizations with ${<}\kappa$-amenability}\label{table:schemeIntermediate}
\end{table}

%%%%%%%%%%%%%%%%%%%%%%%%%%%

\begin{table}[h]

\medskip

\begin{tabular}{| l | c | c |}
 \hline
$\kappa$-amenable and \ldots & $|M|=\kappa$ & $M\prec\HH{\theta}$\\
\hline
${<}\kappa$-complete for $M$ & weakly compact & weakly compact\\
\hline
$M$-normal & $\mathbf T_\omega^\kappa$-Ramsey & completely ineffable\\
\hline
$M$-normal and well-founded & weakly Ramsey & $\omega$-Ramsey\\
\hline
$M$-normal and countably complete & Ramsey & $\cc_\omega^\forall$-Ramsey\\
\hline
$M$-normal and stationary-complete & ineffably Ramsey & $\Delta_\omega^\forall$-Ramsey\\
\hline
genuine & $\infty_\omega^\kappa$-Ramsey & $\Delta_\omega^\forall$-Ramsey\\
\hline
normal & $\Delta_\omega^\kappa$-Ramsey & $\Delta_\omega^\forall$-Ramsey\\
\hline
\end{tabular}

  \medskip

 \caption{Characterizations with $\kappa$-amenability}\label{table:schemeKappa}
\end{table}

%%%%%%%%%%%%%%%%%%%%%%%%%%%

\subsection{Large cardinal ideals}

%Motivated by the above characterizations, we consider the question which subsets of $\kappa$ can be elements of ultrafilters $U$ witnessing the validity of the above schemes. 
 %
%In particular, we study the ideals on $\kappa$ consisting of all subsets of $\kappa$ that cannot be contained in these filters. 
%
%As outlined above, we now want to study the ideals on large cardinals consisting of all subsets that cannot be contained in filters witnessing the validity of the above schemes. 

We next want to study the large cardinal ideals that are canonically induced by our characterizations.

\begin{definition}\label{definition:Ideals}
 Let $\Psi(M,U)$ be a property of models $M$ and filters $U$, and let $\kappa$ be an uncountable cardinal.  
  \begin{enumerate} 
   \item We define $\II^{{<}\kappa}_{\Psi}$ to be the collection of all $A\subseteq\kappa$ with the property that for all sufficiently large regular cardinals $\theta$, there exists a set $x\in\HH{\theta}$ such that for all infinite cardinals $\lambda<\kappa$, if $M\prec\HH{\theta}$ is a $(\lambda,\kappa)$-model with $x\in M$ and $U$ is a uniform $M$-ultrafilter on $\kappa$ with $\Psi(M,U)$, then $A\notin U$. 
  
   \item\label{definition:Ideals-KappaSizedNonelem} We define $\II^\kappa_\Psi$ to be the collection of all $A\subseteq\kappa$ with the property that there exists $x\subseteq\kappa$  such that if $M$ is a transitive weak $\kappa$-model with $x\in M$ and $U$ is a uniform $M$-ultrafilter on $\kappa$ with $\Psi(M,U)$, then $A\notin U$.  
   
   \item\label{definition:Ideals-KappaSizedElem} We define $\II^\kappa_{\prec\Psi}$ to be the collection of all $A\subseteq\kappa$ with the property that for all sufficiently large regular cardinals $\theta$, there exists a set $x\in\HH{\theta}$  such that if $M\prec\HH{\theta}$ is a weak $\kappa$-model with $x\in M$ and $U$ is a uniform $M$-ultrafilter on $\kappa$ with $\Psi(M,U)$, then $A\notin U$.   
  \end{enumerate} 
\end{definition}

It is easy to see that the collections $\II_\Psi^{{<}\kappa}$, $\II_\Psi^\kappa$ and $\II_{\prec\Psi}^\kappa$ always form ideals on $\kappa$. 
 %, 
 Moreover, if Scheme \ref{schemeSmall}, \ref{schemeNoel} or \ref{schemeKappa} holds for some large cardinal property $\Phi(\kappa)$ and some property $\Psi(M,U)$ of models $M$ and filters $U$, then the statement that $\Phi(\kappa)$ holds for some uncountable cardinal $\kappa$ implies the properness of the ideal $\II_\Psi^{{<}\kappa}$, $\II_\Psi^\kappa$, or $\II_{\prec\Psi}^\kappa$ respectively.
 In addition, in all cases covered by Theorem \ref{theorem:SchemesSummary} (and also in most other natural situations), the converse of this implication also holds true. 
This is trivial for instances of Scheme \ref{schemeNoel}. For instances of Schemes \ref{schemeSmall} and \ref{schemeKappa}, this is an easy consequence of the observation that all properties $\Psi$ listed in the theorem are \emph{restrictable}, 
 i.e.\ given uncountable cardinals $\bar{\theta}<\theta$, if $M\prec\HH{\theta}$ with $\bar{\theta}\in M$, $\kappa\in M\cap\bar{\theta}$ is a cardinal and $U$ is a uniform $M$-ultrafilter on $\kappa$ with $\Psi(M,U)$, then $\Psi(\bar{M},U)$ holds, where $\bar M=M\cap\HH{\bar{\theta}}\prec\HH{\bar{\theta}}$. Moreover, in Lemma \ref{lemma:AllIdealsNormal} below, we will see that in most cases these ideals are in fact normal ideals.

The above definitions provide uniform ways to assign ideals to large cardinal properties. 
The next theorem shows that, in the cases where such canonical ideals were already defined, these assignments turn out to coincide with the known notions.  
In the following, given an abbreviation $\Psi_{\ldots}$ for a property of models and filters from the above list, we will write $\II^{{<}\kappa}_{\ldots}$ instead  of $\II^{{<}\kappa}_{\Psi_{\ldots}}$, $\II^\kappa_{\prec\ldots}$ instead of $\II^\kappa_{\prec\Psi_{\ldots}}$, and $\II^\kappa_{\ldots}$ instead of $\II^\kappa_{\Psi_{\ldots}}$. 
 Item (\ref{item:Ideals:WC}) below is essentially due to Baumgartner (see {\cite[Section 2]{MR0540770}}).
The weakly ineffable ideal and the ineffable ideal were introduced by Baumgartner in \cite{MR0384553}.
The completely ineffable ideal was introduced by Johnson in \cite{MR918427}.
Item (\ref{item:Ideals:CI}) below is a generalization of a result for countable models by Kleinberg mentioned in \cite{MR513844} after the proof of its Theorem 2. The Ramsey ideal and the ineffably Ramsey ideal were introduced by Baumgartner in \cite{MR0540770}.
Items (\ref{item:Ideals:r}) and (\ref{item:Ideals:ir}) are essentially due to Mitchell (see \cite[Theorem 2.10]{brent}).

\begin{theorem}\label{theorem:Ideals}
 \begin{enumerate} 
  \item\label{item:Ideals:ia} If $\kappa$ is inaccessible, then $\II^{{<}\kappa}_{ia}$ is the bounded ideal on $\kappa$.  

  \item\label{item:Ideals:reg} If $\kappa$ is a regular and uncountable cardinal, then $\II_{\delta}^{<\kappa}$ is the non-stationary ideal on $\kappa$.  

  \item\label{item:Ideals:WC}  
  If $\kappa$ is a weakly compact cardinal, then $\II_{WC}^\kappa$ is the weakly compact ideal on $\kappa$.  

  \item\label{item:Ideals:wie} If $\kappa$ is a weakly ineffable cardinal, then $\II_{wie}^\kappa$ is the weakly ineffable ideal on $\kappa$.  
  
  \item\label{item:Ideals:ie} If $\kappa$ is an ineffable cardinal, then $\II_{ie}^\kappa$ is the ineffable ideal on $\kappa$. 

  \item\label{item:Ideals:CI} If $\kappa$ is a completely ineffable cardinal, then $\II^\kappa_{{\prec}{cie}}$ is the completely ineffable ideal on $\kappa$. 

  \item\label{item:Ideals:r} If $\kappa$ is a Ramsey cardinal, then $\II_R^\kappa$ is the Ramsey ideal on $\kappa$.  

  \item\label{item:Ideals:ir} If $\kappa$ is an ineffably Ramsey cardinal, then $\II_{iR}^\kappa$ is the ineffably Ramsey ideal on $\kappa$.  

  \item\label{item:Ideals:measurable} If $\kappa$ is a measurable cardinal, then the complement of $\II^\kappa_{\prec{ms}}$ is the union of all normal ultrafilters on $\kappa$. 
 \end{enumerate}
\end{theorem}

%%%%%%%%%%%%%%%%%%%%

\subsection{Structural properties of large cardinal ideals}

We show that many aspects of the relationship between different large cardinal notions are reflected in the relationship of their corresponding ideals.

 First, our results will show that, for many important examples, the ordering of large cardinal properties under direct implication turns out to be identical to the ordering of their corresponding ideals under inclusion. %{i.e.}\ the fact that some large cardinal property provably implies another large cardinal property turns out to be equivalent to the fact that it can be proven that, for every cardinal satisfying the first property, the ideal given by this property contains the ideal given by the second property. 
 %
% Note that the remarks following Definition \ref{definition:Ideals} show that the backward implication of this correspondence is trivially valid for all large cardinal properties studied in this paper. 

Next, our approach to show that the ordering of large cardinal properties by their consistency strength can also be read off from the corresponding ideals 
 is motivated by the fact that the well-foundedness of the \emph{Mitchell order} (see {\cite[Lemma 19.32]{MR1940513}}) implies that for every measurable cardinal $\kappa$, there is a normal ultrafilter $U$ on $\kappa$ with the property that $\kappa$ is not measurable in $\Ult{\VV}{U}$. 
 Translated into the context of our paper (using Theorem \ref{theorem:Ideals}.(\ref{item:Ideals:measurable})) this shows that the set of all non-measurables below $\kappa$ is not contained in the ideal $\II^\kappa_{{\prec}ms}$.\footnote{We will also provide an easy argument for this result that does not make any use of the Mitchell order in Lemma \ref{lemma:PropertiesOfMeasurableIdeal}.}
 To generalize this to other large cardinal properties $\Phi$, if $\kappa$ is a cardinal, we let $$\NN^\kappa_\Phi ~ = ~ \Set{\alpha<\kappa}{\neg\Phi(\alpha)}.$$ 
If $\Phi_{\ldots}$ is an abbreviation for a large cardinal property, then we write $\NN^\kappa_{\ldots}$ instead of $\NN^\kappa_{\Phi_{\ldots}}$. 
We show that the above result for measurable cardinals can be generalized to many other important large cardinal notions.\footnote{For Ramsey, strongly Ramsey, and super Ramsey cardinals, this also follows from the results of \cite{mitchellforramsey}, where the notion of Mitchell rank is generalized to apply to these large cardinal notions.}
  More precisely, for various instances of our characterization schemes, we will show that the above set of ordinals without the given large cardinal property is not contained in the corresponding ideal. 
 These results can be seen as indicators that the derived characterization and the associated ideal canonically describe the given large cardinal property, as one would expect these cardinals to lose some of their large cardinal properties in their ultrapowers. 
 Moreover, our results also show that, in many important cases, the fact that some large cardinal property $\Phi^*$ has a strictly higher consistency strength than some other large cardinal property $\Phi$ is equivalent to the fact that $\Phi^*(\kappa)$ implies that the set $\NN^\kappa_\Phi$ is an element of the ideal on $\kappa$ corresponding to $\Phi^*$.
 This allows us to reconstruct the consistency strength ordering of these properties from structural properties of their corresponding ideals.
  Together with the correspondence described in the last paragraph, it also shows that, in many cases, the fact that some large cardinal property $\Phi^*$ provably implies a large cardinal property $\Phi$ of strictly lower consistency strength yields that $\Phi(\kappa)$ implies the ideal on $\kappa$ corresponding to $\Phi$ to be a proper subset of the ideal on $\kappa$ corresponding to $\Phi^*$. 
 %
   % We study three different structural aspects of the above ideals listed below. 
 %
 %For many classical large cardinal properties, a combination of existence results with the results this paper will completely answer the arising questions. 
 %
%In contrast, many questions about the structural properties of strong Ramsey-like cardinals remain open (see Section \ref{section:questions}) and the possibility that the answer to these questions might differ for these cardinals could hint at fundamental differences between these notions and the properties above and below them in the large cardinal hierarchy. 
 %
%  The second focus of our structural investigations are the relations between the different ideals obtained in the above way. 
  %
 %  Together with a number of existing results, we can  show that, in many important cases, the above ideals form strictly increasing sequences corresponding to direct implications between large cardinal properties. 
   %
 %   Moreover, we can use sets of the form $\NN^\kappa_\Phi$ to show that these sequences are strictly increasing. 
   %
%   As above, many of these questions remain open for several Ramsey-like properties. 

 Finally, we consider the question whether there are fundamental differences between the ideal $\II^\kappa_{{\prec}ms}$ induced by measurability and the ideals induced by weaker large cardinal notions. 
  %  The final focus of our investigations is again motivated by the behavior of the measurable ideal $\II^\kappa_{{\prec}ms}$. 
  %
  By classical results of Kunen (see {\cite[Theorem 20.10]{MR1994835}}), it is possible that there is a unique normal measure on a measurable cardinal $\kappa$. In this case, the ideal $\II^\kappa_{{\prec}ms}$ is equal to the complement of this measure and hence the induced partial order $\POT{\kappa}/\II^\kappa_{{\prec}ms}$ is trivial, hence in particular atomic. We study the question whether the partial orders induced by other large cardinal ideals can also be atomic, conjecturing that the possible atomicity of the quotient partial order is a property that separates measurability from all weaker large cardinal properties (this is motivated by Lemma \ref{lemma:atoms} below). This conjecture turns out to be closely related to the previous topics, and we will verify it for many prominent large cardinal properties.
 %, whereas it remains open for stronger Ramsey-like cardinals. 

 The following theorem provides selected instances of our results, namely some of their consequences for large cardinal notions that had already been introduced in the set theoretic literature. 
 Item (\ref{item:cont:ia}) and the statement that $\II_{ia}^{<\kappa}\subseteq\II_{WC}^\kappa$ in Item (\ref{item:cont:WC}) below are of course trivial consequences of Theorem \ref{theorem:Ideals}.  
 The statement that $\NN^\kappa_{ia}$ belongs to the weakly compact ideal in Item (\ref{item:cont:WC}) has been shown by Baumgartner in \cite[Theorem 2.8]{MR0540770}. 
 The statement that $\II^\kappa_{WC}\subseteq\II^\kappa_{wie}$ in Item (\ref{item:cont:wie}) has been shown by Baumgartner in \cite[Theorem 7.2]{MR0384553}. 
  That $\NN_{ie}^\kappa\in\II_{{\prec}cie}^\kappa$ in Item (\ref{item:cont:ci}) was shown by Johnson in \cite[Corollary 4]{MR853844}, however we will also provide an easy self-contained argument of this result later on for the benefit of our readers.
 Gitman has shown that weakly Ramsey cardinals (which are also known under the name of \emph{$1$-iterable cardinals}) are weakly ineffable limits of completely ineffable cardinals (see {\cite[Theorem 3.3 and Theorem 3.7]{MR2830415}}). Her arguments in the proof of \cite[Theorem 3.7]{MR2830415} actually show that if $\kappa$ is a weakly Ramsey cardinal, then $\NN_{cie}^\kappa\in\II_{\textup{w}R}^\kappa$, as in Item (\ref{item:cont:wr}). 
 That $\II_{\textup{w}R}^\kappa\subseteq\II_R^\kappa$ in Item (\ref{item:cont:r}) is already immediate from our above definitions. The proof of {\cite[Theorem 4.1]{MR2830435}} shows that $\NN_{\textup{w}R}^\kappa\in\II_R^\kappa$, as in Item (\ref{item:cont:r}). 
That $\NN_R^\kappa\notin I_R^\kappa$ in Item (\ref{item:cont:r}), and $\II^\kappa_R\cup\{\NN^\kappa_R\}\subseteq\II^\kappa_{iR}$ and $\NN_{iR}^\kappa\notin I_{iR}^\kappa$ in Item (\ref{item:cont:ir}) are due to Feng (see \cite[Corollary 4.4 and Theorem 4.5]{MR1077260}). 
 Moreover, Theorem \ref{theorem:Ideals} directly shows that $\II^\kappa_{ie}\subseteq\II^\kappa_{iR}$ holds for ineffably Ramsey cardinals $\kappa$. 
That $\NN_{stR}^\kappa\notin I_{stR}^\kappa$ in Item (\ref{item:IdealContain:stR}) and that $\NN_{suR}^\kappa\notin I_{suR}^\kappa$ in Item (\ref{item:IdealContain:suR}) follows easily from the results of \cite{mitchellforramsey}, and these statements will also be immediate consequences of fairly general results from our paper.
That $\II^\kappa_{{\prec}cie}\nsubseteq\II^\kappa_{suR}$ in Item (\ref{item:IdealContain:suR}) was brought to our attention by Gitman, after we had posed this as an open question in an early version of this paper.
The final statement of Item (\ref{item:IdealContain:MS}) is an immediate consequence of the above-mentioned result of Kunen.

\begin{theorem}\label{theorem:IdealContain}
  \begin{enumerate} 
    \item\label{item:cont:ia} If $\kappa$ is an inaccessible cardinal, then $\NN^\kappa_{ia}\notin\II^{{<}\kappa}_{ia}$, and $\POT{\kappa}/\II^{{<}\kappa}_{ia}$ is not atomic.  

  \item\label{item:cont:WC}  
  If $\kappa$ is a weakly compact cardinal, then $\II^{{<}\kappa}_{ia}\cup\{\NN^\kappa_{ia}\}\subseteq\II^\kappa_{WC}$, $\NN^\kappa_{wc}\notin\II^\kappa_{WC}$, and $\POT{\kappa}/\II^\kappa_{WC}$ is not atomic.  

  \item\label{item:cont:wie}  
   If $\kappa$ is a weakly ineffable cardinal, then $\II^\kappa_{WC}\cup\{\NN^\kappa_{wc}\}\subseteq\II^\kappa_{wie}$, $\NN^\kappa_{wie}\notin\II^\kappa_{wie}$, and $\POT{\kappa}/\II^\kappa_{wie}$ is not atomic.  

  \item\label{item:cont:ie}  
   If $\kappa$ is an ineffable cardinal, then $\II^\kappa_{wie}\cup\{\NN^\kappa_{wie}\}\subseteq\II^\kappa_{ie}$, $\NN^\kappa_{ie}\notin\II^\kappa_{ie}$, and $\POT{\kappa}/\II^\kappa_{ie}$ is not atomic.  

  \item\label{item:cont:ci} If $\kappa$ is a completely ineffable cardinal, then $\II^\kappa_{ie}\cup\{\NN^\kappa_{ie}\}\subseteq\II^\kappa_{{\prec}cie}$, $\NN^\kappa_{cie}\notin\II^\kappa_{{\prec}{cie}}$, and $\POT{\kappa}/\II^\kappa_{{\prec}cie}$ is not atomic.  

  \item\label{item:cont:wr}  
   If $\kappa$ is a weakly Ramsey cardinal, then $\II_{wie}^\kappa\cup\{\NN^\kappa_{cie}\}\subseteq\II^\kappa_{\textup{w}R}$, $\NN^\kappa_{\textup{w}R}\notin\II^\kappa_{\textup{w}R}$, $\II^\kappa_{ie}\nsubseteq\II^\kappa_{\textup{w}R}$, and $\POT{\kappa}/\II^\kappa_{\textup{w}R}$ is not atomic. 

  \item\label{item:cont:r}  
    If $\kappa$ is a Ramsey cardinal, then $\II^\kappa_{\textup{w}R}\cup\{\NN^\kappa_{\textup{w}R}\}\subseteq\II^\kappa_R$, $\NN^\kappa_R\notin\II^\kappa_R$, $\II^\kappa_{ie}\nsubseteq\II^\kappa_R$, and $\POT{\kappa}/\II^\kappa_R$ is not atomic.  

  \item\label{item:cont:ir}  
   If $\kappa$ is an ineffably Ramsey cardinal, then $\II^\kappa_{ie}\cup\II^\kappa_R\cup\{\NN^\kappa_R\}\subseteq\II^\kappa_{iR}$, $\NN^\kappa_{iR}\notin\II^\kappa_{iR}$, $\II^\kappa_{{\prec}cie}\nsubseteq\II^\kappa_{iR}$, and $\POT{\kappa}/\II^\kappa_{iR}$ is not atomic.    

  \item\label{item:IdealContain:stR} If $\kappa$ is strongly Ramsey, then $\II^\kappa_R\cup\{\NN^\kappa_{iR}\}\subseteq\II^\kappa_{stR}$, $\NN^\kappa_{stR}\notin\II^\kappa_{stR}$, $\II^\kappa_{ie}\nsubseteq\II^\kappa_{stR}$, and $\POT{\kappa}/\II^\kappa_{stR}$ is not atomic.  
  
  \item\label{item:IdealContain:suR} If $\kappa$ is super Ramsey, then $\II^\kappa_{iR}\cup\II^\kappa_{stR}\cup\{\NN^\kappa_{stR}\}\subseteq\II^\kappa_{suR}$, $\NN^\kappa_{suR}\notin\II^\kappa_{suR}$, $\II^\kappa_{{\prec}cie}\nsubseteq\II^\kappa_{suR}$, and $\POT{\kappa}/\II^\kappa_{suR}$ is not atomic. 

  \item\label{item:IdealContain:LMS} If $\kappa$ is locally measurable, then $\II^\kappa_{stR}\cup\{\NN^\kappa_{suR}\}\subseteq\II^\kappa_{ms}$, $\NN^\kappa_{lms}\notin\II^\kappa_{ms}$, $\II^\kappa_{ie}\nsubseteq\II^\kappa_{ms}$, and $\POT{\kappa}/\II^\kappa_{ms}$ is not atomic.  

  \item\label{item:IdealContain:MS} If $\kappa$ is measurable, then $\II^\kappa_{{\prec}{cie}}\cup\II^\kappa_{suR}\cup\II^\kappa_{ms}\cup\{\NN^\kappa_{lms}\}\subseteq\II^\kappa_{\prec ms}$, $\NN^\kappa_{ms}\notin\II^\kappa_{\prec ms}$, and $\POT{\kappa}/\II^\kappa_{\prec ms}$ may be atomic.  
  \end{enumerate}
\end{theorem}

 Note that the above statements show that the linear ordering of the mentioned large cardinal properties by their consistency strength can be read off from the containedness of sets of the form $\NN^\kappa_\Phi$ in the induced ideals. 
 Moreover, all provable implications and consistent non-implications can be read of from the ordering of the corresponding ideals under inclusion.
 For example, the fact that ineffability and Ramseyness do not provably imply each other corresponds to the fact that $\II^\kappa_{ie}\nsubseteq\II^\kappa_R\nsubseteq\II^\kappa_{ie}$ holds whenever $\kappa$ is both ineffable and Ramsey, where the second non-inclusion is a consequence of $\NN^\kappa_{ie}\subseteq\NN^\kappa_{cie}\in\II^\kappa_{\textup{w}R}\subseteq\II^\kappa_R$ and $\NN^\kappa_{ie}\notin\II^\kappa_{ie}$.

Figure \ref{figure:OrderingIdeals} below summarizes the structural statements listed in Theorem \ref{theorem:IdealContain}. In this diagram, a provable inclusion $I_1\subseteq I_0$ of large cardinal ideals is represented by a solid arrow $\xymatrix{I_0\ar[r]&I_1}$. Moreover, if $I_1$ is an ideal induced by a large cardinal property $\Phi$, then a dashed arrow $\xymatrix{I_0\ar@{-->}[r]&I_1}$ represents the statement that $\NN^\kappa_\Phi\in I_0$ provably holds.

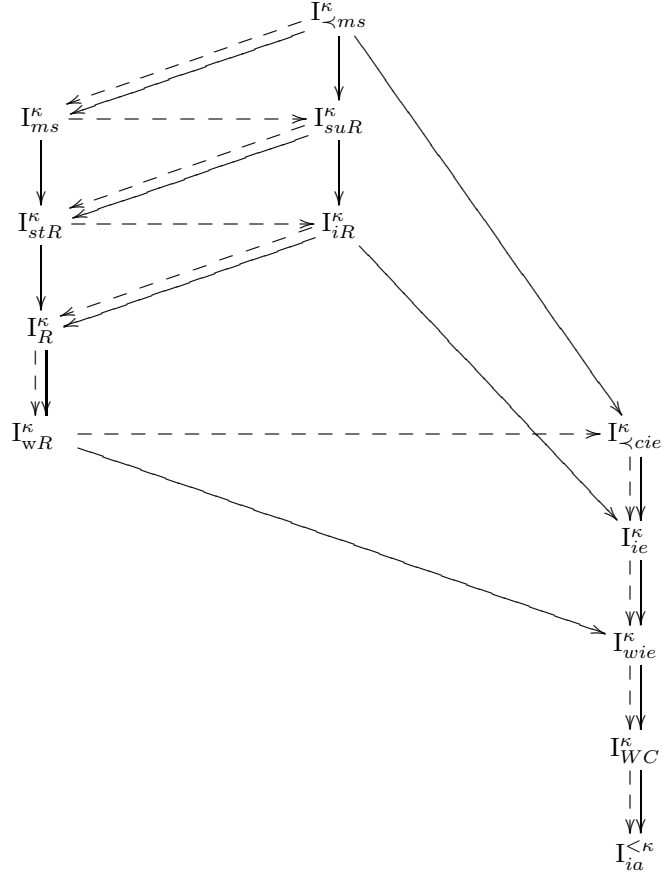
\begin{figure}\label{fig_Implications}
 \begin{displaymath}
  \xymatrix{
     & & & \II^\kappa_{{\prec}ms} \ar[d] \ar@<2pt>[dlll] \ar@{-->}@<-2pt>[dlll]  \ar[ddddrrr] & & & \\
    \II^\kappa_{ms} \ar[d] \ar@{-->}[rrr]  & & & \II^\kappa_{suR} \ar@<2pt>[dlll] \ar@{-->}@<-2pt>[dlll] \ar[d] & & &  \\ 
    \II^\kappa_{stR} \ar[d] \ar@{-->}[rrr] & & & \II^\kappa_{iR} \ar@<2pt>[dlll] \ar@{-->}@<-2pt>[dlll]  \ar[dddrrr] & & & \\ 
   \II^\kappa_R \ar@<2pt>[d] \ar@{-->}@<-2pt>[d] & &  & & & & \\
 \II^\kappa_{\textup{w}R} \ar[ddrrrrrr] \ar@{-->}[rrrrrr]\hole & & & &   & & \II^\kappa_{{\prec}cie} \ar@<2pt>[d] \ar@{-->}@<-2pt>[d] \\
   & & & &  &  & \II^\kappa_{ie} \ar@<2pt>[d] \ar@{-->}@<-2pt>[d] \\
     & &  & &  &  & \II^\kappa_{wie} \ar@<2pt>[d] \ar@{-->}@<-2pt>[d]\\
         & &   & &  &  & \II^\kappa_{WC} \ar@<2pt>[d] \ar@{-->}@<-2pt>[d] \\
                 &  &         & &  &  & \II^{{<}\kappa}_{ia} \\
 }
\end{displaymath}
\caption{Ordering of large cardinal ideals.}
\label{figure:OrderingIdeals}
\end{figure}

%%%%%%%%%%%%%%%%%%%%%%%%%%%%
%%%%%%%%%%%%%%%%%%%%%%%%%%%%

\section{Some basic notions}\label{generalizations}

A key ingredient for our results will be the generalization of a number of standard notions to the context of non-transitive models, and, in the case of elementary embeddings, also to possibly non-wellfounded target models. While most of these definitions are very much standard, we will take some care in order to present them in a way that makes them applicable also in these generalized settings. They clearly correspond to their usual counterparts in the case of transitive models $M$. 
In the following, we let $\ZFC^-$ denote the collection of axioms of $\ZFC$ without the powerset axiom (but, as usual, with the axiom scheme of Collection rather than the axiom scheme of Replacement). 
 In order to avoid unnecessary technicalities, we restrict our attention to \emph{$\Sigma_0$-correct models}, i.e.\ models that are $\Sigma_0$-elementary in $\VV$. 
 Since every $\Sigma_0$-elementary submodel of a transitive class is  $\Sigma_0$-correct, all models considered in the above schemes are $\Sigma_0$-correct. 
 Note that if $M$ is $\Sigma_0$-correct, $\alpha\in M$ is an ordinal in $M$ and $f\in M$ is a function with domain $\alpha$ in $M$, then $\alpha$ is an ordinal and $f$ is a function with domain $\alpha$.

\begin{definition}[Properties of $M$-ultrafilters]\label{ultrafilterproperties}
 Let $M$ be a class that is a $\Sigma_0$-correct model of $\ZFC^-$ and let $\kappa$ be a cardinal of $M$. 
  \begin{itemize} 
  \item A collection $U\subseteq M\cap\POT{\kappa}$ is an \emph{$M$-ultrafilter on $\kappa$} if $\langle M,U\rangle\models\anf{\textit{$U$ is an ultrafilter on $\kappa$}}$.  
\end{itemize}
 In the following, let $U$ denote an $M$-ultrafilter on $\kappa$. 
\begin{itemize} 
  \item $U$ is \emph{non-principal} if $\{\alpha\}\notin U$ for all $\alpha<\kappa$. 

  \item $U$ \emph{contains all final segments of $\kappa$ in $M$} if $[\alpha,\kappa)\in U$ whenever $\alpha\in M\cap\kappa$. 

  \item $U$ is \emph{uniform} if $\betrag{x}^M=\kappa$ for all $x\in U$. 

  \item $U$ is \emph{${<}\kappa$-amenable} (respectively, \emph{$\kappa$-amenable) for $M$} if whenever $\alpha<\kappa$ (respectively, $\alpha=\kappa$)  and $\seq{x_\beta}{\beta<\alpha}$ is a sequence of subsets of $\kappa$ that is an element of $M$, then $$\langle M,U\rangle\models\anf{\exists x ~  \forall\beta<\alpha ~ [x_\beta\in U ~ \longleftrightarrow ~ \beta\in x]}.\footnote{Note that, from the point of view of $V$, given such a sequence $\seq{x_\beta}{\beta<\alpha}$, a witness $x$ for amenability in $M$ need only satisfy that \anf{$x_\beta\in U\longleftrightarrow\beta\in x$} whenever $\beta\in M\cap\alpha$.}$$ 

  \item Given $\alpha\leq\kappa$ in $M$, $U$ is \emph{${<}\alpha$-complete for $M$} if $\langle M,U\rangle\models\anf{\textit{$U$ is ${<}\alpha$-complete}}$. 

 \item $U$ is \emph{$M$-normal}  if  $\langle M,U\rangle\models\anf{\textit{$U$ is normal}}$. 

   \item $U$ is \emph{$M$-normal with respect to $\subseteq$-decreasing sequences} if whenever $\seq{x_\alpha}{\alpha<\kappa}$ is a sequence of subsets of $\kappa$ that is an element of $M$ and satisfies $$\langle M,U\rangle\models\anf{\forall\alpha\leq\beta<\kappa ~ [x_\alpha\in U ~ \wedge ~ x_\beta\subseteq x_\alpha]},$$ then $\langle M,U\rangle\models\anf{\Delta_{\alpha<\kappa}x_\alpha\in U}$.\footnote{An easy argument shows that an $M$-ultrafilter on $\kappa$ that contains all final segments of $\kappa$ in $M\models\ZFC^-$ is $M$-normal if and only if it is ${<}\kappa$-complete and $M$-normal with respect to $\subseteq$-decreasing sequences.} 

  \item $U$ is \emph{countably complete} if whenever $\seq{x_n}{n<\omega}$ is a sequence of elements of $U$, then $\bigcap_{n<\omega}x_n\neq\emptyset$. 

  \item $U$ is \emph{stationary-complete} if whenever $\seq{x_n}{n<\omega}$ is a sequence of elements of $U$, then $\bigcap_{n<\omega}x_n$ is a stationary subset of $\kappa$. 

  \item $U$ is \emph{genuine} if either $\betrag{U}=\kappa$ and $\Delta_{\alpha<\kappa}U_\alpha$ is unbounded in $\kappa$ for every sequence $\seq{U_\alpha}{\alpha<\kappa}$ of elements of $U$, or $\betrag{U}<\kappa$ and $\bigcap U$ is unbounded in $\kappa$. 

  \item $U$ is \emph{normal} if either $\Delta_{\alpha<\kappa}U_\alpha$ is a stationary subset of $\kappa$ for every (or equivalently, for some) enumeration $\seq{U_\alpha}{\alpha<\kappa}$ of $U$, or $\betrag{U}<\kappa$ and $\bigcap U$ is a stationary subset of $\kappa$. 
  \end{itemize}
\end{definition}

The following observation shows that for weak $\kappa$-models, the first columns of Table \ref{table:schemeIntermediate} and of Table  \ref{table:schemeKappa} list the above properties according to their strength.

\begin{lemma}
  In the situation of Definition \ref{ultrafilterproperties}, if additionally $\kappa$ is regular and $\kappa\subseteq M$, and $U$ is genuine and contains all final segments of $\kappa$ in $M$, then $U$ is $M$-normal and stationary-complete.
\end{lemma}

\begin{proof}
  Assume that $U$ is a genuine $M$-ultrafilter. In order to show that $U$ is $M$-normal, let $\seq{x_\alpha}{\alpha<\kappa}$ be a sequence of elements of $U$ in $M$. If this sequence had diagonal intersection $x\notin U$, then the complement $y$ of $x$ would be an element of $U$. But then the set $\Delta_{\alpha<\kappa}(y\cap x_\alpha)=\{0\}$ is a diagonal intersection of elements of $U$, contradicting that $U$ is genuine.

In order to show that $U$ is stationary-complete, let $\seq{x_n}{n<\omega}$ be a sequence of elements of $U$ and let $C$ be a club subset of $\kappa$ consisting of limit ordinals. 
 Let $\seq{U_\alpha}{\alpha<\kappa}$ be the sequence of elements of $U$ with $U_n=x_n$ for every $n<\omega$ and $U_\alpha=\kappa\setminus\min(C\setminus(\alpha+1))$ for every $\omega\leq\alpha<\kappa$. 
 It is easy to check that $$\Delta_{\alpha<\kappa} U_\alpha ~ \subseteq ~ \omega ~ \cup ~ C$$ holds.  
 Using that $U$ is genuine, it follows that $\Delta_{\alpha<\kappa}U_\alpha$ is an unbounded subset of $\kappa$. In particular, we know that $C\cap\bigcap_{i<\omega}x_i\neq\emptyset$, showing that $U$ is stationary-complete. 
\end{proof}

\begin{lemma}\label{lemma:AllIdealsNormal}
  If $\Psi(M,U)$ implies that $U$ is $M$-normal, then the ideals $I_\Psi^{{<}\kappa}$, $I_\Psi^\kappa$ and $I_{{\prec}\Psi}^\kappa$ are all normal. 
\end{lemma}

\begin{proof}
  We will only present the proof for the ideal $I_\Psi^{{<}\kappa}$, which requires the most difficult argument of the three. 
   Assume thus that $\vec A=\seq{A_\alpha}{\alpha<\kappa}$ is a sequence of elements of $I_\Psi^{{<}\kappa}$. Fix $\theta$ sufficiently large, and fix a sequence $\vec x=\seq{x_\alpha}{\alpha<\kappa}$ such that for every $\alpha<\kappa$, the set $x_\alpha$ witnesses that $A_\alpha\in I_\Psi^{<\kappa}$ with respect to $\theta$. 
   Pick $M\prec\HH{\theta}$ such that $\vec x,\vec A\in M$, $\betrag{M}<\kappa$ and $\Psi(M,U)$ holds. By elementarity, $\vec{A}\in M$ implies that $\nabla\vec A\in M$.  Since the model $\langle M,U\rangle$ thinks that $A_\alpha\not\in U$ for every $\alpha<\kappa$ and $\Psi(M,U)$ implies that $U$ is $M$-normal, it follows that $\nabla\vec A\not\in U$. This argument shows that $\nabla\vec A\in I_\Psi^{<\kappa}$.
\end{proof}

We now turn our attention to extended definitions regarding elementary embeddings. 
As mentioned above, we identify classes $M$ with the corresponding $\epsilon$-structures $\langle M,\in\rangle$. 
Given transitive classes $M$ and $N$, the critical point of an elementary embedding $\map{j}{M}{N}$ is simply defined as the least ordinal $\alpha\in M$ with $j(\alpha)>\alpha$. 
We need a generalization of this concept for elementary embeddings $\map{j}{M}{\langle N,\epsilon_N\rangle}$ when the class $M$ is not necessarily transitive and the $\epsilon$-structure $\langle N,\epsilon_N\rangle$ is not necessarily well-founded. 
In the following, we let $\map{j}{M}{\langle N,\epsilon_N\rangle}$ always denote an elementary embedding between $\epsilon$-structures, whose domain is a $\Sigma_0$-correct $\ZFC^-$-model.

\begin{definition}[Jump]
  Given $\map{j}{M}{\langle N,\epsilon_N\rangle}$ and an ordinal $\alpha\in M$, we say that \emph{$j$ jumps at $\alpha$} if there exists an $N$-ordinal $\gamma$ with $\gamma \IN_N j(\alpha)$ and  $j(\beta) \IN_N \gamma$ for all $\beta\in M\cap\alpha$. 
\end{definition}

Note that, in the above situation, for every $N$-ordinal $\gamma$, there is at most one ordinal $\alpha$ in $M$ such that $\gamma$ witnesses that $j$ jumps at $\alpha$. 
 Moreover, elementarity directly implies that elementary embeddings only jump at limit ordinals.

\begin{definition}[Critical Point]
  Given $\map{j}{M}{\langle N,\epsilon_N\rangle}$, if there exists an ordinal $\alpha\in M$ such that $j$ jumps at $\alpha$, then we denote the minimal such ordinal by $\crit{j}$, the \emph{critical point of $j$}. 
\end{definition}

It is easy to see that if $\crit{j}$ exists and $\alpha\in\Ord$, then $\crit{j}>\alpha$ holds if and only if $$j[M\cap\alpha] ~ = ~ \Set{\beta\in N}{\beta \IN_N j(\alpha)}.$$ 
This shows that the map $j\restriction(M\cap\crit{j})$ is an $\IN$-isomorphism between  $M\cap\crit{j}$ and the proper initial segment $\Set{\beta\in N}{\exists\alpha<\crit{j} ~ \beta \IN_N j(\alpha)}$ of the $N$-ordinals. 
 In particular, this initial segment is contained in the well-founded part of $\langle N,\IN_N\rangle$ and therefore the ordinal $\otp{M\cap\crit{j}}$ is a subset of the transitive collapse of this set. 
 We will tacitly make use of these facts throughout this paper.

Next, we need to generalize the notions of ${<}\kappa$- and $\kappa$-powerset preservation to a non-transitive context. 
The idea behind an embedding $\map{j}{M}{\langle N,\epsilon_N\rangle}$  being ${<}\kappa$-powerset preserving (respectively, $\kappa$-powerset preserving) is that $M$ and $N$ contain the same subsets of ordinals below $\kappa$ (respectively, the same subsets of $\kappa$). Since the relevant subsets of $M$ are, in a sense made precise below, always contained in $N$, only one of those inclusions is part of the following definitions.

\begin{definition}[${<}\kappa$-powerset preservation]\label{lkpp}
  Given $\map{j}{M}{\langle N,\epsilon_N\rangle}$ with $\crit{j}=\kappa$, the embedding $j$ is \emph{${<}\kappa$-powerset preserving} if $$\forall y\in N ~ \exists x\in M ~ \left[\exists\alpha<M\cap\kappa ~ \langle N,\epsilon_N\rangle\models\anf{y\subseteq j(\alpha)}  ~ \longrightarrow ~ j(x)=y\right].$$  
\end{definition}

As we will also see later on, this notion is an important concept in the study of embeddings between smaller models of set theory and it  turns out to be closely related to the behavior of the continuum function below $\kappa$ in $M$.  
 % is not of much interest in the context of transitive  models (or weak $\kappa$-models): if we assume that $\kappa$ is a cardinal satisfying $\kappa=2^{{<}\kappa}$, and that some subset of $\kappa$ coding $\HH{\kappa}$ is an element of $M$, then any embedding $\map{j}{M}{N}$ with $\crit{j}=\kappa$ is ${<}\kappa$-powerset preserving. 

 \begin{proposition}\label{proposition:CharaLessThanKappaPowerPreserving}
  Let $\map{j}{M}{\langle N,\epsilon_N\rangle}$ with $\crit{j}=\kappa$. 
  \begin{enumerate} 
   \item If $M\models\anf{\textit{If $\alpha<\kappa$, then $\POT{\alpha}$ exists and $2^\alpha<\kappa$}}$ holds, then $j$ is ${<}\kappa$-powerset preserving. 
   
   \item If $j$ is ${<}\kappa$-powerset preserving, then $M\models\anf{\textit{If $\alpha<\kappa$, then there is no injection from $\kappa$ into $\POT{\alpha}$}}$. 
  \end{enumerate}
 \end{proposition}
 
 \begin{proof}
  (1) Pick $\gamma\in M\cap \kappa$ and $y\in N$ with $\langle N,\IN_N\rangle\models\anf{y\subseteq j(\gamma)}$. By our assumptions, there is an enumeration $\vec{x}=\seq{x_\xi}{\xi<\alpha}$ of $M\cap\POT{\gamma}$ in $M$ with $\alpha<\kappa$. By elementarity, there is an $N$-ordinal $\beta$ with $\beta\IN_N j(\alpha)$ and $\langle N,\IN_N\rangle\models\anf{y=j(\vec{x})(\beta)}$. In this situation, the fact that $\alpha<\kappa$ yields $\xi\in M\cap\alpha$ satisfying $j(\xi)=\beta$ and $j(x_\xi)=y$. 
  
  (2) Assume, towards a contradiction, that $j$ is ${<}\kappa$-powerset preserving and $\map{\iota}{\kappa}{\POT{\alpha}}$ is an injection in $M$ with $\alpha<\kappa$. Let $\gamma$ be an $N$-ordinal witnessing that $j$ jumps at $\kappa$ and pick $y\in N$ satisfying $\langle N,\IN_N\rangle\models\anf{j(\iota)(\gamma)=y\subseteq j(\alpha)}$. 
  By our assumptions, there is $x\in M\cap\POT{\alpha}$ with $j(x)=y$ and elementarity yields $\xi\in M\cap\kappa$ with $\iota(\xi)=x$. Since $\iota$ is an injection, this implies that $j(\xi)=\gamma$, contradicting the fact that $\gamma$ witnesses that $j$ jumps at $\kappa$.   
 \end{proof}

The following definition shows that there is still a useful notion of a $\kappa$-powerset preserving elementary embedding, even if we do not have a representative for $\kappa$ in the target model of our embedding.

\begin{definition}[$\kappa$-powerset preservation]\label{kpp}
  Given $\map{j}{M}{\langle N,\epsilon_N\rangle}$ with $\crit{j}=\kappa$, the embedding $j$ is \emph{$\kappa$-powerset preserving} if $$\forall y\in N ~ \exists x\in M ~ \left[\langle N,\epsilon_N\rangle\models\anf{y\subseteq j(\kappa)}~ \longrightarrow ~ M\cap x=\Set{\alpha\in M\cap\kappa}{j(\alpha) \IN_N y}\right].$$
\end{definition}

 Note that if $M$ and $N$ are weak $\kappa$-models, then the usual notions of critical point and of $\kappa$-powerset  preservation for an embedding $\map{j}{M}{N}$ clearly coincide with our respective notions.

We close this section by isolating a property that implies the existence of a canonical representative for $\kappa$ in the target model of our elementary embedding. 

\begin{definition}[$\kappa$-embedding]
  Given $\map{j}{M}{\langle N,\epsilon_N\rangle}$ that jumps at $\kappa$, the embedding $j$ is a \emph{$\kappa$-embedding} if there exists an $\epsilon_N$-minimal $N$-ordinal $\gamma$ witnessing that $j$ jumps at $\kappa$. We denote this ordinal by $\kappa^N$.\footnote{If $\crit j=\kappa$, then Proposition \ref{proposition:KappaEmbeddingOrdinalHeight} shows that $\kappa^N$ is the unique $N$-ordinal on which the $\epsilon_N$-relation has order-type $M\cap\kappa$. Otherwise, $\kappa^N$ might also depend on the embedding $j$, which we nevertheless suppress in our notation.}
\end{definition}

\begin{proposition}\label{proposition:KappaEmbeddingOrdinalHeight}
 Given $\map{j}{M}{\langle N,\IN_N\rangle}$ with $\crit{j}=\kappa$, the following statements are equivalent: 
 \begin{enumerate}
  \item $j$ is a $\kappa$-embedding. 
  
  \item The ordinal $\otp{M\cap\kappa}$ is an element of the transitive collapse of the well-founded part of $\langle N,\IN_N\rangle$. 
 \end{enumerate}
\end{proposition}

\begin{proof}
 First, assume that $j$ is a $\kappa$-embedding. Fix $\gamma\in N$ with $\gamma\IN_N\kappa^N$. Since $\kappa$ is a limit ordinal and $\kappa^N\IN_N j(\kappa)$, the fact that $\gamma$ does not witness that $j$ jumps at $\kappa$ yields an $\alpha\in M\cap\kappa$ with $\gamma\IN_N j(\alpha)$. By our earlier observations, this shows that $j[M\cap\kappa]=\Set{\beta\in N}{\beta\IN_N\kappa^N}$. In particular, the $N$-ordinal $\kappa^N$ is contained in the well-founded part of $\langle N,\IN_N\rangle$ and the transitive collapse of this set maps $\kappa^N$ to $\otp{M\cap\kappa}$. 
 
 In the other direction, assume that there is an $N$-ordinal $\beta$ in the well-founded part  of $\langle N,\IN_N\rangle$ that is mapped to $\otp{M\cap\kappa}$ by the transitive collapse of this set. By our earlier observations, this shows that $j[M\cap\kappa]=\Set{\delta\in N}{\delta\IN_N\beta}$. Let $\gamma\in N$ witness that $j$ jumps at $\kappa$. By our computations, we then have either $\beta=\gamma$ or $\beta\IN_N\gamma$. In particular, we have $\beta\IN_N j(\kappa)$ and $\beta$ witnesses that $j$ jumps at $\kappa$. But these computations also show that $\beta$ is the $\IN_N$-minimal $N$-ordinal with this property.  
\end{proof}

%%%%%%%%%%%%%%%%%%%%%%%%%%%%%%
%%%%%%%%%%%%%%%%%%%%%%%%%%%%%%

\section{Correspondences between ultrapowers and elementary embeddings}\label{section:ultrapowersandembeddings}

The results of this section will allow us to interchangeably talk about ultrafilters or about embeddings for models of $\ZFC^-$. 
If $M$ is a class that is a $\Sigma_0$-correct model of $\ZFC^-$, $\kappa$ is a cardinal of $M$, and $U$ is an $M$-ultrafilter on $\kappa$, then we can use the $\Sigma_0$-correctness of $M$\footnote{Note that, given a $\Sigma_0$-correct $\ZFC^-$-model $M$ and functions $\map{f,g}{\kappa}{M}$ in $M$, then the set $\Set{\alpha<\kappa}{f(\alpha)=g(\alpha)}$ and $\Set{\alpha<\kappa}{f(\alpha)\in g(\alpha)}$ are both contained in $M$ and satisfy the same defining properties in it.} to define the induced ultrapower embedding $\map{j_U}{M}{\langle\Ult{M}{U},\epsilon_U\rangle}$ as usual: 
define an equivalence relation $\equiv_U$ on the class of all functions $\map{f}{\kappa}{M}$ contained in $M$ by setting $f\equiv_U g$ if and only if $\Set{\alpha<\kappa}{f(\alpha)=g(\alpha)}\in U$, 
let $\Ult{M}{U}$ consists of all sets $[f]_U$ of rank-minimal elements of $\equiv_U$-equivalence classes, 
define $[f]_U \IN_U [g]_U$ to hold if and only if $\Set{\alpha<\kappa}{f(\alpha)\in g(\alpha)}\in U$ 
and set $j_U(x)=[c_x]_U$, where $c_x\in M$ denotes the constant function with domain $\kappa$ and value $x$. 
It is easy to check that the assumption that $M\models\ZFC^{-}$ implies that {\L}os' Theorem still holds true in our setting, i.e. we have $$\Ult{M}{U}\models\varphi([f_0]_U,\ldots,[f_{n-1}]_U) ~ \Longleftrightarrow ~ \langle M,U\rangle\models\anf{\exists x\in U ~ \forall\alpha\in x ~ \varphi(f_0(\alpha),\ldots,f_{n-1}(\alpha))}$$ for every first order $\IN$-formula $\varphi(v_0,\ldots,v_{n-1})$ and all functions $\map{f_0,\ldots,f_{n-1}}{\kappa}{M}$ in $M$.

Given an elementary embedding $\map{j}{M}{\langle N,\epsilon_N\rangle}$ that jumps at $\kappa$, let $\gamma$ be a witness for this, and let $$U_j^\gamma ~ = ~ \Set{A\in M\cap\POT{\kappa}}{\gamma \IN_N j(A)}$$ denote the $M$-ultrafilter induced by $\gamma$ and by $j$. Since $\gamma$ is not in the range of $j$, the filter $U_j^\gamma$ is non-principal. 
If $j$ is a $\kappa$-embedding and $\gamma=\kappa^N$, then we call $U_j=U_j^\gamma$ the \emph{canonical} $M$-ultrafilter induced by $j$, or simply \emph{the} $M$-ultrafilter induced by $j$.

In the following, we say that a property $\Psi(M,U)$ of $\Sigma_0$-correct $\ZFC^-$-models $M$ and $M$-ultrafilters $U$ \emph{corresponds} to a property $\Theta(M,j)$ of such models $M$ and elementary embeddings $\map{j}{M}{\langle N,\epsilon_N\rangle}$ if the following statements hold: 
\begin{itemize}   

  \item If $\Psi(M,U)$ holds for an $M$-ultrafilter $U$,  then $\Theta(M,j_U)$ holds. 
 
  \item If $\Theta(M,j)$ holds for an elementary embedding $\map{j}{M}{\langle N,\epsilon_N\rangle}$ and $\gamma$ witnesses that $j$ jumps at some $\kappa\in M$, then $\Psi(M,U_j^\gamma)$ holds. 
\end{itemize}

Most of the correspondences below are well-known, in a perhaps slightly less general setup.

\begin{proposition}\label{proposition:correspondence1}
 Let $\kappa$ be an ordinal. 

  \begin{enumerate} 
    \item \anf{$U$ is an $M$-ultrafilter on $\kappa\in M$ that contains all final segments of $\kappa$ in $M$} corresponds to \anf{$j$ jumps at $\kappa\in M$}. 

    \item Given $\alpha\leq\kappa$, \anf{$\alpha\in M$ and $U$ is an $M$-ultrafilter on $\kappa\in M$ that is ${<}\alpha$-complete for $M$ and contains all final segments of $\kappa$ in $M$} corresponds to \anf{$\crit{j}\geq\alpha\in M$ and $j$ jumps at $\kappa\in M$}. 

    \item \anf{$U$ is a non-principal $M$-ultrafilter on $\kappa\in M$ that is ${<}\kappa$-complete for $M$} corresponds to \anf{$\crit{j}=\kappa\in M$}. 

    \item \anf{$U$ is a non-principal $M$-ultrafilter on $\kappa\in M$ that is ${<}\kappa$-complete for $M$ and ${<}\kappa$-amenable for $M$} corresponds to \anf{$\crit{j}=\kappa\in M$ and $j$ is ${<}\kappa$-powerset preserving}. 

    \item \anf{$U$ is a non-principal $M$-ultrafilter on $\kappa\in M$ that is ${<}\kappa$-complete and $\kappa$-amenable for $M$} corresponds to \anf{$\crit{j}=\kappa\in M$ and $j$ is $\kappa$-powerset preserving}.
  \end{enumerate}
\end{proposition}

\begin{proof}
 Throughout this proof, we let $M$ denote a $\Sigma_0$-correct $\ZFC^-$-model with $\kappa\in M$. 

 (1) Let $U$ be an $M$-ultrafilter on $\kappa$ that contains all final segments of $\kappa$ in $M$. 
 Then $\id_\kappa\in M$ and, given $\alpha<\kappa$ in $M$, we have $(\alpha,\kappa)\in U$  and hence $j_U(\alpha)=[c_\alpha]_U \IN_U [\id_\kappa]_U \IN_U j_U(\kappa)$. Hence $[\id_\kappa]_U$ witnesses that $j_U$ jumps at $\kappa$. 
 In the other direction, if $\gamma$ witnesses that $\map{j}{M}{\langle N,\epsilon_N\rangle}$ jumps at $\kappa$ and $\alpha\in M\cap\kappa$, then we have $[\alpha,\kappa)\in M$,  $\gamma \IN_N [j(\alpha),j(\kappa))^N=j([\alpha,\kappa))$ and hence $[\alpha,\kappa)\in U_j^\gamma$, as desired.

 (2) Pick $\alpha\leq\kappa$ in $M$. Let $U$ be an $M$-ultrafilter on $\kappa$ that is ${<}\alpha$-complete for $M$ and contains all final segments of $\kappa$ in $M$. 
 By (1), $\crit{j_U}$ exists. Assume, towards a contradiction, that $j_U$ jumps at $\beta<\alpha$. Then there is $\map{f}{\kappa}{\beta}$ in $M$ such that $[c_\delta]_U \IN_U [f]_U \IN_U [c_\beta]_U$ holds for every $\delta\in M\cap\beta$. By our assumptions on $M$, there is a sequence $\seq{x_\delta}{\delta<\beta}$ of subsets of $\kappa$ in $M$ such that $x_\delta=\Set{\xi<\kappa}{\delta<f(\xi)<\beta}$ for all $\delta<\beta$ and $x_\delta\in U$ for all $\delta\in M\cap\beta$. In this situation, the ${<}\alpha$-completeness of $U$ implies that $\bigcap_{\delta<\beta}x_\delta\in U$. Pick $\xi\in M\cap\bigcap_{\delta<\beta}x_\delta$. Then $f(\xi)\in M\cap\beta$ and $\xi\in x_{f(\xi)}$, a contradiction. 
 
 In the other direction, assume that $\gamma$ witnesses $\map{j}{M}{\langle N,\epsilon_N\rangle}$ to jump at $\kappa$, and that $\crit{j}\geq\alpha$. Pick a sequence $\vec{x}=\seq{x_\delta}{\delta<\beta}\in M$ with $\beta<\alpha$, and with $\gamma \IN_N j(x_\delta)$ for all $\delta\in M\cap\beta$. 
By our assumption, we have $j[\beta]=\Set{\delta\in N}{\delta \IN_N j(\beta)}$, and hence  $\gamma \IN_N j(\vec{x})(\xi)$ holds for all $\xi\in N$ with $\xi \IN_N j(\beta)$. This allows us to conclude that $\gamma \IN_N j(\bigcap_{\delta<\beta}x_\delta)$ and hence $\bigcap_{\delta<\beta}x_\delta\in U_j^\gamma$. 

 (3) This statement is a direct consequence of (2), because every non-principal $M$-ultrafilter on $\kappa$ that is ${<}\kappa$-complete for $M$ contains all final segments of $\kappa$ in $M$. 

 (4) Let $U$ be a non-principal $M$-ultrafilter on $\kappa$ that is ${<}\kappa$-complete for $M$ and ${<}\kappa$-amenable for $M$. Then (3) implies that $\crit{j_U}=\kappa$. 
 Fix a function $\map{f}{\kappa}{M}$ in $M$ and $\alpha\in M\cap\kappa$ such that $[f]_U$ is a subset of $j_U(\alpha)$ in $\langle\Ult{M}{U},\epsilon_U\rangle$. 
 Then the sequence $\seq{x_\beta}{\beta<\alpha}$ with $x_\beta=\Set{\xi<\kappa}{\beta\in f(\xi)}$ for all $\beta<\alpha$ is an element of $M$. Given $\beta\in M\cap\alpha$, \L{}os' Theorem implies that $j_U(\beta) \IN_U [f]_U$ if and only if $x_\beta\in U$. The ${<}\kappa$-amenability of $U$ now yields an $x\in M$ with $M\cap x=\Set{\beta\in M\cap\alpha}{j_U(\beta) \IN_U [f]_U}$. 
 Since $j_U[\alpha]=\Set{\gamma\in\Ult{M}{U}}{\gamma \IN_U j_U(\alpha)}$, extensionality allows us to conclude that $j_U(x)=[f]_U$. 

 In the other direction, let $\map{j}{M}{\langle N,\epsilon_N\rangle}$ be a ${<}\kappa$-powerset preserving elementary embedding with $\crit{j}=\kappa$ and let $\gamma$ be any witness that $\crit{j}=\kappa$. By (3), we know that $U^\gamma_j$ is ${<}\kappa$-complete for $M$ and non-principal. Fix a sequence $\vec{x}=\seq{x_\beta}{\beta<\alpha}$ of subsets of $\kappa$ in $M$ with $\alpha<\kappa$. 
 Then there is $y\in N$ with $$\langle N,\epsilon_N\rangle\models\anf{y\subseteq j(\alpha) ~ \wedge ~ \forall\beta<j(\alpha) ~ [\beta\in y ~ \longleftrightarrow ~ \gamma\in j(\vec{x})(\beta)]}.$$ By our assumption, there is $x\in M$ with $j(x)=y$ and $$x_\beta\in U^\gamma_j ~ \Longleftrightarrow ~ \gamma \IN_N j(x_\beta) ~ = ~ j(\vec{x})(j(\beta)) ~ \Longleftrightarrow ~ j(\beta) \IN_N y ~ = ~ j(x) ~ \Longleftrightarrow ~ \beta\in x$$ for all $\beta\in M\cap\alpha$. This shows that $\langle M,U^\gamma_j\rangle\models\anf{\forall\beta<\alpha ~ [x_\beta\in U \longleftrightarrow \beta\in x]}$.

(5) Let $U$ be a ${<}\kappa$-complete, non-principal and $\kappa$-amenable $M$-ultrafilter on $\kappa$. By (3), we have $\crit{j}=\kappa$. Fix a function $\map{f}{\kappa}{M}$ in $M$ with the property that $[f]_U$ is a subset of $j_U(\kappa)$ in $\langle\Ult{M}{U},\epsilon_U\rangle$. 
 Then $M$ contains the sequence $\seq{x_\beta}{\beta<\kappa}$ with $x_\beta=\Set{\xi<\kappa}{\beta\in f(\xi)}$ and $\kappa$-amenability yields an $x\in M$ with $M\cap x=\Set{\beta\in M\cap\kappa}{x_\beta\in U}$. Given $\beta\in M\cap\kappa$, it is now easy to see that $\beta\in x$ if and only if $j_U(\beta)\epsilon_U[f]_U$. 

 In the other direction, let $\map{j}{M}{\langle N,\epsilon_N\rangle}$ be $\kappa$-powerset preserving with $\crit{j}=\kappa$ and let $\gamma$ be a witness that $j$ jumps at $\kappa$. Then (3) shows that $U$ is ${<}\kappa$-complete and non-principal. 
 Fix a sequence $\seq{x_\beta}{\beta<\kappa}$ of subsets of $\kappa$ in $M$ and $y\in N$ with $$\langle N,\epsilon_N\rangle\models\anf{y\subseteq j(\kappa) ~ \wedge ~ \forall\beta<j(\kappa) ~ [\beta\in y ~ \longleftrightarrow ~ \gamma\in j(\vec{x})(\beta)]}.$$ Then there is $x\in M$ with $M\cap x=\Set{\beta\in M\cap\kappa}{j(\beta) \IN_N y}$. 
 Given $\beta\in M\cap \kappa$, we then have $\beta\in x$ if and only if $x_\beta\in U^\gamma_j$. 
\end{proof}

We next consider situations in which an elementary embedding $\map{j}{M}{\langle N,\epsilon_N\rangle}$ induces a canonical $M$-ultrafilter $U$ on $\kappa$, i.e.  situations in which $j$ is a $\kappa$-embedding. 
Given an ordinal $\kappa$,  a property $\Psi(M,U)$ of $\Sigma_0$-correct $\ZFC^-$-models $M$ containing $\kappa$ and $M$-ultrafilters $U$ on $\kappa$ \emph{$\kappa$-corresponds} to a property $\Theta(M,j)$ of such models $M$ and elementary embeddings $\map{j}{M}{\langle N,\epsilon_N\rangle}$ if the following statements hold: 
\begin{itemize} 
 \item If $\Psi(M,U)$ holds for an $M$-ultrafilter $U$ on $\kappa$,  then $\Theta(M,j_U)$ holds. 
 
  \item If $\Theta(M,j)$ holds for an elementary embedding $\map{j}{M}{\langle N,\epsilon_N\rangle}$, then $j$ is a $\kappa$-embedding and $\Psi(M,U_j)$ holds. 
  \end{itemize}

\begin{proposition}\label{proposition:correspondence2}
 \anf{$U$ is an $M$-ultrafilter on $\kappa$ that is $M$-normal with respect to $\subseteq$-decreasing sequences and contains all final segments of $\kappa$ in $M$} $\kappa$-corresponds to \mbox{\anf{$j$ is a $\kappa$-embedding}.}
\end{proposition}

\begin{proof}
 Let $M$ denote a $\Sigma_0$-correct $\ZFC^-$-model with $\kappa\in M$. 
  
   First, assume that $U$ is $M$-normal with respect to $\subseteq$-decreasing sequences and contains all final segments of $\kappa$ in $M$. Then, the proof of Proposition \ref{proposition:correspondence1}.(1) shows that $[\id_\kappa]_U$ witnesses that $j_U$ jumps at $\kappa$. 
 Assume, towards a contradiction, that there is an $\map{f}{\kappa}{M}$ in $M$ with $[f]_U \IN_U [\id_\kappa]_U$ and $j_U(\beta) \IN_U [f]_U$ for all $\beta<\kappa$. 
 Then the sequence $\seq{x_\beta}{\beta<\kappa}$ with $x_\beta=\Set{\xi<\kappa}{\beta<f(\xi)<\xi}$ for all $\beta<\kappa$ is an element of $M$, and we have $x_\beta\in U$ for all $\beta\in M\cap\kappa$. 
 Since this sequence is $\subseteq$-decreasing, we know that $\Delta_{\beta<\kappa}x_\beta\in U$. But then, there is $\xi\in M\cap\Delta_{\beta<\kappa}x_\beta$ with $\xi\in x_{f(\xi)}$, a contradiction. This shows that $[\id_\kappa]_U$ witnesses that $j_U$ is a $\kappa$-embedding. 

 Now, assume that $\map{j}{M}{\langle N,\epsilon_N\rangle}$ is a $\kappa$-embedding. Then, Proposition \ref{proposition:correspondence1}.(1) shows that $U$ contains all final segments of $\kappa$ in $M$. Let $\vec{x}=\seq{x_\beta}{\beta<\kappa}$ be a $\subseteq$-decreasing sequence of subsets of $\kappa$ in $M$ with $x_\beta\in U_j$ for all $\beta\in M\cap\kappa$. 
 Pick $\gamma\in N$ with $\gamma \IN_N \kappa^N$.  Then the minimality of $\kappa^N$ yields $\beta\in M\cap\kappa$ with $\gamma \IN_N j(\beta) \IN_N \kappa^N$. Since $\langle N,\epsilon_N\rangle$ believes that $j(\vec{x})$ is $\subseteq$-decreasing and $x_\beta\in U_j$ implies that $\kappa^N \IN_N j(x_\beta)$, this shows that $\kappa^N \IN_N j(\vec{x})(\gamma)$. But this shows that $\kappa^N ~ \epsilon_N ~ j(\Delta_{\beta<\kappa} x_\beta)$ and hence $\Delta_{\beta<\kappa}x_\beta\in U_j$.  
\end{proof}

If $\map{j}{M}{\langle N,\epsilon_N\rangle}$ is a $\kappa$-embedding that is induced by an $M$-ultrafilter $U$, we may also write $\kappa^U$ rather than $\kappa^N$.
We can now add the assumptions from Proposition  \ref{proposition:correspondence2} to each item in Proposition  \ref{proposition:correspondence1}. 
For example, Clauses (4) and  (5) in Proposition \ref{proposition:correspondence1} yields the following:

\begin{corollary}\label{corollary:correspondence1}
  \anf{$U$ is an $M$-ultrafilter on $\kappa$ that contains all final segments of $\kappa$ in $M$, 
  %is uniform, 
  and is $M$-normal and ${<}\kappa$-amenable (respectively, $\kappa$-amenable) for $M$}  $\kappa$-corresponds to \anf{$\crit{j}=\kappa\in M$ and $j$ is a ${<}\kappa$-powerset preserving (respectively,  $\kappa$-powerset preserving) $\kappa$-embedding}. \qed 
\end{corollary}

\begin{remark}
 Using the above results, one could easily rephrase  the results from \cite{MR0460120}, \cite{MR513844} and \cite{MR2617841} cited in the introduction in order to obtain characterizations of inaccessible, of weakly compact, and of completely ineffable cardinals in terms of the existence of certain elementary embeddings on countable elementary submodels of structures of the form $\HH{\theta}$. We leave this -- given the above results, straightforward -- task to the interested reader (for inaccessible cardinals, this was done in \cite{ranosch}).  
\end{remark}

The following lemma will be useful later on.

\begin{lemma}\label{lemma:kappapowersetpreservings}
  Let $\kappa$ be an inaccessible cardinal, let $\map{b}{\kappa}{\VV_\kappa}$ be a bijection, let $M$ be a $\Sigma_0$-correct model of $\ZFC^-$ with $(\kappa+1)\cup\{b\}\subseteq M$ and let  $\map{j}{M}{\langle N,\epsilon_N\rangle}$ be a $\kappa$-powerset preserving $\kappa$-embedding with $\crit{j}=\kappa$. 
 \begin{enumerate} 
  \item The map $$\Map{j_*}{M\cap\VV_{\kappa+1}}{\langle\Set{y\in N}{y\IN_N\VV^N_{\kappa^N+1}},\epsilon_N\rangle}{x}{(j(x)\cap V_{\kappa^N})^N}$$ is an $\IN$-isomorphism extending $j\restriction(M\cap\VV_\kappa)$. 

  \item\label{lemma:kappapowersetpreserving:2} There is an $\IN$-isomorphism $$\map{j^*}{\HH{\kappa^+}^M}{\langle\Set{y\in N}{y\IN_N\HH{(\kappa^N)^+}^N},\epsilon_N\rangle}$$ extending $j_*$. 
 \end{enumerate}
\end{lemma}

\begin{proof}
 (1) First, note that, using $\Sigma_0$-correctness,  one can show that $\ran{b}^M=\ran{b}=\VV_\kappa=\VV_\kappa^M\in M$ and $\VV_{\kappa+1}^M=\VV_{\kappa+1}\cap M$. 
 Now, if $x\in\VV_\kappa$, then there is $\alpha<\kappa$ with $x\subseteq\VV_\alpha$ and, since $j(\alpha)\IN_N\kappa^N$ holds, we have $j_*(x)=j(x)\IN_N\VV_{\kappa^N}^N$. 
 This shows that $j_*(x)\IN_N\VV_{\kappa^N+1}^N$ for all $x\in\VV_\kappa$ and $j_*\restriction\VV_\kappa$ is an $\IN$-homomorphism, i.e. given $x,y\in\VV_\kappa$, we have $x\in y$ if and only if $j_*(x)\IN_N j_*(y)$. 
  The proof of Proposition \ref{proposition:KappaEmbeddingOrdinalHeight} shows that for every $N$-ordinal $\gamma$ with $\gamma\IN_N\kappa^N$, there is $\beta\in\kappa$ with $j(\beta)=\gamma$. 
%  \begin{proof}[Proof of the Claim]
%  Assume that the statement fails and pick the least $\beta\in M$  such that $\gamma\epsilon_N j(\beta)$. By the definition of $\kappa^N$, such $\beta$ exists and is strictly smaller than $\kappa$. But then, $\gamma$ witnesses that $j$ jumps at $\beta$, contradicting that $\kappa$ is the critical point of $j$.
%  \end{proof}
 % 
   In particular, if $z\in\VV_{\kappa+1}^M\setminus\VV_\kappa$, then $j_*(z)\IN_N(\VV_{\kappa^N+1})^N$ and  $j_*(z)\not{\hspace{-2.3pt}\IN}_N\VV_{\kappa^N}^N$. 
 This shows that $j_*$ is an $\IN$-homomorphism and, by Extensionality, this also shows that $j_*$ is injective.

 Now, pick a club subset $C$ of $\kappa$ in $M$ with $b[\alpha]=\VV_\alpha$ for all $\alpha\in C$. Note that the bijectivity of $b$ implies that $b[x]\cap\VV_\alpha=b[x\cap\alpha]$ holds for all $\alpha\in C$ and $x\subseteq\kappa$. 
 Since $j[\kappa]$ is the set of all elements of $\kappa^N$ in $\langle N,\IN_N\rangle$, elementarity implies that $\kappa^N\IN_N j(C)$ and  
 %
 %In this situation, the above claim and elementarity imply that 
 %
 $$j[\VV_\kappa] ~ = ~ (j\circ b)[\kappa]~ = ~ \Set{y\in N}{y\IN_N(j(b)[\kappa^N])^N} ~ = ~ \Set{y\in N}{y\IN_N\VV^N_{\kappa^N}}.$$

 Finally, pick $z\in N$ with $z\IN_N\VV^N_{\kappa^N+1}$. By elementarity and the above computations, there is $y\in N$ with $\langle N,\IN_N\rangle\models\anf{y\subseteq\kappa^N\wedge j(b)[y]=z}$. 
  Using that $j$ is $\kappa$-powerset preserving, pick $x\in\POT{\kappa}^M$ satisfying  $$x ~ = ~ M\cap x ~ = ~ \Set{\alpha<\kappa}{j(\alpha)\IN_N y}.$$  
Then $b[x]\in\VV_{\kappa+1}^M$ and the fact that $\kappa^N\IN j(C)$ implies that $$j_*(b[x]) ~ = ~ ((j(b)[j(x)])\cap\VV_{\kappa^N})^N ~ = ~ (j(b)[j(x)\cap\kappa^N])^N.$$ 

 \begin{claim*}
  $j_*(b[x])=z$. 
 \end{claim*}

 \begin{proof}[Proof of the Claim]
  First, fix $u\in N$ with $u\IN_N j_*(b[x])$. Then there is an $N$-ordinal $\gamma$ with $\gamma\IN_N\kappa^N$, $\gamma\IN_N j(x)$ and $(j(b)(\gamma))^N=u$. 
 In this situation, we can find $\beta<\kappa$ with $j(\beta)=\gamma$ and $j(b(\beta))=u$. But then elementarity implies that $\beta\in x$, $\gamma\IN_N y$ and $u\IN_N z$. 
  In the other direction, fix $u\in N$ with $u\IN_N z$. Then there is an $N$-ordinal $\gamma$ with $\gamma\IN_N y$ and $(j(b)(\gamma))^N=u$. As above, we can find $\beta<\kappa$ with $j(\beta)=\gamma$. Then $\beta\in x$ and $u=j(b(\beta))\IN_N j_*(b[x])$. 
 By Extensionality, these computations yield the desired equality. 
 \end{proof}

 Since the above claim shows that $j_*$ is surjective, we now know that this map is an $\IN$-isomorphism. 

\medskip

 (2)  We are going to make use of the standard coding of elements of $\HH{\kappa^+}$ by subsets of $\kappa$: There are first order $\IN$-formulas $\varphi_0(v_0)$, $\varphi_1(v_0,v_1)$, $\varphi_2(v_0,v_1)$, $\varphi_3(v_0,v_1,)$ and $\varphi_4(v_0,v_1,v_2)$ 
 with the property that the axioms of $\ZFC^-$ prove that whenever $\map{b}{\kappa}{\VV_\kappa}$ is a bijection for some inaccessible cardinal $\kappa$, then 
 \begin{itemize} 
  \item the formula $\varphi_0$ defines a \emph{set of codes} $D_\kappa\subseteq\pow(\kappa)$,\footnote{The set $D_\kappa$ consists of subsets of $\kappa$ that code sets of hereditary cardinality at most $\kappa$ in some canonical way. For example, we can define $D_\kappa$ to consist of all $x\subseteq\kappa$ with the property that there exists an element $y\in\HH{\kappa^+}$ and a surjection $\map{s}{\kappa}(\tcl{\{y\})}$ with $$x ~ = ~ \Set{\goedel{0}{\alpha}}{\alpha<\kappa, ~ s(\alpha)\in y} ~ \cup ~ \Set{\goedel{1}{\goedel{\alpha}{\beta}}}{\alpha,\beta<\kappa, ~ s(\alpha)\in s(\beta)},$$ where $\map{\goedel{\cdot}{\cdot}}{\Ord\times\Ord}{\Ord}$ denotes the G\"odel pairing function.} 

  \item the formula $\varphi_1$ defines an equivalence relation $\equiv_\kappa$ on $D_\kappa$, 

  \item the formula $\varphi_2$ defines an $\equiv_\kappa$-invariant binary relation $E_\kappa$ on $D_\kappa$,   

 \item the formula $\varphi_3$ defines an  epimorphism $\map{\pi_\kappa}{\langle D_\kappa,\equiv_\kappa,E_\kappa\rangle}{\langle \HH{\kappa^+},=,\in\rangle}$, and 

  \item the formula $\varphi_4$ and the parameter $b$ define a function $\map{b_*}{\VV_{\kappa+1}}{D_\kappa}$ with $\pi_\kappa\circ b_*=\id_{\VV_{\kappa+1}}$. 
\end{itemize}

 By (1), we now obtain a map $\map{j^*}{\HH{\kappa^+}^M}{\Set{y\in N}{y\IN_N\HH{(\kappa^N)^+}^N}}$ extending $j_*$ with $$\langle N,\IN_N\rangle\models\anf{\pi_{\kappa^N}(j_*(x))=j^*(\pi_\kappa(x))}$$ for all $x\in D_\kappa$. 
 By the above assumptions on the uniform definability of $D_\kappa$, $\equiv_\kappa$ and $E_\kappa$, we can conclude that $j^*$ is an $\IN$-isomorphism extending $j_*$. 
\end{proof}

We introduce one further type of correspondence between ultrafilters and elementary embeddings by saying that, given an ordinal $\kappa$, a property $\Psi(M,U)$ of $\Sigma_0$-correct $\ZFC^-$-models $M$ containing $\kappa$ and $M$-ultrafilters $U$ on $\kappa$ \emph{weakly $\kappa$-corresponds} to a property $\Theta(M,j)$ of such models $M$ and elementary embeddings $\map{j}{M}{\langle N,\epsilon_N\rangle}$ if the following properties hold:
  \begin{itemize} 
  \item Whenever $\Psi(M,U)$ holds for an $M$-ultrafilter $U$ on $\kappa$, then $\Theta(j_U,M)$ holds. 

  \item $\Theta(M,j)$ implies that $j$ jumps at $\kappa$. 
  
  \item Whenever $\Theta(j,M)$ holds for an elementary embedding $\map{j}{M}{\langle N,\IN_N\rangle}$, then $\Psi(M,U_j^\gamma)$ holds for some $N$-ordinal $\gamma$ witnessing that $j$ jumps at $\kappa$.
  \end{itemize}

Note that if $\Psi(M,U)$ corresponds to $\Theta(M,j)$, then these properties weakly $\kappa$-correspond for some ordinal $\kappa$.
Moreover, if $\Psi(M,U)$ weakly $\kappa$-corresponds  to $\Theta(M,j)$ and $\Theta(M,j)$ implies that $j$ is a $\kappa$-embedding, then these properties also $\kappa$-correspond. 
 Finally, if $\Psi(M,U)$ and $\Theta(M,j)$ $\kappa$-correspond, then they also weakly $\kappa$-correspond. 
 Our next result is an easy consequence of \L{}os' theorem, and is a most frequently used standard result in a less general setup.

\begin{lemma}\label{correspondence3}
  Given $A\subseteq\kappa$,  \anf{$A\in U$ and $U$ contains all final segments of $\kappa$ in $M$} weakly $\kappa$-corresponds to \anf{$A\in M$ and there is $\gamma\in N$ with $\gamma\IN_N j(A)$ witnessing that $j$ jumps at $\kappa$}.
\end{lemma}

\begin{proof}
  First, assume that $M$ is a $\Sigma_0$-correct $\ZFC^-$-model with $\kappa\in M$ and $U$ is an $M$-ultrafilter on $\kappa$ such that $A\in U$ and $U$ contains all final segments of $\kappa$ in $M$. 
 Set $\gamma=[\id_\kappa]_U$. Then $\gamma\IN_N j_U(A)$ holds by \L{}os' theorem and $\gamma$ witnesses that $j_U$ jumps at $\kappa$. 
   In the other direction, if $\map{j}{M}{\langle N,\IN_N\rangle}$ is an elementary embedding, $\gamma\in N$ witnesses that $j$ jumps at $\kappa$ and $\gamma\IN_N j(A)$, then $A\in U_j^\gamma$ and Proposition \ref{proposition:correspondence1}.(1) shows that $U$ contains all final segments of $\kappa$ in $M$. 
 \end{proof}

Combining earlier observations with arguments from the proofs of Proposition \ref{proposition:correspondence2} and Lemma \ref{correspondence3} immediately yields the following correspondence, which will be of use later on:

\begin{corollary}\label{corollary:correspondence2}
  Given $A\subseteq\kappa$,  \anf{$U$ is $M$-normal with respect to $\subseteq$-decreasing sequences, $U$ contains all final segments of $\kappa$ in $M$, and $A\in U$} $\kappa$-corresponds to \anf{$A\in M$ and $j$ is a $\kappa$-embedding with $\kappa^N\IN_N j(A)$}. \qed
\end{corollary}

%HIER

We want to close this section with two $\kappa$-correspondences, which may seem somewhat trivial, but which will be useful to have available later on.

\begin{lemma}\label{wellfounded}
  \mbox{\anf{$U$ contains all final segments of $\kappa$ in $M$ and $\langle\Ult{M}{U},\IN_U\rangle$ is well-found\-ed}} $\kappa$-corresponds to {\anf{$j$ jumps at $\kappa$ and $\langle N,\epsilon_N\rangle$ is well-found\-ed}}.
\end{lemma}

\begin{proof}
  The forward direction is immediate from Proposition \ref{proposition:correspondence1}.(1). 
 On the other hand, assume that $\map{j}{M}{\langle N,\IN_N\rangle}$ is such that 
 %$\crit{j}=\kappa$ 
 $j$ jumps at $\kappa$ and $\langle N,\IN_N\rangle$ is well-founded. 
 Then the well-foundedness of $\IN_N$ directly implies that $j$ is a $\kappa$-embedding and Proposition \ref{proposition:correspondence1}.(1) shows that $U_j$ contains all final segments of $\kappa$ in $M$. 
 As in the standard setting, we can now define a map $k$ from $\Ult{M}{U_j}$ to $N$ that sends $[f]_{U_j}$ to $(j(f)(\kappa^N))^N$. Since this map satisfies $k\circ j_{U_j}=j$, the well-foundedness of $\langle N,\IN_N\rangle$ implies the well-foundedness of $\langle\Ult{M}{U_j},\IN_{U_j}\rangle$. 
 % Let $U=U_j$ be the canonical $M$-ultrafilter induced by $j$, defined by letting $A\in U\iff\kappa\in j(A)$. Now let $j_U\colon M\to\bar N$ be defined as usual. As in the standard setting, one can show that there exists an elementary embedding $k\colon\bar N\to N$ such that $j_U\circ k=j$: let $k([f])=j(f)(\kappa)$ -- we will leave the easy details of verifying this to the interested reader. But then $\bar N$ has to be well-founded, for $N$ is. This, together with , finishes our argument.
\end{proof}

\begin{lemma}\label{lemma:general}
  Given a first order $\IN$-formula $\varphi(v_0,v_1)$, \anf{$U$ is $M$-normal with respect to $\subseteq$-decreasing sequences, $U$ contains all final segments of $\kappa$ in $M$, and $\varphi(M,U)$ holds} $\kappa$-corresponds to \anf{$j$ is a $\kappa$-embedding with $\varphi(M,U_j)$}. 
\end{lemma}

\begin{proof}
  For the forward direction, we know by Proposition \ref{proposition:correspondence2} and its proof that $[\id]_U$ witnesses that $j_U$ is a $\kappa$-embedding. 
This directly implies that $U=U_{j_U}$ and hence $\varphi(M,U_{j_U})$ holds. 
 The backward direction is a direct consequence of Proposition \ref{proposition:correspondence2}. 
\end{proof}

%%%%%%%%%%%%%%%%%%%

\section{Inaccessible cardinals and the bounded ideal}\label{section:Inaccessibles}

In this section, we characterize inaccessible limits of certain types of ordinals through the existence of ${<}\kappa$-amenable filters for small models $M$. 
 We then use these characterizations to determine the  corresponding ideals, which turn out to be the bounded ideals on the corresponding cardinals. 
 The following direct consequence of Proposition \ref{proposition:CharaLessThanKappaPowerPreserving} will be crucial for these characterizations.

\begin{corollary}\label{corollary:inaccessible}
  Let $\kappa$ be an inaccessible cardinal and let $\map{j}{M}{\langle N,\IN_N\rangle}$ be an elementary embedding with $\crit{j}=\kappa$.  If $M$ is a $\Sigma_0$-correct model of $\ZFC^-$ with $M\models\anf{\textit{$\HH{\kappa}$ exists}}$, then $j$ is ${<}\kappa$-powerset preserving. In particular, if $M\prec\HH{\theta}$ for some regular $\theta>\kappa$, then $j$ is ${<}\kappa$-powerset preserving. 
\end{corollary}

\begin{proof}
 Given $\alpha\in M\cap\kappa$, we have $\POT{\alpha}^M\subseteq\HH{\kappa}^M\in M$, $\POT{\alpha}^M\in M$ and $(2^\alpha)^M<2^\alpha<\kappa$. Hence Proposition \ref{proposition:CharaLessThanKappaPowerPreserving} implies that $j$ is ${<}\kappa$-powerset preserving. 
 Finally, if $M\prec\HH{\theta}$ for some regular $\theta>\kappa$, then the inaccessibility of $\kappa$ implies that $\HH{\kappa}^M=\HH{\kappa}\in M\cap\HH{\theta}$. 
 %
%  Let $\alpha\in M\cap\kappa$ and pick $y\in N$ with $\langle N,\IN_N\rangle\models\anf{y\subseteq j(\alpha)}$. 
   %
%  Since $\kappa$ is inaccessible in $M$, there is $\delta<\kappa$ in $M$ and a bijection $\map{b}{\delta}{\POT{\alpha}}$ that is an element of $M$. 
  %
% By elementarity, $j(b)$ is a bijection between $j(\delta)$ and the powerset of $j(\alpha)$ in $N$. 
  %
%  Then there is $\gamma\in N$ with $\gamma\IN_N j(\delta)$ and $\langle N,\IN_N\rangle\models\anf{y=j(b)(\gamma)}$. 
 %
% Since $\alpha<\crit{j}$, there is $\beta<\alpha$ with $j(\beta)=\gamma$ and elementarity implies that $y=j(b(\beta))$.  
\end{proof}

 We will also make use of a classical characterization of inaccessible cardinals. 
 Following {\cite[Definition 2.2]{MR3800756}}, our formulation of this result uses slightly generalized notions of filters on arbitrary collections of subsets of $\kappa$. 
 It is easy to see that these notions correspond well with our already defined notions of $M$-ultrafilters.\footnote{That is, if $M$ is a model of $\ZFC^-$ with $\kappa\in M$, an $M$-ultrafilter $U$ on $\kappa$ is uniform (respectively, uniform and ${<}\kappa$-complete for $M$) just in case $U$ is uniform (respectively, uniform and ${<}\kappa$-complete) in the sense of Definition \ref{definition:filter}.
 % We leave this as an easy exercise for the interested reader to check.
}

\begin{definition}\label{definition:filter}
\begin{enumerate} 
  \item A \emph{uniform filter} on $\kappa$ is a subset $F$ of $\POT{\kappa}$ such that $\betrag{\bigcap_{i<n}A_i}=\kappa$ whenever $n<\omega$ and $\seq{A_i}{i<n}$ is a sequence of elements of $F$. 

  \item A uniform filter $F$ on $\kappa$ \emph{measures} a subset $A$ of $\kappa$ if $A\in F$ or $\kappa\setminus A\in F$ and it \emph{measures} a subset $X$ of $\POT{\kappa}$ if it measures every element of $X$. % and  $F$ is an \emph{ultrafilter} on $\kappa$ if it measures $\POT{\kappa}$. 

\item A uniform filter $F$ on $\kappa$ is \emph{${<}\kappa$-complete} if $\betrag{\bigcap_{i<\gamma}A_i}=\kappa$  
 %(maybe just $\neq\emptyset$)
holds for every $\gamma<\kappa$ and every sequence $\seq{A_i}{i<\gamma}$ of elements of $F$. 
%\item
%If $\kappa$ is regular, a filter $F$ on $\kappa$ is \emph{normal} if for every sequence $\vec{X}=\langle X_\alpha\mid \alpha<\kappa\rangle$ of elements of $F$, the diagonal intersection $\largetriangleup \vec{X}$ is a stationary subset of $\kappa$.
%\item
%If $M$ is a weak $\kappa$-model, then a filter $F$ on $\kappa$ is \emph{$M$-normal} if it measures $\mathcal P(\kappa)\cap M$ and $\largetriangleup \vec{X}\in F$ whenever $\vec X=\langle X_\alpha\mid\alpha<\kappa\rangle\in M$ is a sequence of elements of $F$.
\end{enumerate} 
\end{definition}

\begin{lemma}[{\cite[Corollary 1.1.2]{MR0460120}}]\label{smallfilterextension}
 An uncountable cardinal $\kappa$ is inaccessible if and only if it has the \emph{${<}\kappa$-filter extension property}, i.e. whenever $F$ is a uniform, ${<}\kappa$-complete filter on $\kappa$, of size less than $\kappa$, and $X$ is a collection of subsets of $\kappa$ with $X$ of size less than $\kappa$, then there exists a uniform, ${<}\kappa$-complete filter $F'\supseteq F$ that measures $X$.
\end{lemma}

We are  now ready to state the key proposition of this section.  These results will directly yield the statements of  Theorems \ref{theorem:Ideals}.(\ref{item:Ideals:ia})  and \ref{theorem:IdealContain}.(\ref{item:cont:ia}), because the accessible cardinals are unbounded in every inaccessible cardinal and every unbounded subset of such a cardinal can be split into two disjoint unbounded subsets.

\begin{proposition}\label{proposition:inaccessiblelimit}
  Let $\kappa$ be an inaccessible cardinal. 
  \begin{enumerate} 
    \item If $A$ is an unbounded subset of $\kappa$, $\theta>\kappa$ is a regular cardinal and  $M\prec\HH{\theta}$ with $\betrag{M}<\kappa$ and $A\in M$, then there is a uniform $M$-ultrafilter $U$ on $\kappa$ with $A\in U$ and $\Psi_{ia}(M,U)$. 

  \item The set $\II_{ia}^{{<}\kappa}$ is the ideal of bounded subsets of $\kappa$. 
   % Let $I$ be the collection of all subsets of $\kappa$ that do not belong to $U$, for any $M\prec H(\theta)$ of size less than $\kappa$ with $\kappa\in M$, and any uniform, ${<}\kappa$-complete, ${<}\kappa$-amenable $M$-ultrafilter $U$ on $\kappa$. Then, 
 \end{enumerate}
\end{proposition}

\begin{proof}
 (1) Assume that $A\subseteq\kappa$ has size $\kappa$. 
 Then, $\{A\}$ is a uniform, ${<}\kappa$-complete filter on $\kappa$, of size less than $\kappa$. 
 If $\theta>\kappa$ is a regular cardinal and $M\prec \HH{\theta}$ of size less than $\kappa$ with $\kappa\in M$ and $A \in M$, then we can apply Lemma \ref{smallfilterextension} to extend $\{A\}$ to a uniform, ${<}\kappa$-complete filter $U$ on $\kappa$ that measures $M\cap\POT{\kappa}$. 
 As mentioned above, this exactly means for $U$ to be a uniform, ${<}\kappa$-complete $M$-ultrafilter on $\kappa$. 
 Let $\map{j_U}{M}{\langle\Ult{M}{U},\IN_U\rangle}$  be the induced ultrapower embedding. 
 By Proposition \ref{proposition:correspondence1}.(3), we know that $\crit{j_U}=\kappa$. 
 Moreover, Lemma \ref{correspondence3} shows that there is $\gamma\in \Ult{M}{U}$ with $\gamma\IN_U j_U(A)$ that witnesses that $j_U$ jumps at $\kappa$. 
 Since $\kappa$ is inaccessible, we can apply Lemma \ref{corollary:inaccessible} to see that $j_U$ is ${<}\kappa$-powerset preserving. 
 In this situation, we can use Proposition \ref{proposition:correspondence1}.(4) and Lemma \ref{correspondence3} to conclude that $U^\gamma_{j_U}$ is a uniform $M$-ultrafilter on $\kappa$  that is ${<}\kappa$-complete and ${<}\kappa$-amenable for $M$ and contains the subset $A$.  
  
 (2) If $\theta>\kappa$ is a regular cardinal, $M\prec \HH{\theta}$ with $\betrag{M}<\kappa$ and $U$ is a uniform $M$-ultrafilter on $\kappa$, then the uniformity of $U$ implies that every element of $U$ has size $\kappa$. By (1), this yields the desired equality. 
\end{proof}

 The following lemma provides us with the reverse direction for our desired characterization.
 Moreover, it shows that for the listed types of domain models $M$, $\kappa$-embeddings $\map{j}{M}{\langle N,\IN_N\rangle}$ with $\crit{j}=\kappa$ induce uniform $M$-ultrafilters $U_j$ on $\kappa$.

\begin{lemma}\label{lemma:inaccessible2}
  Assume that either
  \begin{itemize}
    \item for some regular cardinal $\theta>\kappa$ and some $M\prec\HH{\theta}$, or
    \item for many transitive weak $\kappa$-models $M$,\footnote{We would like to thank Victoria Gitman for pointing out a mistake in our proof with respect to the second item of this lemma in an earlier version of our paper.}
  \end{itemize}
  there exists an elementary embedding $\map{j}{M}{\langle N,\IN_N\rangle}$ with $\crit{j}=\kappa$. Then, $\kappa$ is regular. Moreover, if we can also assume that the given embeddings are ${<}\kappa$-powerset preserving, then $\kappa$ is inaccessible. 
\end{lemma}

\begin{proof}
  Assume that $\kappa$ is singular.  
  By our assumptions, we can find a cofinal function $\map{c}{\alpha}{\kappa}$ with $\alpha<\kappa$ and an elementary embedding $\map{j}{M}{\langle N,\IN_N\rangle}$ with $\crit{j}=\kappa$ and the property that $M$ is a $\Sigma_0$-correct $\ZFC^-$-model  with $c\in M$. 
  Let $\gamma\in N$ witness that $j$ jumps at $\kappa$. By elementarity, there is $\delta\in N$ with $\delta\IN_N j(\alpha)$ and $\gamma\IN_N(j(c)(\delta))^N$. Since $\alpha\in M\cap\kappa$, there is $\beta<\alpha$ with $j(\beta)=\delta$ and hence $\gamma\IN_N(j(c)(\delta))^N=j(c(\beta))\IN_N \gamma$, a contradiction. 
  
  Now, assume that we can also find  ${<}\kappa$-powerset preserving embeddings for the desired models and assume, towards a contradiction, that $2^\alpha\geq\kappa$ holds for some $\alpha<\kappa$. 
  Then our assumption yields an injection $\map{\iota}{\kappa}{\POT{\alpha}}$ and a ${<}\kappa$-powerset preserving embedding $\map{j}{M}{\langle N,\IN_N\rangle}$ with $\crit{j}=\kappa$ and the property that $M$ is a $\Sigma_0$-correct $\ZFC^-$-model  with $\iota\in M$. The existence of such a model directly contradicts Proposition \ref{proposition:CharaLessThanKappaPowerPreserving}.(2). 
 %
%  Pick a surjection $\map{s}{\POT{\alpha}}{\kappa}$ in $M$ and let $\gamma\in N$ witness that $j$ jumps at $\kappa$. By elementarity, there is $y\in N$ with $\langle N,\IN_N\rangle\models\anf{y\subseteq j(\alpha)\wedge j(s)(y)=\gamma}$. In this situation, ${<}\kappa$-powerset preservation yields some $x\in M$ with $j(x)=y$. Then, $x\subseteq\alpha$, and $\gamma=j(s)(y)=j(s(x))\IN_N\gamma$, a contradiction.  
\end{proof}

%\begin{corollary}\label{corollary:inaccessible2}
% If $\kappa$ is an uncountable cardinal such that for many  transitive weak $\kappa$-models $M$ there exists a ${<}\kappa$-powerset preserving elementary embedding $\map{j}{M}{\langle N,\IN_N\rangle}$ with $\crit{j}=\kappa$, then $\kappa$ is inaccessible.  
%\end{corollary}
%
%\begin{proof}
% Fix a bijection $b$ between $\kappa$ and some $M_0\prec\HH{\kappa^+}$ and define $A=\Set{\goedel{\alpha}{\beta}}{b(\alpha)\in b(\beta)}\subseteq\kappa$, %  %
% Pick a transitive weak $\kappa$-models $M$ with $A\in M$ and a ${<}\kappa$-powerset preserving elementary embedding $\map{j}{M}{\langle N,\IN_N\rangle}$ with $\crit{j}=\kappa$. Then $M_0\in M$ and elementarity implies that $j\restriction M_0$ is a ${<}\kappa$-powerset preserving elementary embedding with $\crit{j}=\kappa$. Hence Lemma \ref{lemma:inaccessible2} implies that $\kappa$ is inaccessible.  
%\end{proof}

The next result is now an immediate consequence of what has been shown above, and in particular implies Theorem \ref{theorem:SchemesSummary}.(\ref{item:ina}).

\begin{theorem}\label{theorem:inaccessiblelimit}
  Let $\kappa$ be an uncountable cardinal, let $\delta<\kappa$,  let $\varphi(v_0,v_1)$ be a first order $\epsilon$-formula and let $\theta>\kappa$ be a regular cardinal  such that the statement $\varphi(\alpha,\delta)$ is absolute between $\HH{\theta}$ and $\VV$ for all $\alpha<\kappa$. 
 Then, the following statements are equivalent: 
  \begin{enumerate} 
    \item The cardinal $\kappa$ is an inaccessible limit of ordinals $\alpha$ such that the property $\varphi(\alpha,\delta)$ holds. 

    \item For any (equivalently, for some) $M\prec \HH{\theta}$ of size less than $\kappa$ with $\kappa\in M$, there exists a uniform $M$-ultrafilter $U$ on $\kappa$ with $\Psi_{ia}(M,U)$ and $\Set{\alpha<\kappa}{\varphi(\alpha,\delta)}\in U$. 

    \item For any (equivalently, for some) $M\prec \HH{\theta}$ of size less than $\kappa$ with $\kappa\in M$, there exists a ${<}\kappa$-powerset preserving elementary embedding $\map{j}{M}{\langle N,\IN_N\rangle}$ with $\crit{j}=\kappa$ such that $\langle N,\IN_N\rangle\models\varphi(\gamma,j(\delta))$ for some $\gamma\in N$ witnessing that $j$ jumps at $\kappa$. 
  \end{enumerate}
\end{theorem}

\begin{proof}
  The implication from (1) to (2) is immediate from Proposition \ref{proposition:inaccessiblelimit}. 
  Proposition \ref{proposition:correspondence1}.(4) and Lemma \ref{correspondence3} show that both the universal and the existential statement in (2) are equivalent to the respective statements in (3). 
 Finally, Lemma \ref{lemma:inaccessible2} shows that the existential statement in (3) implies that $\kappa$ is inaccessible and hence the existential statement in (2) allows us to use Proposition \ref{proposition:inaccessiblelimit} to derive (1). 
\end{proof}

We will next obtain a further characterization of inaccessible cardinals, in which we may require stronger properties of the ultrafilters and elementary embeddings used. 
 For this, we need two standard lemmas. %, which for example is an easy consequence of the more general results \cite[Corollary 2.2 (i) and Lemma 3.4]{hln}.

\begin{lemma}\label{smallembedding}
  If $\kappa<\theta$ are uncountable regular cardinals, $S\subseteq\kappa$ is stationary and $x\in\HH{\theta}$, then there is a transitive set $M$ of cardinality less than $\kappa$ and an elementary embedding $\map{j}{M}{\HH{\theta}}$ with $\crit{j}\in S$, $j(\crit{j})=\kappa$ and $x\in\ran{j}$.  
\end{lemma}

\begin{proof}
 Let $\seq{N_\alpha}{\alpha<\kappa}$ be a continuous and increasing sequence of elementary substructures of $\HH{\theta}$ of cardinality less than $\kappa$ with $x\in N_0$ and $\alpha\subseteq N_\alpha\cap\kappa\in\kappa$ for all $\alpha<\kappa$. Since the set $\Set{N_\alpha\cap\kappa}{\alpha<\kappa}$ is closed unbounded in $\kappa$, there is $\alpha<\kappa$ with $N_\alpha\cap\kappa\in S$. Set $\map{\pi}{N_\alpha}{M}$ denote the corresponding transitive collapse. Then $\map{\pi^{{-}1}}{M}{\HH{\theta}}$ is an elementary embedding satisfying  $\crit{\pi^{{-}1}}=N_\alpha\cap\kappa\in S$, $\pi^{{-}1}(\crit{\pi^{{-}1}})=\pi^{{-}1}(N_\alpha\cap\kappa)=\kappa$ and $x\in N_\alpha=\ran{\pi^{{-}1}}$. 
\end{proof}

\begin{lemma}\label{lemma:regular}
  If $\kappa<\theta$ are regular uncountable cardinals, $S$ is a stationary subset of $\kappa$ and $M\prec\HH{\theta}$ with  $\betrag{M}<\kappa$ and $S\in M$, then there exists a normal $M$-ultrafilter $U$ on $\kappa$ with $S\in U$. 
\end{lemma}

\begin{proof}
  By Lemma \ref{smallembedding}, there exists a transitive set $X$ of cardinality less than $\kappa$ and an elementary embedding $\map{j}{X}{\HH{\theta}}$ with  $j(\crit{j})=\kappa$, $\crit{j}\in S$ and $M\in\ran{j}$. 
 Then, we have  $j[X]\prec\HH{\theta}$ and $M\subseteq j[X]$, because elementarity implies that $\betrag{M}\in j[X]\cap\kappa$ and hence $\betrag{M}<\crit{j}\subseteq j[X]$.  
 Define $U=\Set{A\in j[X]\cap\POT{\kappa}}{\crit{j}\in A}$. 
 Then $U$ is a $j[X]$-ultrafilter on $\kappa$ and, since $j[X]\cap\kappa=\crit{j}$ and $j[X]\prec\HH{\theta}$,  it follows that $U$ is $j[X]$-normal and ${<}\kappa$-complete in $\VV$,  and all elements of $U$ are stationary subsets of $\kappa$. 
 Since $\betrag{M}<\kappa$, it follows that $M\cap U$ is a normal $M$-ultrafilter on $\kappa$ and $\crit{j}\in S\in M$ implies that $S\in U$. 
\end{proof}

 Note that, in combination with Proposition \ref{proposition:correspondence1}.(3) and Lemma \ref{lemma:inaccessible2}, the previous lemma directly implies Theorem \ref{theorem:SchemesSummary}.(\ref{item:rega}).
  Moreover, Lemma \ref{lemma:regular} also yields the following results, which in particular implies Theorem \ref{theorem:SchemesSummary}.(\ref{item:inb}).

\begin{theorem}\label{theorem:inaccessible}
 The following statements are equivalent for every uncountable cardinal $\kappa$ and all regular cardinals $\theta>\kappa$:  
  \begin{enumerate} 
    \item The cardinal $\kappa$ is inaccessible. 

    \item For any (equivalently, for some) $M\prec \HH{\theta}$ of size less than $\kappa$ with $\kappa\in M$, there exists a ${<}\kappa$-amenable, normal $M$-ultrafilter $U$ on $\kappa$. 

    \item For any (equivalently, for some) $M\prec \HH{\theta}$ of size less than $\kappa$ with $\kappa\in M$, there exists a ${<}\kappa$-powerset preserving elementary embedding $\map{j}{M}{\langle N,\IN_N\rangle}$ with $\crit{j}=\kappa$, such that $j$ induces a normal $M$-ultrafilter on $\kappa$. 
  \end{enumerate}
\end{theorem}

\begin{proof}
  First, assume that  $\kappa$ is inaccessible and $M\prec\HH{\theta}$  with $\betrag{M}<\kappa$ and $\kappa\in M$. Then Lemma \ref{lemma:regular} yields a normal $M$-ultrafilter $U$ on $\kappa$. By Proposition \ref{proposition:correspondence2}, we know that $j_U$ is a $\kappa$-embedding and the proof of this proposition shows that $U=U_j$. In particular,  a combination of 
  Proposition \ref{proposition:correspondence1}.(4) and Corollary \ref{corollary:inaccessible} now shows that $U$ is ${<}\kappa$-amenable.  
 Next, a combination of Proposition \ref{proposition:correspondence1}.(3)+(4), Proposition \ref{proposition:correspondence2},  Lemma \ref{wellfounded},  
  Lemma \ref{lemma:general} for the statement $\varphi(M,U)\equiv\anf{\textit{$U$ is normal}}$ and arguments from the first implication shows that, if $M\prec \HH{\theta}$ with $\betrag{M}<\kappa$ and $\kappa\in M$, then every ${<}\kappa$-amenable, normal $M$-ultrafilter $U$ on $\kappa$ induces a $\kappa$-embedding $j_U$ with $\crit{j_U}=\kappa$ that is ${<}\kappa$-powerset preserving and induces a normal $M$-ultrafilter.  
 This shows that both the universal and the existential statement in (2) imply the respective statements in (3). 
 The equivalence between the corresponding statements in (2) and (3) then follows from Proposition \ref{proposition:correspondence1}.(4).
 %, Proposition \ref{proposition:correspondence2}, and Lemma \ref{lemma:general}, using the same formula $\varphi$ as above. 
 %
 Finally, Theorem \ref{theorem:inaccessiblelimit} directly shows that the existential statements in (2) and (3) both imply (1). 
\end{proof}

%%%%%%%%%%%%%%%%%%%%%%%%%%%%

\section{Regular stationary limits and the non-stationary ideal}\label{section:stationarylimits}

In this section, we characterize Mahlo-like cardinals, that is regular stationary limits of certain ordinals,\footnote{Note that in particular, regular, inaccessible and Mahlo cardinals are Mahlo-like.} through the existence of $M$-normal filters for small models $M$. 
 We then use these characterizations to define the  corresponding ideals, which turn out to be the non-stationary ideal below the considered set of ordinals. 
 We start by proving the corresponding statement of Theorem \ref{theorem:Ideals} with the help of Lemma \ref{lemma:regular}.\footnote{Note that the forward direction of this proof is quite different to that of the seemingly similar Proposition \ref{proposition:inaccessiblelimit}.} 
% The next proposition is an easy consequence of Lemma \ref{lemma:regular}, and in particular implies Theorem \ref{theorem:Ideals}.(\ref{item:Ideals:reg}).

\begin{proof}[Proof of Theorem \ref{theorem:Ideals}.(\ref{item:Ideals:reg})] 
 Fix a regular and uncountable cardinal $\kappa$. 
  First, let $A$ be a stationary subset of $\kappa$, let $\theta>\kappa$ be a regular cardinal and  let $M\prec\HH{\theta}$ with $\betrag{M}<\kappa$ and $A\in M$. 
  In this situation, we can use Lemma \ref{lemma:regular} to find a normal $M$-ultrafilter $U$ on $\kappa$ with $A\in U$. Then $U$ is uniform with $\Psi_\delta(M,U)$ and hence $U$ witnesses that $A\notin \II_\delta^{{<}\kappa}$. 
 In the other direction, let $A$ be a non-stationary subset of $\kappa$, let $M\prec\HH{\theta}$ with $\betrag{M}<\kappa$ and $A\in M$, and let $U$ be a uniform, $M$-normal $M$-ultrafilter on $\kappa$. 
 By elementarity, we find a club subset $C$ of $\kappa$ in $M$ that is disjoint from $A$. 
 By the $M$-normality and uniformity of $U$, every club subset of $\kappa$ in $M$ is contained in $U$ and this shows that $A\notin U$. We can conclude that $A\in\II_\delta^{{<}\kappa}$. 
\end{proof}

The next result is an immediate consequence of Theorem \ref{theorem:Ideals}.(\ref{item:Ideals:reg}),  and in particular implies Theorem   \ref{theorem:SchemesSummary}.(\ref{item:regb}).

\begin{theorem}\label{theorem:mahlolike}
  Let $\kappa$ be an uncountable cardinal, let $\delta<\kappa$, let $\varphi(v_0,v_1)$ be a first order $\IN$-formula and let $\theta>\kappa$ be a regular cardinal such that the statement $\varphi(\alpha,\delta)$ is absolute between $\HH{\theta}$ and $\VV$ for all $\alpha<\kappa$. 
  Then, the following statements are equivalent: 
  \begin{enumerate} 
    \item $\kappa$ is a regular stationary limit of ordinals $\alpha$ satisfying $\varphi(\alpha,\delta)$. 

    \item For any (equivalently, for some) $M\prec \HH{\theta}$ of size less than $\kappa$ with $\kappa\in M$, there exists a uniform, $M$-normal $M$-ultrafilter $U$ on $\kappa$ with $\Set{\alpha<\kappa}{\varphi(\alpha,\delta)}\in U$. 
%We may additionally require that $U$ is normal. 

  \item Same as (2), but we also require $U$ to be normal. 

    \item For any  (equivalently, for some) $M\prec \HH{\theta}$ of size less than $\kappa$ with $\kappa\in M$, there exists a $\kappa$-embedding $\map{j}{M}{\langle N,\IN_N\rangle}$ with $\crit{j}=\kappa$ and $\langle N,\epsilon_N\rangle\models\varphi(\kappa^N,j(\delta))$. 
 % We may additionally require that $\epsilon_N$ is well-founded, or even that $j$ induces a normal ultrafilter.

    \item Same as (4), but we also require $\epsilon_N$ to be well-founded. 

        \item Same as (4), but we also require $j$ to induce a normal ultrafilter. 
  \end{enumerate}
\end{theorem}

\begin{proof}
  The implication from (1) to (3) is given by Lemma \ref{lemma:regular}. 
 %Regarding the extra normality assumption, note that we could have added the assumption of normality of $U$ in Proposition \ref{proposition:mahlolike} throughout. 
 %
 The combination of Proposition \ref{proposition:correspondence1}.(3), Corollary \ref{corollary:correspondence2}, Lemma \ref{wellfounded} and Lemma \ref{lemma:general} shows that both the universal and the existential statement in (2) are equivalent to the corresponding statement in (4). 
The same is true for the corresponding statements in (3) and (6). 
Since ultrapowers of models $M$ of size less than $\kappa$ formed using normal $M$-ultrafilters are obviously well-founded, the results listed above also show that the universal statement in (3) implies the universal statement in (5) and the existential statement in (5) trivially implies the corresponding statement in (4). 
 Finally, since Lemma \ref{lemma:inaccessible2} shows that the existential statement in (4) implies that $\kappa$ is regular, we can use Theorem \ref{theorem:Ideals}.(\ref{item:Ideals:reg})  and the implications derived above to conclude that the existential statement in (2) implies (1). 
 % All remaining implications are immediate. 
\end{proof}

%%%%%%%%%%%%%%%%%%%%%%%%%%%%

\section{Weakly compact cardinals and $\kappa$-amenability}\label{section:WCamenablecomplete}

In this section, we extend Kunen's results from \cite{MR2617841} and we characterize weakly compact cardinals $\kappa$ through the existence of $\kappa$-amenable ultrafilters for models of size at most $\kappa$. 
 For this, we need a classical result on weakly compact cardinals, which we present using the notions introduced in Definition \ref{definition:filter}.

\begin{lemma}[{\cite[Corollary 1.1.4]{MR0460120}}, see also {\cite[Proposition 2.9]{MR3800756}}]\label{filterextension}
 An uncountable cardinal $\kappa$ is weakly compact if and only if it has the \emph{filter extension property},  i.e. whenever $F$ is a uniform ${<}\kappa$-complete filter on $\kappa$ of size at most $\kappa$, and $X$ is a collection of subsets of $\kappa$ with $X$ of size at most $\kappa$, then there exists a uniform ${<}\kappa$-complete filter $F'\supseteq F$ that measures $X$. 
\end{lemma}

\begin{corollary}
 If $\kappa$ is weakly compact, then $\II^\kappa_{ia}=\II^\kappa_{\prec ia}=\II^{{<}\kappa}_{ia}$ is the bounded ideal on $\kappa$. 
\end{corollary}

\begin{proof}
 Let $A$ be unbounded in $\kappa$ and let $M$ be a weak $\kappa$-model with $A\in M$. Then $F=\Set{A\cap[\alpha,\kappa)}{\alpha<\kappa}$ is a uniform ${<}\kappa$-complete filter on $\kappa$ of size $\kappa$. 
 Using Lemma \ref{filterextension}, we find a uniform ${<}\kappa$-complete filter $U\supseteq F$ that measures $\POT{\kappa}\cap M$. 
  Then Proposition \ref{proposition:correspondence1}.(3) and Corollary  \ref{corollary:inaccessible} show that $\Psi_{ia}(M,U)$ holds. By uniformity, these computations show that both $\II^\kappa_{ia}$ and $\II^\kappa_{\prec ia}$ are the bounded ideal on $\kappa$ and, by Theorem \ref{theorem:Ideals}.(\ref{item:Ideals:ia}), this also shows that they are equal to $\II^{<\kappa}_{ia}$. 
\end{proof}

 The following lemma will allow us to characterize weak compactness through the existence of $\kappa$-amenable, ${<}\kappa$-complete ultrafilters.

\begin{lemma}\label{lemma:WCmodelsBoundedIdeal}
 If $\kappa$ is a weakly compact cardinal, $\lambda\leq\kappa$ is a cardinal, $\theta>\kappa$ is a regular cardinal, $A$ is an unbounded subset of $\kappa$ and $x\in\HH{\theta}$, then there is $(\kappa,\lambda)$-model $M\prec\HH{\theta}$ with $x\in M$ and a uniform $M$-ultrafilter $U$ on $\kappa$ with $A\in U$ and $\Psi_{wc}(M,U)$. 
\end{lemma}

\begin{proof}
 We recursively construct $\omega$-sequences $\seq{M_n}{n<\omega}$ of weak $\kappa$-models $M_n\prec\HH{\theta}$, and $\seq{U_n}{n<\omega}$ of $M_n$-ultrafilters on $\kappa$. 
 Pick a weak $\kappa$-model $M_0\prec\HH{\theta}$ with $x\in M_0$, and, using Lemma \ref{filterextension}, let $U_0$ be a uniform ${<}\kappa$-complete $M_0$-ultrafilter on $\kappa$.
 Now, assume that $M_n$ and $U_n$ are already constructed, let $M_{n+1}\prec\HH{\theta}$ be a weak $\kappa$-model  with $M_n,U_n\in M_{n+1}$, and, using Lemma \ref{filterextension}, let $U_{n+1}$ be a uniform ${<}\kappa$-complete $M_{n+1}$-ultrafilter extending $U_n$. 
 Set $M=\bigcup_{n<\omega}M_n$, and let $U=\bigcup_{n<\omega}U_n$. Then, $U$ is a uniform $M$-ultrafilter that is ${<}\kappa$-complete for $M\prec\HH{\theta}$. 
 If $\vec{x}=\seq{x_\alpha}{\alpha<\kappa}$ is a sequence of subsets of $\kappa$ in $M$, then $\vec{x}\in M_n$ for some $n<\omega$. Hence, each $x_\alpha$ is measured by $U_n\subseteq U$, and thus, by our choice of $M_{n+1}$, we know that $\Set{\alpha<\kappa}{ x_\alpha\in U}\in M_{n+1}\subseteq M$, showing that $U$ is $\kappa$-amenable for $M$ and therefore proving the lemma for $\lambda=\kappa$. Given $\lambda<\kappa$, we simply take a $(\lambda,\kappa)$-model $\langle\bar M,\bar U\rangle\prec\langle M,U\rangle$ with $x\in\bar M$. Then, by elementarity, $\langle\bar M,\bar U\rangle$ has the desired properties. 
\end{proof}

\begin{corollary}\label{corollary:WCmodelsBoundedIdeal}
 If $\kappa$ is weakly compact, then $\II^\kappa_{wc}=\II^\kappa_{\prec wc}=\II^{{<}\kappa}_{wc}$ is the bounded ideal on $\kappa$.
\end{corollary}

\begin{proof}
 By uniformity, Lemma \ref{lemma:WCmodelsBoundedIdeal} implies that both $\II^\kappa_{{\prec}wc}$ and $\II^{{<}\kappa}_{wc}$ are the bounded ideal on $\kappa$. Moreover, by choosing $\kappa=\lambda$ and $\theta=\kappa^+$ in Lemma \ref{lemma:WCmodelsBoundedIdeal}, we can conclude that $\II^\kappa_{\prec wc}$ is also equal to this ideal. 
\end{proof}

 Corollary \ref{corollary:WCmodelsBoundedIdeal} suggests that the ideals $\II^\kappa_{wc}$, $\II^\kappa_{{\prec}wc}$ and $\II^{{<}\kappa}_{wc}$ are not canonically connected to weak  compactness. 
  We will present such an ideal  in Section \ref{section:weaklycompactideal}.  
 The next result directly implies Theorem \ref{theorem:SchemesSummary}.(\ref{item:wca}).

\begin{theorem}\label{theorem:weaklycompact}
 The following statements are equivalent for every uncountable cardinal $\kappa$, every cardinal $\lambda\leq\kappa$ and every regular cardinal $\theta>\kappa$: 
  \begin{enumerate} 
    \item $\kappa$ is weakly compact. 

    \item For many $(\lambda,\kappa)$-models (equivalently, for some $(\lambda,\kappa)$-model) $M\prec\HH{\theta}$, there exists a uniform $M$-ultrafilter $U$ on $\kappa$ with $\Psi_{wc}(M,U)$.  

    \item For many $(\lambda,\kappa)$-models (equivalently, for some $(\lambda,\kappa)$-model) $M\prec\HH{\theta}$, there exists a $\kappa$-powerset preserving elementary embedding $\map{j}{M}{\langle N,\IN_N\rangle}$ with $\crit{j}=\kappa$. 

    \item For many transitive weak $\kappa$-models $M$, there exists a uniform $M$-ultrafilter $U$ on $\kappa$ with $\Psi_{wc}(M,U)$. 

    \item For many transitive weak $\kappa$-models $M$, there exists a $\kappa$-powerset preserving  elementary embedding $\map{j}{M}{\langle N,\IN_N\rangle}$ with $\crit{j}=\kappa$. 
  \end{enumerate}
\end{theorem}

\begin{proof}
  The implication from (1) to (2)  and (4) follows from Lemma \ref{lemma:WCmodelsBoundedIdeal}.
 Using Proposition \ref{proposition:correspondence1}.(5), we can see that both statements in (2) are equivalent to the respective statements in (3) and both statements in (4) are equivalent to the respective statements in (5). 
 Now, assume that (1) fails and $\map{j}{M}{\langle N,\IN_N\rangle}$ witnesses that the existential statement in (3) holds. 
 %By picking a suitable elementary substructure,  we may assume that $\lambda<\kappa$. 
 %
 Then, Lemma  \ref{lemma:inaccessible2} implies that  $\kappa$ is inaccessible and our assumption implies  that there exists a $\kappa$-Aronszajn tree. 
 By elementarity, there is such a tree $T$ in $M$ with underlying set $\kappa$. 
 Then, $j(T)$ is a $j(\kappa)$-Aronszajn tree with underlying set $j(\kappa)$ in $N$. 
 Pick $\gamma\in N$ witnessing that $\crit{j}=\kappa$,  let $\delta$ be a node of $j(T)$ on level $\gamma$ in $N$ and let $y\in N$ be the set of predecessors of $\delta$ in $j(T)$ in $N$. 
 By $\kappa$-powerset preservation, there is $x\in M$ with $M\cap x=\Set{\beta\in M\cap\kappa}{j(\beta)\IN_N y}$. 
 By elementarity and the fact that $\crit{j}=\kappa$, the set $x$ is linearly ordered and downwards-closed in $T$. 
 Since $T$ is a $\kappa$-Aronszajn tree, there is some $\alpha\in M\cap\kappa$ such that $x$ does not intersect the $\alpha$-th level of $T$. Then there is $\varepsilon\in M\cap\kappa$ and a surjection $s$ from $\varepsilon$ onto the $\alpha$-th level of $T$ in $M$. 
 Since $\varepsilon<\crit{j}$, there we can find  $\xi\in M\cap\varepsilon$ with $j(s(\xi))\IN_N y$ and hence $\xi\in x$, a contradiction.  
The argument that (5) implies (1) again proceeds analogously by first using Lemma \ref{lemma:inaccessible2} 
 to show that $\kappa$ is inaccessible and then pick a transitive weak $\kappa$-model $M$ that contains a $\kappa$-Aronszajn tree $T$ as an element and is the domain of a $\kappa$-powerset preserving  elementary embedding with critical point $\kappa$. 
\end{proof}

In the above result, instead of using all $M\prec H(\theta)$, as in Kunen's result for countable models, and as in our earlier sections, we pass to a characterization using only \emph{many} models $M\prec H(\theta)$. The results of our later sections will show that this is in fact necessary, for if $M\prec H(\theta)$ were closed under countable sequences and satisfies (2) in Theorem \ref{theorem:weaklycompact}, then we would obtain that $U$ induces a well-founded ultrapower of $M$, which would imply that $\kappa$ is completely ineffable by Theorem \ref{theorem:completelyineffable}.

%%%%%%%%%%%%%%%%%%%%%%%%%%
%%%%%%%%%%%%%%%%%%%%%%%%%%

\section{Weakly compact cardinals without $\kappa$-amenability}\label{section:weaklycompactideal}

 In order to find a characterization of weak compactness that is connected to a canonical ideal, we now consider characterization using models of the same cardinality as the given cardinal. 
 We start by recalling the definition of the weakly compact ideal, which is due to L{\'e}vy.

\begin{definition}\label{definition:weaklycompactideal}
  Let $\kappa$ be a weakly compact cardinal. The \emph{weakly compact ideal on $\kappa$} consists of all $A\subseteq\kappa$ for which there exists a $\Pi^1_1$-formula $\varphi(v^1)$ and $Q\subseteq\VV_\kappa$ with $\VV_\kappa\models\varphi(Q)$ and $\VV_\alpha\models\neg\varphi(Q\cap\VV_\alpha)$ for all $\alpha\in A$. 
\end{definition}

 It is well-known that the weakly compact ideal is strictly larger than the non-stationary ideal whenever $\kappa$ is a weakly compact cardinal, because the weakly compact ideal contains the stationary set $\NN^\kappa_{ia}$ of smaller accessible cardinals in this case (see \cite[Theorem 2.8]{MR0540770}). 
 By a classical result of Levy (see \cite[Proposition 6.11]{MR1994835}), the weakly compact ideal is a normal ideal. 
 We now provide a characterization of the weakly compact ideal which resembles our earlier characterizations, and which in particular shows that $\II_{WC}^\kappa$ is the weakly compact ideal on $\kappa$, proving Theorem \ref{theorem:Ideals}.(\ref{item:Ideals:WC}).   
 This result is a variant of results of Baumgartner in \cite[Section 2]{MR0540770}. 
 %Note that an easy adaption of below argument also shows that $\II_{\prec\delta}^\kappa=\II_\delta^\kappa$ is the weakly compact ideal on $\kappa$. 
 %
 In the proof of the first item, we proceed somewhat similar to the argument for {\cite[Theorem 1.3]{MR1133077}}.

 In the following, whenever $M$ is a $\Sigma_0$-correct $\ZFC^-$-model, $\kappa$ is a cardinal of $M$ and $U$ is an $M$-ultrafilter on $\kappa$ that is $M$-normal with respect to $\subseteq$-decreasing sequences and contains all final segments of $\kappa$ in $M$, then we write $\kappa^U$ instead of $\kappa^{\Ult{M}{U}}$ (see Proposition \ref{proposition:correspondence2}).

\begin{theorem}\label{theorem:weaklycompactideal}
  Let $\kappa$ be a weakly compact cardinal. 
  \begin{enumerate} 
    \item If $A\subseteq\kappa$ is not contained in the weakly compact ideal, $\theta>\kappa$ is a regular cardinal and  $M\prec\HH{\theta}$ is a weak $\kappa$-model with $A\in M$, then there is a uniform $M$-ultrafilter $U$ on $\kappa$ with $A\in U$ and $\Psi_{WC}(M,U)$. 

  \item $\II^\kappa_\delta=\II^\kappa_{\prec\delta}=\II^\kappa_{WC}=\II^\kappa_{{\prec}WC}$ is the weakly compact ideal on $\kappa$. 

%  \item If $A\not\in\II^\kappa_\delta$, then, for any weak $\kappa$-model $M\prec H(\theta)$ with $A\in M$, there is a uniform, $M$-normal $M$-ultrafilter $U$ on $\kappa$ with $A\in U$.
 \end{enumerate}
\end{theorem} 

\begin{proof}
 (1) First, assume towards a contradiction, that there is no uniform, $M$-normal $M$-ultrafilter $U$ on $\kappa$ with $A\in U$. 
 Let $\map{\pi}{M}{X}$ denote the transitive collapse of $M$, pick a bijection $\map{b}{\kappa}{X}$ and define $$E ~ = ~ \Set{\langle\alpha,\beta\rangle\in\kappa\times\kappa}{b(\alpha)\in b(\beta)}.$$ 
Let $T$ be the elementary diagram of $\langle\kappa,E\rangle$, coded as a subset of $\VV_\kappa$ in a canonical way. 
 Now, let $\varphi(E,T)$ be a $\Pi^1_1$-statement expressing the conjunction of the following two statements over $\VV_\kappa$: 
\begin{enumerate} 
  \item[(i)] There is no $U\subseteq\kappa$ such that $\langle\kappa,E,U\rangle$ thinks that $U$ is a uniform, normal ultrafilter on $b^{{-1}}(\kappa)$ that contains $b^{{-}1}(A)$. 
%  \item $E$ is extensional. 

  \item[(ii)] $T$ is the elementary diagram of $\langle\kappa,E\rangle$. 
\end{enumerate}
 Then $\VV_\kappa\models\varphi(E,T)$ and, since $A$ is not contained in the weakly compact ideal on $\kappa$, we can find an inaccessible $\alpha\in A$ with $\alpha>b^{{-}1}(A)=(b^{{-}1}\circ \pi)(A)$ and $\VV_\alpha\models\varphi(E\cap\VV_\alpha,T\cap\VV_\alpha)$.  
%  $\varphi(E,T)$ must reflect to some inaccessible cardinal $\alpha\in A$. %$E\cap V_\alpha$ is extensional, and thus we may let $\langle M^*,\in\rangle$ be the Mostowski collapse of $\langle\alpha,E\cap V_\alpha\rangle$.
 Since (ii) is reflected to $\alpha$, the structure $\langle\alpha,E\cap\VV_\alpha\rangle$ is an elementary  substructure of $\langle\kappa,E\rangle$. 
 Set $M_*=(\pi^{{-}1}\circ b)[\alpha]$. Then $M_*\prec\HH{\theta}$ with $\betrag{M_*}<\kappa$ and $A\in M_*$. 
  Since $A$ is stationary in $\kappa$, Theorem \ref{theorem:Ideals}.(\ref{item:Ideals:reg}) yields a uniform, $M_*$-normal $M_*$-ultrafilter $U_*$ on $\kappa$ with $A\in U_*$. Set $\bar{U}=(b^{{-}1}\circ\pi)[U_*]\subseteq\kappa$. 
 Then, $\langle\alpha,E\cap\VV_\alpha,\bar{U}\rangle$  thinks that $\bar{U}$ is a uniform, normal ultrafilter on $b^{{-1}}(\kappa)$ that contains $b^{{-}1}(A)$, contradicting the fact that (i) reflects to $\alpha$.

 Now, pick uniform, $M$-normal $M$-ultrafilter $U$ on $\kappa$ with $A\in U$. By Proposition \ref{proposition:correspondence1}.(3) and Proposition \ref{proposition:correspondence2}, the map $j_U$ is a $\kappa$-embedding with $\crit{j_U}=\kappa$ and  Corollary \ref{corollary:inaccessible} implies that $j_U$ is ${<}\kappa$-powerset preserving.
 Since $U=U_{j_U}$, Proposition \ref{proposition:correspondence1}.(4) now shows that $U$ is ${<}\kappa$-amenable for $M$ and hence we can conclude $\Psi_{WC}(M,U)$ holds. 

 (2) By definition, we have $\II^\kappa_\delta\subseteq\II^\kappa_{WC}$ and $\II^\kappa_{{\prec}\delta}\subseteq\II^\kappa_{{\prec}WC}$. 
  Moreover,  (1) shows that $\II^\kappa_{{\prec}WC}$ is contained in the weakly compact ideal on $\kappa$. 
 Now, pick $A\in\POT{\kappa}\setminus\II^\kappa_{{\prec}\delta}$. 
 Then there is a regular cardinal $\theta>\kappa$, a weak $\kappa$-model $M\prec\HH{\theta}$, and a  uniform, $M$-normal $M$-ultrafilter $U$ on $\kappa$ with $A\in U$. 
  By Proposition \ref{proposition:correspondence2}, the induced ultrapower map $\map{j_U}{M}{\langle\Ult{M}{U},\IN_U\rangle}$ is a $\kappa$-embedding. 
 Then Lemma \ref{correspondence3} implies that $\kappa^U\IN_U j_U(A)$. 
 Now, let $\varphi(v_0^1,v_1^1)$ be a $\Sigma^1_0$-formula and assume that there is $Q\subseteq\VV_\kappa$ with $\VV_\kappa\models\forall^1 Z ~ \varphi(Q,Z)$ and $\VV_\alpha\models\exists^1 Z ~ \neg\varphi(Q\cap\VV_\alpha,Z)$ for all $\alpha\in A$. 
 Since $M\prec\HH{\theta}$, we may assume that $Q\in M$, and that the above statements hold in $M$. 
 In this situation, the elementarity of $j$ implies that there is $S\in\Ult{M}{U}$ such that $S\subseteq\VV_{\kappa^U}$ and $\VV_{\kappa^U}\models\neg\varphi(j(Q)\cap\VV_{\kappa^U},S)$ hold in $\Ult{M}{U}$. 
 Since $\kappa\subseteq M$ is inaccessible, we have $\VV_\kappa\subseteq M$ and the $M$-normality of $U$ implies that $j_U\restriction\VV_\kappa$ is an $\in$-isomorphism between $\langle\VV_\kappa,\in\rangle$ and $\langle\VV_{\kappa^U}^{\Ult{M}{U}},\IN_U\rangle$. Set $R=\Set{x\in\VV_\kappa}{j_U(x)\in S}$. 
 Then we have $$j[Q] ~ = ~ \Set{y\in\Ult{M}{U}}{y\IN_U(j(Q)\cap\VV_{\kappa^U})^{\Ult{M}{U}}}$$ and $j[R]=\Set{y\in\Ult{M}{U}}{y\IN_U S}$. This allows us to conclude that $\VV_\kappa\models\neg\varphi(Q,R)$, a contradiction. 
 This shows that $A$ is not contained in the weakly compact ideal on $\kappa$. 

 The above computations show that $\II^\kappa_{{\prec}\delta}$ is the weakly compact ideal on $\kappa$. 
 Finally, by choosing $\theta=\kappa^+$ in (1), we know that $\II^\kappa_{WC}$ is contained in the weakly compact ideal and we can show that these ideals are equal by proceeding as in the above argument, however picking a weak $\kappa$-model $M$ containing the set $Q$ as an element.
\end{proof}

Note that a small variation of the argument used in the last paragraph of the first part of the above proof shows that $\II^\kappa_\delta=\II^\kappa_{WC}$ and $\II^\kappa_{{\prec}\delta}=\II^\kappa_{{\prec}WC}$ holds for every inaccessible cardinal $\kappa$.  
As pointed out to us by the anonymous referee, it is possible to use an unpublished result of Hamkins (see \cite{HamkinsTalk}) to separate these ideals. 
 His result shows that, after adding $\kappa^+$-many Cohen reals 	to a model of set theory containing a weakly compact cardinal $\kappa$, the cardinal $\kappa$ still possesses the \emph{weakly compact embedding property}, i.e. for every subset $A$ of $\kappa$, there are transitive models $M$ and $N$ of $\ZFC^-$ with $A,\kappa\in M$ and an elementary embedding $\map{j}{M}{N}$ with $\crit{j}=\kappa$. 
 In particular, the results of Section \ref{section:ultrapowersandembeddings} show that $\II^\kappa_\delta$ is a proper ideal in this model. 
  Since the results contained in the remainder of this section show that weak compactness can be characterized through the property $\Psi_{WC}$ using Scheme \ref{schemeNoel}, it follows that $\kappa\in\II^\kappa_{WC}\supsetneq\II^\kappa_\delta$ holds in the constructed forcing extension.

 Hamkins' result can easily be generalized to show that, by adding $\kappa^+$-many Cohen subsets of $\omega_1$ for some weakly compact cardinal $\kappa$, we obtain a forcing extension with the property that for every subset $A$ of $\kappa$, there is an elementary embedding $\map{j}{M}{N}$ between transitive models of $\ZFC^-$ such that $\crit{j}=\kappa$, ${}^\omega M\cup\{A,\kappa\}\subseteq M$ and all stationary subsets of $\kappa$ in $M$ are stationary in $\VV$.  
 This shows that the existence of many transitive weak $\kappa$-models $M$ with the property that there exists a uniform, $M$-normal and stationary-complete $M$-ultrafilter is not a large cardinal property of $\kappa$, i.e. it does not imply the inaccessibility of $\kappa$. 
 In particular, this explains why Table \ref{table:schemeSmall} does not consider normality properties for filters on transitive weak $\kappa$-models that are weaker than genuineness.

 Finally, it should be noted that, analogous to Theorem \ref{theorem:inaccessiblelimit} and \ref{theorem:mahlolike}, Theorem \ref{theorem:weaklycompactideal} can be used to obtain an explicit characterization of weakly compact cardinals with the property that certain definable subsets of these cardinals do not lie in the corresponding weakly compact ideal. 
 Similar generalizations can be proven for all stronger large cardinal notions discussed below.

We now provide the desired characterization of weak compactness. The following result directly implies Item (\ref{item:wcb}) of Theorem \ref{theorem:SchemesSummary}.

\begin{theorem}
 The following statements are equivalent for every uncountable cardinal $\kappa$ and every regular cardinal $\theta>\kappa$: 
  \begin{enumerate} 
    \item $\kappa$ is weakly compact. 
    
     \item For all weak $\kappa$-models $M\prec\HH{\theta}$, there exists a uniform $M$-ultrafilter $U$ on $\kappa$ with $\Psi_{WC}(M,U)$.  

   \item For many weak $\kappa$-models $M\prec\HH{\theta}$, there exists a uniform, stationary-complete $M$-ultrafilter $U$ on $\kappa$ with $\Psi_{WC}(M,U)$. 
   
   \item For some weak $\kappa$-model $M\prec\HH{\theta}$, there exists a uniform $M$-ultrafilter $U$ on $\kappa$ with $\Psi_{ia}(M,U)$. 

%    \item For all weak $\kappa$-models  $M\prec\HH{\theta}$, there exists a ${<}\kappa$-powerset preserving $\kappa$-embedding $\map{j}{M}{\langle N,\IN_N\rangle}$ with $\crit{j}=\kappa$.  
    
%   \item For some weak $\kappa$-model $M\prec\HH{\theta}$, there exists a ${<}\kappa$-powerset preserving elementary embedding $\map{j}{M}{\langle N,\IN_N\rangle}$ with $\crit{j}=\kappa$.  

   \item For many transitive weak $\kappa$-models $M$, there exists a uniform, stationary-complete $M$-ultrafilter $U$ on $\kappa$ with $\Psi_{WC}(M,U)$. 
    
      \item For many transitive weak $\kappa$-models $M$, there exists a uniform $M$-ultrafilter $U$ on $\kappa$ with $\Psi_{ia}(M,U)$. 
   
%    \item For many transitive weak $\kappa$-models $M$, there exists a ${<}\kappa$-powerset preserving $\kappa$-embedding $\map{j}{M}{\langle N,\IN_N\rangle}$ with $\crit{j}=\kappa$. 
  \end{enumerate}
\end{theorem}

\begin{proof}
 The implication from (1) to (2) follows directly from Theorem \ref{theorem:weaklycompactideal}. 
  Moreover, since the inaccessibility of $\kappa$ implies that every element of $\HH{\theta}$ is contained in a weak $\kappa$-model $M\prec\HH{\theta}$ that is closed under countable sequence and all $M$-normal $M$-ultrafilters for such models $M$ are stationary-complete, Theorem \ref{theorem:weaklycompactideal} also shows that (1) implies both (3) and (5). 
  %
  % Using Corollary \ref{corollary:correspondence1} and Lemma \ref{lemma:inaccessible2}, we can see that both statements in (2) are equivalent to the corresponding statements in (4), and statement (5) is equivalent to statement (7). 
  %
  Now, assume, towards a contradiction, that $\kappa$ is not weakly compact, $M\prec\HH{\theta}$ is a weak $\kappa$-model and $U$ is a uniform $M$-ultrafilter that is ${<}\kappa$-amenable for $M$ and ${<}\kappa$-complete for $M$. 
  By Proposition \ref{proposition:correspondence1}.(4), we know that $j_U$ is  a ${<}\kappa$-powerset preserving elementary embedding  with $\crit{j}=\kappa$.  
   In this situation, Lemma \ref{lemma:inaccessible2} shows that $\kappa$ is inaccessible and hence our assumption implies the existence of a $\kappa$-Aronszajn tree $T$ with underlying set $\kappa$ in $M$. 
  Pick an $\Ult{M}{U}$-ordinal $\gamma$ that witnesses that $j_U$ jumps at $\kappa$ and pick an element $\beta$ of the $\gamma$-th level of $j_U(T)$ in $\langle\Ult{M}{U},\IN_U\rangle$. 
   Given $\alpha<\kappa$, there is $\bar{\alpha}<\kappa$ with the property that the initial segment of $T$ of height $\alpha$ is a subset of $\bar{\alpha}$ and hence we can find $\xi_\alpha<\bar{\alpha}$ with the property that, in $\langle\Ult{M}{U},\IN_U\rangle$, the ordinal $j_U(\xi_\alpha)$ is the unique element of the $j_U(\alpha)$-th level of $j_U(T)$ that lies below $\beta$. 
   But then elementarity implies that the set $\Set{\xi_\alpha}{\alpha<\kappa}$ is a cofinal branch through $T$, a contradiction. 
  These computations show that (4) implies (1). By first using Lemma \ref{lemma:inaccessible2}  to show that $\kappa$ is inaccessible and then capturing a $\kappa$-Aronszajn tree in a transitive weak $\kappa$-model, a variation of the previous argument shows that (6) also implies (1). 
  Since both (2) and (3) imply (4), and (5) implies (6), this completes the proof of the theorem. 
\end{proof}

 We end this section by proving the second statement of Theorem \ref{theorem:IdealContain}.(\ref{item:cont:WC}).

\begin{lemma}\label{lemma:notinweaklycompactideal}
  If $\kappa$ is weakly compact, then $\NN^\kappa_{wc}\not\in I_{WC}^\kappa$.
\end{lemma} 

\begin{proof}
 Assume that $\NN_{wc}^\kappa\in\II^\kappa_{WC}$. Then Theorem \ref{theorem:weaklycompactideal} implies that $\NN_{wc}^\kappa$ is an element of the weakly compact ideal. Then we may assume that $\kappa$ is  the least weakly compact cardinal with this property. 
 Then {\cite[Lemma 1.15]{MR1245524}} shows that the set  $A=\Set{\alpha<\kappa}{\textit{$\alpha$ is weakly compact}}$ contains a \emph{$1$-club}, i.e. there is a stationary subset of $A$ that contains all of its reflection points. 
 Since weak compactness implies stationary reflection, there is an $\alpha<\kappa$ with the property that $A\cap\alpha$ is stationary in $\alpha$. 
 Then $\alpha\in A$, $\alpha$ is weakly compact and $A\cap\alpha$ is a $1$-club in $\alpha$. 
Since the results of \cite{MR1245524} show that a subset of a weakly compact cardinal is an element of the weakly compact ideal if and only if its complement contains a $1$-club, we can conclude that $\NN_{wc}^\alpha\in\II^\kappa_{WC}$, a contradiction. 
\end{proof}

%%%%%%%%%%%%

\section{Weakly ineffable and ineffable cardinals}

Remember that, given a set $A$, an \emph{$A$-list} is a sequence $\seq{d_a}{a\in A}$ with $d_a\subseteq a$ for all $a\in A$. 
 Given an uncountable regular cardinal $\kappa$, a set $A\subseteq\kappa$ is then called \emph{ineffable} (respectively, \emph{weakly ineffable}) if for every $A$-list $\seq{d_\alpha}{\alpha\in A}$, there is $D\subseteq\kappa$ such that the set $\Set{\alpha\in A}{D\cap\alpha=d_\alpha}$ is stationary (respectively, \emph{unbounded}) in $\kappa$. 
The \emph{ineffable} (respectively, \emph{weakly ineffable}) \emph{ideal on $\kappa$} is the collection  of all subsets of $\kappa$ that are not ineffable (respectively, weakly ineffable). 
 These ideals were introduced by Baumgartner, and he has shown them to be normal ideals on $\kappa$ whenever $\kappa$ is (weakly) ineffable (see \cite{MR0384553}). 
 The key proposition is now an adaptation of \cite[Theorem 1.2.1]{MR0460120}, which shows that the ideals $\II_{wie}^\kappa$ and $\II_{{\prec}wie}^\kappa$ (respectively, the ideals $\II_{ie}^\kappa$ and $\II_{{\prec}ie}^\kappa$) both agree with the weakly ineffable (respectively, the ineffable) ideal, yielding Theorem \ref{theorem:Ideals}.(\ref{item:Ideals:wie}) and \ref{theorem:Ideals}.(\ref{item:Ideals:ie}). 
Moreover, the following result also immediately yields Theorem \ref{theorem:SchemesSummary} (\ref{item:wie}) and \ref{theorem:SchemesSummary}(\ref{item:ie}).

\begin{theorem}\label{theorem:ineffableideal}
 \begin{enumerate} 
  \item If $\kappa$ is an uncountable cardinal, $\vec{d}=\seq{d_\alpha}{\alpha\in A}$ is an $A$-list with $A\subseteq\kappa$, $M$ is a weak $\kappa$-model with $\vec{d}\in M$ and $U$ is an $M$-ultrafilter with $\Psi_{wie}(M,U)$ and with $A\in U$, then there is $D\subseteq\kappa$ such that the set $\Set{\alpha\in A}{D\cap\alpha=d_\alpha}$ is unbounded in $\kappa$. 

  \item   Let $\kappa$ be a weakly ineffable cardinal. 
   \begin{enumerate}
    \item If $A\subseteq\kappa$ is not contained in the weakly ineffable ideal, $\theta>\kappa$ is a regular cardinal and  $M\prec\HH{\theta}$ is a weak $\kappa$-model with $A\in M$, then there is an $M$-ultrafilter $U$ on $\kappa$ with $A\in U$ and $\Psi_{wie}(M,U)$. 
 %if $A\not\in \II^\kappa_{wie}$, then, for any (transitive) weak $\kappa$-model $M$ with $A\in M$, there is a genuine $M$-ultrafilter $U$ on $\kappa$ with $A\in U$.

    \item $\II^\kappa_{wie}=\II^\kappa_{\prec wie}$ is the weakly ineffable ideal on $\kappa$.
  \end{enumerate}

  \item If $\kappa$ is an uncountable cardinal, $\vec{d}=\seq{d_\alpha}{\alpha\in A}$ is an $A$-list with $A\subseteq\kappa$, $M$ is a weak $\kappa$-model with $\vec{d}\in M$ and $U$ is an $M$-ultrafilter with $\Psi_{ie}(M,U)$ and with $A\in U$, then there is $D\subseteq\kappa$ with the property that the set $\Set{\alpha\in A}{D\cap\alpha=d_\alpha}$ is stationary in $\kappa$. 

  \item Let $\kappa$ be an ineffable cardinal. 
    \begin{enumerate}
      \item If $A\subseteq\kappa$ is not contained in the  ineffable ideal, $\theta>\kappa$ is a regular cardinal and  $M\prec\HH{\theta}$ is a weak $\kappa$-model with $A\in M$, then there is an $M$-ultrafilter $U$ on $\kappa$ with $A\in U$ and $\Psi_{ie}(M,U)$. 
        % if $A\not\in \II^\kappa_{ie}$, then, for any (transitive) weak $\kappa$-model $M$ with $A\in M$, there is a normal $M$-ultrafilter $U$ on $\kappa$ with $A\in U$.

      \item  $\II_{ie}^\kappa=\II^\kappa_{\prec ie}$ is the ineffable ideal on $\kappa$. 
    \end{enumerate} 
 \end{enumerate}
\end{theorem}

\begin{proof}
  We only prove (3) and (4), since the proof for the case  of weakly ineffable cardinals proceeds in complete analogy (replacing \emph{stationary} by \emph{unbounded}, and replacing \emph{normal} by \emph{genuine} throughout). 

(3)  For every $\xi<\kappa$, let $x_\xi=\Set{\alpha\in A}{\xi\in d_\alpha}$. Then $\seq{x_\xi}{\xi<\kappa}\in M$. 
   Now, given $\xi<\kappa$, set $u_\xi=x_\xi$ if $x_\xi\in U$, and set $u_\xi=A\setminus x_\xi$ otherwise. 
 By our assumptions on $U$, we have $u_\xi\in U$ for all $\xi<\kappa$ and hence $H=\Delta_{\xi<\kappa}u_\xi$  is a stationary subset of $\kappa$. 
 Now, fix $\alpha,\beta\in H$ with $\alpha<\beta$ and $\xi<\alpha$. Then $\alpha,\beta\in u_\xi$. If $x_\xi\in U$, then $\alpha,\beta\in x_\xi$ and hence $\xi\in d_\alpha\cap d_\beta$. In the other case, if $x_\xi\notin U$, then $\alpha,\beta\in A\setminus x_\xi$ and hence $\xi\notin d_\alpha\cup d_\beta$. 
 In combination, this shows that $d_\alpha=d_\beta\cap\alpha$ holds for all $\alpha,\beta\in H$ with $\alpha<\beta$. Define $D=\bigcup\Set{d_\alpha}{\alpha\in H}$. Then our arguments show that the set $\Set{\alpha\in A}{D\cap\alpha=d_\alpha}$ is stationary in $\kappa$. 

 (4)(a) Let $\kappa$ be an ineffable cardinal and let $A\subseteq\kappa$ be ineffable, $\theta>\kappa$ be regular and $M\prec\HH{\theta}$ be a weak $\kappa$-model. 
  Pick an enumeration $\Set{x_\xi}{\xi<\kappa}$ of all subsets of $\kappa$ in $M$, and, for every $\alpha\in A$, set $d_\alpha=\Set{\xi<\alpha}{\alpha\in x_\xi}$. 
 Then there is $H\subseteq A$  stationary in $\kappa$, and $D\subseteq\kappa$ with $D\cap\alpha=d_\alpha$ for all $\alpha\in H$. 
Given $\xi<\kappa$, set $u_\xi=x_\xi$ if $\xi\in D$, and set $u_\xi=A\setminus x_\xi$ otherwise. 
 Define $U=\Set{u_\xi}{\xi<\kappa}$. The next claim  provides the desired conclusion.

\begin{claim*}
  $U$ is a normal $M$-ultrafilter with $A\in U$.
\end{claim*}

\begin{proof}[Proof of the Claim]
$H\setminus(\xi+1)\subseteq x_\xi$ for all $\xi\in D$ and $H\cap x_\xi\subseteq\xi+1$ for all $\xi\in\kappa\setminus D$. 
 Hence, we have $H\setminus(\xi+1)\subseteq u_\xi$ for all $\xi<\kappa$ and this directly implies that $U$ is an $M$-ultrafilter. 
 Moreover, it shows that  $H\subseteq\Delta_{\xi<\kappa}u_\xi$, and hence $U$ is normal. 
%  Note that, if $\alpha\in H$ and $\xi<\alpha$, then  either $$u_\xi=x_\xi ~ \Longrightarrow ~ \xi\in D ~  \Longrightarrow ~ \xi\in d_\alpha ~ \Longrightarrow ~ \alpha\in x_\xi ~ \Longrightarrow ~ \alpha\in u_\xi$$ or $$u_\xi=A\setminus x_\xi ~ \Longrightarrow ~ \xi\notin D ~ \Longrightarrow ~ \xi\not\in d_\alpha ~ \Longrightarrow ~ \alpha\not\in x_\xi ~ \Longrightarrow ~ \alpha\in u_\xi.$$ 
 % It is then easy to check that this implies the statement of the claim.
\end{proof}

 (4)(b) Let $\kappa$ be ineffable and assume that $A\subseteq\kappa$ is not an element of $\II_{ie}^\kappa\cap\II_{\prec ie}^\kappa$. 
 Then every $A$-list is contained in a weak $\kappa$-model $M$ with the property that there is a normal $M$-ultrafilter $U$ on $\kappa$ with $A\in U$. 
 By (3), this shows that $A$ is ineffable. This argument shows that the ineffable ideal is contained in both $\II^\kappa_{ie}$ and $\II^\kappa_{\prec ie}$. In the other direction, (4)(a) directly shows that $\II^\kappa_{\prec ie}$ is contained in the ineffable ideal. Moreover, by choosing $\theta=\kappa^+$ in (4)(a), the same conclusion can be established for $\II^\kappa_{ie}$. 
\end{proof}

We close this section by verifying Theorem \ref{theorem:IdealContain} (\ref{item:cont:wie}) and \ref{theorem:IdealContain} (\ref{item:cont:ie}).

\begin{lemma}
  If $\kappa$ is weakly ineffable, then $\NN^\kappa_{wc}\in\II^\kappa_{wie}$ and $\NN_{wie}^\kappa\notin\II^\kappa_{wie}$.
\end{lemma}

\begin{proof}
  For the first statement, let $A$ be the set of all inaccessibles $\alpha<\kappa$ which are not weakly compact. 
  Fix a bijection $\map{b}{\VV_\kappa}{\kappa}$ with $b[\VV_\alpha]=\alpha$ for all inaccessible $\alpha<\kappa$. 
  For every $\alpha\in A$, define $$d_\alpha ~ = ~ \{\goedel{0}{\lceil\varphi_\alpha\rceil}\} ~ \cup ~ \Set{\goedel{1}{b(x)}}{x\in X_\alpha} ~ \subseteq ~ \alpha,$$ where $X_\alpha\subseteq\VV_\alpha$ and $\lceil\varphi_\alpha\rceil\in\VV_\omega$ is the G\"odel number of a $\Pi^1_1$-formula $\varphi_\alpha(v^1)$ such that $\VV_\alpha\models\varphi_\alpha(X_\alpha)$ and $\VV_\beta\models\neg\varphi_\alpha(X_\alpha\cap\VV_\beta)$ for all $\beta<\alpha$.  
  Assume, for a contradiction, that $A$ is weakly ineffable. Then, the sequence $\seq{d_\alpha}{\alpha\in A}$ is an $A$-list, and, by the weak ineffability of $A$, we find $D\subseteq\kappa$ such that $U=\Set{\alpha\in A}{D\cap\alpha=d_\alpha}$ is an unbounded subset of $\kappa$. 
  Pick $\alpha,\beta\in U$ with $\alpha<\beta$. Then $\varphi_\alpha\equiv\varphi_\beta$ and $X_\alpha=X_\beta\cap\VV_\alpha$. Hence, $\VV_\alpha\models\varphi_\beta(X_\beta\cap\VV_\alpha)$, a contradiction. 
   Since $\NN^\kappa_{ia}\in\II^\kappa_{WC}\subseteq\II^\kappa_{wie}$, the above arguments show that $\NN^\kappa_{wc}\in\II^\kappa_{wie}$. 

  For the second statement, assume for a contradiction that $\kappa$ is the least weakly ineffable cardinal with the property that $\NN_{wie}^\kappa\in\II^\kappa_{wie}$. 
  Let $\vec{d}=\seq{d_\alpha}{\alpha\in\NN_{wie}^\kappa}$ be an $\NN_{wie}^\kappa$-list, and define $$A ~ = ~ \kappa\setminus\NN_{wie}^\kappa ~ = ~ \Set{\alpha<\kappa}{\textit{$\alpha$ is weakly ineffable}} ~ \notin ~ \II^\kappa_{wie}.$$ 
For every $\alpha\in A$, the fact that $\vec{d}\restriction\alpha$ is an $\NN_{wie}^\alpha$-list implies that there is  $D_\alpha\subseteq\alpha$ with the property that $\Set{\xi\in\NN_{wie}^\alpha}{d_\xi=D_\alpha\cap\xi}$ is an unbounded subset of $\alpha$. 
 But then, the sequence $\seq{D_\alpha}{\alpha\in A}$ is an $A$-list, and hence there is $D\subseteq\kappa$ such that $\Set{\alpha\in A}{D\cap\alpha=D_\alpha}$ is an unbounded subset of $\kappa$. 
 In this situation, the set $\Set{\alpha\in\NN_{wie}^\kappa}{D\cap\alpha=d_\alpha}$ is also unbounded  in $\kappa$. These computations show that the subset $N^\kappa_{wie}$ of $\kappa$ is weakly ineffable, contradicting our initial assumption. 
\end{proof}

\begin{lemma}
  If $\kappa$ is ineffable, then $\NN^\kappa_{wie}\in\II^\kappa_{ie}$ and $\NN_{ie}^\kappa\notin\II^\kappa_{ie}$.
\end{lemma}

\begin{proof}
   First, let $A$ denote the set of all inaccessibles $\alpha<\kappa$ which are not weakly ineffable. 
   For every $\alpha\in A$, pick an $\alpha$-list $\seq{d^\alpha_\xi}{\xi<\alpha}$  witnessing that $\alpha$ is not weakly ineffable. Given $\alpha\in A$, define $$D_\alpha ~ = ~ \Set{\goedel{\xi}{\zeta}}{\xi<\alpha, ~ \zeta\in d^\alpha_\xi} ~ \subseteq ~ \alpha.$$
   
   Assume, for a contradiction, that $A$ is ineffable and pick $D\subseteq\kappa$ such that the set $S=\Set{\alpha\in A}{D\cap\alpha=D_\alpha}$ is stationary in $A$. 
 Then there is a unique $\kappa$-list $\seq{d_\xi}{\xi<\kappa}$ with $d_\xi=d_\xi^\alpha$ for all $\alpha\in S$ and $\xi<\alpha$. 
  Since $\kappa$ is ineffable, there is $E\subseteq\kappa$ with the property that the set $T=\Set{\xi<\kappa}{E\cap\xi=d_\xi}$ is stationary in $\kappa$. Pick $\alpha\in S\cap\Lim(T)\subseteq A$. If $\xi\in T\cap\alpha$, then $d^\alpha_\xi=d_\xi=E\cap\xi$. Since $T\cap\alpha$ is unbounded in $\alpha$, this shows that the set $\Set{\xi<\alpha}{E\cap\xi=d^\alpha_\xi}$ is unbounded in $\alpha$, contradicting the fact that $\seq{d^\alpha_\xi}{\xi<\alpha}$ witnesses that $\alpha$ is not weakly ineffable. 
%     If we let $a^\lambda\subseteq\lambda$ be a canonical code for $\vec a ^\lambda$ for every $\lambda\in X$ (with the property that whenever $\alpha<\lambda$ is inaccessible, then $a^\lambda\cap\alpha$ is the canonical code for $\vec a^\lambda\upharpoonright\alpha$), then $\langle a^\lambda\mid\lambda\in X\rangle$ is an $X$-list, and hence there is $Y\subseteq\kappa$ such that $S=\{\lambda<\kappa\mid Y\cap\lambda=a^\lambda\}$ is a stationary subset of $X$. 
 %   This means that there is a $\kappa$-list $\vec a=\langle a_\alpha\mid\alpha<\kappa\rangle$ such that $\vec a^\lambda=\vec a\upharpoonright\lambda$ for every $\lambda\in S$. 
   %   Since $X$ is ineffable, there is $Z\subseteq\kappa$ such that $T=\{\alpha\in X\mid Z\cap\alpha=a_\alpha\}$ is a stationary subset of $\kappa$. Let $\lambda$ be a limit point of $S$ that is also an element of $S$, and also a limit point of $T$. Then, there is an unbounded subset $R$ of $\lambda$ such that $a_\lambda\cap\alpha=a_\alpha$ for every $\alpha\in R$. This contradicts $\vec a^\lambda=\vec a\upharpoonright\lambda$ to witness that $\lambda$ is not weakly ineffable.

  For the second statement, assume for a contradiction that $\kappa$ is the least ineffable cardinal for which $\NN_{ie}^\kappa\in\II^\kappa_{ie}$ holds. 
   Let $\seq{d_\alpha}{\alpha\in\NN_{ie}^\kappa}$ be an $\NN_{ie}^\kappa$-list, and let $A=\Set{\alpha<\kappa}{\textit{$\alpha$ is ineffable}}\notin\II^\kappa_{ie}$. 
    For every $\alpha\in A$, we find a set $D_\alpha\subseteq\alpha$ such that the set $\Set{\xi\in\NN_{ie}^\alpha}{D_\alpha\cap\xi=d_\xi}$ is stationary in $\alpha$.
    %    since $\vec a\upharpoonright\lambda$ is an $\NN_{ie}^\lambda$-list,  
    Then, the sequence $\seq{D_\alpha}{\alpha\in A}$ is an $A$-list, and hence there is $D\subseteq\kappa$ so that $S=\Set{\alpha\in A}{D\cap\alpha=D_\alpha}$ is a stationary subset of $\kappa$. 
    Let $C$ be a club subset of $\kappa$, and pick $\alpha\in\Lim(C)\cap S\subseteq A$. Then $\alpha$ is regular, $C\cap\alpha$ is a club in $\alpha$ and there is $\xi\in C\cap\NN_{ie}^\alpha$ with $d_\xi=D_\alpha\cap\xi=D\cap\xi$. This allows us to conclude that the set $\Set{\xi\in\NN_{ie}^\kappa}{D\cap\xi=d_\xi}$ is stationary in $\kappa$. 
    These arguments show that $\NN_{ie}^\kappa$ is ineffable, a contradiction. 
  %     to be a limit point of $C$ in $X$ so that $Y\cap\lambda=A_\lambda$. But then, $C\cap\lambda$ is a club subset of $\lambda$, so we find $\alpha<\lambda$ for which $a_\alpha=A_\lambda\cap\alpha=Y\cap\alpha$. That is, we have shown that $\{\alpha<\kappa\mid a_\alpha=Y\cap\alpha\}$ is a stationary subset of $\kappa$, showing that
\end{proof}

%%%%%%%%%%%%%%%%%%%%%%%%%%%%%%%%%%%%%
%%%%%%%%%%%%%%%%%%%%%%%%%%%%%%%%%%%%%

\section{A formal notion of Ramsey-like cardinals}\label{section:ramseylike}

In this section, we generalize the $\alpha$-Ramsey cardinals from \cite{MR3800756} to the class of $\Psi$-$\alpha$-Ramsey cardinals, and verify analogous results for this larger class of large cardinal notions.  
We start by introducing a number of generalizations of notions from \cite{MR3800756}. %We will then provide a general definition of canonical ideals induced by these large cardinals.
In the later sections of our paper, we will consider a number of special cases of these fairly general concepts. Our generalizations will be based on games that are similar to those from \cite{MR3800756}, which however allow for quite general extra winning conditions $\Psi$. We will usually only require them to satisfy the property introduced in the next definition.

\begin{definition}
  We say that a first order $\IN$-formula $\Psi(v_0,v_1)$ \emph{remains true under restrictions} if $\Psi(X,F\cap X)$ holds whenever $\emptyset\neq X\subseteq M$, $F\subseteq M$ and $\Psi(M,F)$ holds. 
  %    Let $\Psi$ be  in two free variables. 
\end{definition}

\begin{definition}\label{generalfiltergames}
Given uncountable regular cardinals $\kappa<\theta$ with $\kappa=\kappa^{<\kappa}$, a limit ordinal $\gamma\leq\kappa^+$, an unbounded subset $A$ of $\kappa$ and a first order formula $\Psi(v_0,v_1)$, 
 we let $G\Psi_\gamma^\theta(A)$ denote the game of perfect information between two players, the {\sf Challenger} and the {\sf Judge}, who take turns to produce $\subseteq$-increasing sequences $\seq{M_\alpha}{\alpha<\gamma}$ and $\seq{F_\alpha}{\alpha<\gamma}$, such that the following holds for every $\alpha<\gamma$: 
\begin{enumerate} 
  \item At any stage $\alpha<\gamma$, the {\sf Challenger} plays a $\kappa$-model $M_\alpha\prec\HH{\theta}$ such that the set $A$ and the sequences $\seq{M_{\bar{\alpha}}}{\bar{\alpha}<\alpha}$ and $\seq{F_{\bar{\alpha}}}{\bar{\alpha}<\alpha}$ are contained in $M_\alpha$, and then the {\sf Judge} plays an $M_\alpha$-ultrafilter $F_\alpha$ on $\kappa$. 
  
  \item $A\in F_0$.
\end{enumerate}
 In the end, we let $M_\gamma=\bigcup_{\alpha<\gamma}M_\alpha$ and $F_\gamma=\bigcup_{\alpha<\gamma}F_\alpha$. The {\sf Judge} wins the run of the game %a run $\langle\seq{M_\alpha}{\alpha<\gamma},\seq{F_\alpha}{\alpha<\gamma}\rangle$ of $G\Psi_\gamma^\theta(A)$
if $F_\gamma$ is an $M_\gamma$-normal filter such that $\Psi(M_\gamma,F_\gamma)$ holds. 
  Otherwise, the {\sf Challenger} wins.
\end{definition}

%In the following, we let $G_\gamma^\theta(A)$ denote the game $G\Psi_\gamma^\theta(A)$ with trivial \emph{extra winning condition} $\Psi$, i.e.\ the {\sf Judge} wins a run of $G_\gamma^\theta(A)$ if $F_\gamma$ is an $M_\gamma$-normal filter. 
%
 Note that if the {\sf Judge} ever plays a filter $F_\alpha$ that is not normal, then the {\sf Challenger} wins, for if $\Delta F_\alpha$ is non-stationary, then $F_\gamma$ cannot be $M_\gamma$-normal, for otherwise it has to contain $\Delta F_\alpha$ as an element. 
 On the other hand, if every $F_\alpha$ is $M_\alpha$-normal, then clearly also $F_\gamma$ is $M_\gamma$-normal.

\begin{definition}
  Let $\kappa$ be an uncountable cardinal with $\kappa=\kappa^{{<}\kappa}$, let $A$ be an unbounded subset of $\kappa$, let $\theta>\kappa$ be a regular cardinal, let $\gamma\leq\kappa^+$ be a limit ordinal, and let $\Psi(v_0,v_1)$ be a first order $\IN$-formula. 
  \begin{itemize} 
   \item $A$ has the \emph{$\Psi_\gamma^\theta$-filter property} if the {\sf Challenger} does not have a winning strategy in $G\Psi_\gamma^\theta(A)$. 
   
   \item $A$ has the \emph{$\Psi_\gamma^\forall$-filter property} if it has the $\Psi_\gamma^\vartheta$-filter property for all regular $\vartheta>\kappa$. 
  \end{itemize}
\end{definition}

%Following \cite{MR3800756},  if $\Psi$ denotes the trivial extra condition, then we say that $A$ has the $\gamma$-filter property rather than the $\Psi_\gamma^\forall$-filter property.
%
 Extending notions from \cite{MR2830415} and from \cite{MR3800756}, we introduce a generalization of the notion of Ramseyness.

\begin{definition}\label{definition:ramseylike}
  Let $\kappa<\theta$ be uncountable regular cardinals, let $\alpha\leq\kappa$ be an  infinite regular cardinal, let $A$ be an unbounded subset of $\kappa$ and let $\Psi(v_0,v_1)$ be a first order $\IN$-formula. 
\begin{itemize} 
  \item $A$ is \emph{$\Psi_\alpha^\kappa$-Ramsey} if for every $x\subseteq\kappa$, there is a transitive weak $\kappa$-model $M$ closed under ${<}\alpha$-sequences and a uniform, $\kappa$-amenable $M$-normal $M$-ultrafilter $U$ on $\kappa$ such that $x\in M$, $A\in U$ and $\Psi(M,U)$ holds. 
  
  \item $A$ is \emph{$\Psi_\alpha^\theta$-Ramsey} if for every $x\in\HH{\theta}$, there is a weak $\kappa$-model $M\prec\HH{\theta}$ closed under ${<}\alpha$-sequences and a uniform, $\kappa$-amenable $M$-normal $M$-ultrafilter $U$ on $\kappa$ such that $x\in M$, $A\in U$ and $\Psi(M,U)$ holds.

  \item $A$ is \emph{$\Psi_\alpha^\forall$-Ramsey} if it is $\Psi_\alpha^\theta$-Ramsey for every regular cardinal $\theta>\kappa$. 
  
  \item If $\vartheta\in\{\kappa,\theta,\forall\}$, then $\kappa$ is a \emph{$\Psi_\alpha^\vartheta$-Ramsey cardinal} if $\kappa$ is $\Psi_\alpha^\vartheta$-Ramsey as a subset of itself. 
 \end{itemize}
\end{definition}

 The above definition of Ramsey-like cardinals fits well with the main topics of this paper: 
Given an ordinal $\alpha$ and a property $\Psi(M,U)$ of models $M$ and $M$-ultrafilters $U$, we may form a stronger property $\bar{\Psi}_\alpha(M,U)$ by conjuncting the properties that $M$ is closed under ${<}\alpha$-sequences and $U$ is uniform, $M$-normal and $\kappa$-amenable for $M$.
 Then Scheme \ref{schemeNoel} holds true for $\Psi^\kappa_\alpha$-Ramsey cardinals and the property $\bar{\Psi}_\alpha(M,U)$ and Scheme \ref{schemeKappa} holds true for $\Psi^\forall_\alpha$-Ramsey cardinals and the property $\bar{\Psi}_\alpha(M,U)$. 
 Moreover, Theorem \ref{theorem:filtervsramsey} will show that Scheme \ref{schemeSmall} holds true as well for $\Psi^\forall_\omega$-Ramsey cardinals and the property $\bar{\Psi}_\omega(M,U)$.

The above definition covers many instances of specific Ramsey-like cardinals that have already been defined in the set-theoretic literature. Let $\alpha\le\kappa$ be regular cardinals. % Note that all of the properties $\Psi(v_0,v_1)$ that we are using in the definitions below are easily seen to remain true under restrictions.

\begin{itemize} 
  \item In Section \ref{section:completelyineffable}, we will show that a cardinal $\kappa$ is completely ineffable if and only if it is $\triv_\omega^\forall$-Ramsey, where $\triv(M,U)$ denotes the (trivial) property that $U=U$. 
  
  \item\cite[Definition 1.2]{MR2830415} A cardinal $\kappa$ is \emph{weakly Ramsey} if it is $\wf_\omega^\kappa$-Ramsey, where $\wf(M,U)$ denotes the property that the ultrapower $\Ult{M}{U}$ is well-founded. 
  
  \item\cite[Definition 2.11]{MR2830435} Given an ordinal $\beta\leq\omega_1$, a cardinal $\kappa$ is \emph{$\beta$-iterable} if it is $\wf\beta_\omega^\kappa$-Ramsey, where $\wf\beta(M,U)$ denotes the property that $U$ produces not only a well-founded ultrapower, but also $\beta$-many well-founded iterates of $M$. 
  
  \item\cite[Definition 4.5]{MR3800756} A cardinal $\kappa$ is \emph{super weakly Ramsey} if it is $\wf_\omega^{\kappa^+}$-Ramsey. 
  
  \item\cite[Definition 5.1]{MR3800756} A cardinal $\kappa$ is \emph{$\omega$-Ramsey} if it is $\wf_\alpha^\forall$-Ramsey. 
  
  \item\cite[Definition 4.11]{nielsen-welch} Given an ordinal $\beta\le\omega_1$, a cardinal $\kappa$ is $(\omega,\beta)$-Ramsey if it is $\wf\beta_\omega^\forall$-Ramsey. 
  
  \item\cite[Theorem 1.3]{MR2830415} A cardinal $\kappa$ is Ramsey if and only if it is $\cc_\omega^\kappa$-Ramsey, where $\cc(M,U)$ denotes the property that $U$ is countably complete. 

  \item\cite[Proof of Theorem 3.19]{mitchellforramsey} A cardinal $\kappa$ is \emph{weakly super Ramsey} if it is $\cc_\omega^{\kappa^+}$-Ramsey.
  
  \item Ineffably Ramsey cardinals were introduced by Baumgartner in \cite{MR0540770}. Adapting the above result on Ramsey cardinals, it will follow in Section \ref{section:ramsey} that a cardinal $\kappa$ is ineffably Ramsey if and only if it is $\scc_\omega^\kappa$-Ramsey, where $\scc(M,U)$ denotes the property that $U$ is stationary-complete. 
  
  \item In \cite[Definition 3.2]{MR1077260}, Feng introduced a hierarchy of Ramsey-like cardinals denoted as $\Pi_\beta$-Ramsey cardinals, for $\beta\in\Ord$, with $\Pi_0$-Ramsey cardinals being exactly Ramsey cardinals, and with $\Pi_1$-Ramsey cardinals being exactly ineffably Ramsey cardinals. All of these cardinals can be seen to fit into our hierarchy  of Ramsey-like cardinals. 

  \item\cite[Definition 5.1]{MR3800756} Given an uncountable regular cardinal $\alpha$, a cardinal $\kappa\geq\alpha$ is \emph{$\alpha$-Ramsey} if it is $\triv_\alpha^\forall$-Ramsey (or, equivalently, $\cc_\alpha^\forall$-Ramsey).  $\omega_1$-Ramsey cardinals were also called \emph{$\omega$-closed Ramsey} in \cite[Definition 2.6]{mitchellforramsey}.

  \item\cite[Definition 1.4]{MR2830415} A cardinal $\kappa$ is \emph{strongly Ramsey} if it is $\triv_\kappa^\kappa$-Ramsey (or, equivalently, $\cc_\kappa^\kappa$-Ramsey).

  \item\cite[Definition 1.5]{MR2830415} A cardinal $\kappa$ is \emph{super Ramsey} if it is $\triv_\kappa^{\kappa^+}$-Ramsey (or, equivalently, $\cc_\kappa^{\kappa^+}$-Ramsey). 
  
%  \item (Definition \ref{definition:elementaryramsey}) $\kappa$ is an \emph{elementary Ramsey} cardinal (or $\prec$-Ramsey) if $\kappa$ is $cc_\omega^\forall$-Ramsey. 

  \item\cite[Definition 2.7]{nielsen-welch} A cardinal $\kappa$ is a \emph{genuine $\alpha$-Ramsey} cardinal if $\kappa$ is $\infi_\alpha^\forall$-Ramsey, where $\infi(M,U)$ denotes the property that $U$ is genuine. 
  
  %We say that $\kappa$ is a \emph{genuine Ramsey} cardinal (or $\infty$-Ramsey) if it is $\infty_\omega^\forall$-Ramsey.
  
  \item\cite[Definition 2.7]{nielsen-welch} A cardinal $\kappa$ is a \emph{normal $\alpha$-Ramsey} cardinal if $\kappa$ is $\delt_\alpha^\forall$-Ramsey, where $\delt(M,U)$ denotes the property that $U$ is normal. 
  %We say that $\kappa$ is a \emph{normal Ramsey} cardinal (or $\Delta$-Ramsey) if it is $\Delta_\omega^\forall$-Ramsey.

 \item A cardinal $\kappa$ is locally measurable if and only if it is $(\Psi_{ms})^\kappa_\omega$-Ramsey.  

 \item A cardinal $\kappa$ is measurable if and only if it is $(\Psi_{ms})^\forall_\alpha$-Ramsey for some (equivalently, for all) regular $\alpha\leq\kappa$. 
\end{itemize}

The following lemma is a straightforward generalization of \cite[Theorem 5.6]{MR3800756}.

\begin{lemma}\label{lemma:filtervsramsey}
  Let $\kappa<\theta$ be uncountable regular cardinals with $\kappa=\kappa^{{<}\kappa}$, let $A$ be an unbounded subset of $\kappa$, let $\gamma\leq\kappa$ be regular, and let $\Psi(v_0,v_1)$ be a first order $\IN$-formula. 
 Then, the following statements hold: 
\begin{enumerate} 
  \item If $A$ has the $\Psi_\gamma^\theta$-filter property, then $A$ is $\Psi_\gamma^\theta$-Ramsey. 

  \item If $\Psi$ remains true under restrictions and $\vartheta>2^{{<}\theta}$ is a regular cardinal, and $A$ is $\Psi_\gamma^{\vartheta}$-Ramsey, then $A$ has the $\Psi_\gamma^\theta$-filter property.
\end{enumerate}
\end{lemma}

\begin{proof}
  First, assume that $A$ has the $\Psi_\gamma^\theta$-filter property. Given $x\in\HH{\theta}$, let $\sigma$ be a any strategy for the {\sf Challenger} in the game $G\Psi_\gamma^\theta(A)$ that ensures that $x\in M_0$. 
Since, by our assumption, $\sigma$ cannot be a winning strategy for the {\sf Challenger}, it follows that there is a run $\langle\seq{M_\alpha}{\alpha<\gamma},\seq{F_\alpha}{\alpha<\gamma}\rangle$ of this game which the {\sf Judge} wins. 
Then, $M=\bigcup_{\alpha<\gamma}M_\alpha$ is closed under ${<}\gamma$-sequences, and $F_\gamma=\bigcup_{\alpha<\gamma}F_\alpha$ is a uniform $M$-normal $M$-ultrafilter such that $A\in F_\gamma$ and $\Psi(M,F_\gamma)$ holds. 
% Let $\langle M_n\mid n<\omega\rangle$ be the sequence of $\kappa$-models $M_n\prec H(\theta)$, and let $\langle U_n\mid n<\omega\rangle$ be the sequence of $M_n$-ultrafilters on $\kappa$ played during this run. 
 % Let $M=\bigcup_{n<\gamma}M_n$, and let $U=\bigcup_{n<\gamma}U_n$. 
 %
 By the same argument as in the proof of Lemma  \ref{lemma:WCmodelsBoundedIdeal}, the filter $F_\gamma$ is $\kappa$-amenable for $M$, as desired. 

 Now, assume that $\vartheta>2^{{<}\theta}$ is regular, and that $A$ is $\Psi_\gamma^{\vartheta}$-Ramsey, as witnessed by $M\prec\HH{\vartheta}$ and $U$. 
 Assume for a contradiction that $A$ does not have the $\Psi_\gamma^\theta$-filter property. Then, there exists a winning strategy $\sigma\subseteq\HH{\theta}$ for the {\sf Challenger} in the game $G\Psi_\gamma^\theta(A)$. It follows that $\sigma\in\HH{\vartheta}$, and hence, by elementarity, such a winning strategy exists also in $M$. 
 But this is a contradiction, because the {\sf Judge} can obviously win any run of the game in $M$ by playing  suitable pieces of $U$, using that initial segments of the run of the game are contained in $M$, since $M$ is closed under ${<}\gamma$-sequences, and that $\Psi$ remains true under restrictions. 
\end{proof}

The next lemma is a generalization of {\cite[Lemma 3.3]{MR3800756}}, and shows, together with Lemma \ref{lemma:filtervsramsey}, that in many cases the $\Psi_\gamma^\forall$-Ramseyness of some cardinal $\kappa$ is in fact a local property -- namely it is equivalent to its $\Psi_\gamma^\theta$-Ramseyness for $\theta=(2^\kappa)^+$. We will need the following, which is a property that is shared by all $\Psi$'s considered in this paper, except for the case when $\Psi(M,U)\equiv\wf(M,U)$.

\begin{definition}
  A first order $\IN$-formula $\Psi(v_0,v_1)$ \emph{remains true under $\kappa$-restrictions} if $\Psi(X,F\cap X)$ holds whenever $\emptyset\neq\mathcal P(\kappa)^X\subseteq\mathcal P(\kappa)^M$ and $\Psi(M,F)$ holds. 
\end{definition}

\begin{lemma}\label{lemma:filterequivalence}
    Let $\kappa=\kappa^{<\kappa}$ be an uncountable cardinal, let $\gamma\leq\kappa^+$ be a limit ordinal, let $\theta>\kappa$ be regular, and let $\Psi(M,U)$ be a property that remains true under $\kappa$-restrictions. Then, an unbounded subset $A$ of $\kappa$ has the $\Psi_\gamma^\forall$-filter property if and only if it has the $\Psi_\gamma^\theta$-filter property.    %
%    In particular, an unbounded subset  $A$ of $\kappa$ has the $\gamma$-filter property %(i.e. the $\mathbf T_\gamma^\forall$-filter property)     if and only if $A$ has the $\mathbf T_\gamma^\theta$-filter property.  
\end{lemma}

\begin{proof}
  Let $\theta_0$ and $\theta_1$ both be regular cardinals greater than $\kappa$, and assume that the {\sf Challenger} has a winning strategy $\sigma_0$ in the game $G\Psi_\gamma^{\theta_0}(A)$. We construct a winning strategy $\sigma_1$ for the {\sf Challenger} in the game $G\Psi_\gamma^{\theta_1}(A)$. Whenever the {\sf Challenger} would play $M_\alpha$ in a run of the game $G\Psi_\gamma^{\theta_0}(A)$ where he is following his winning strategy $\sigma_0$, then $\sigma_1$ shall tell him to play some $M_\alpha^*$ which is a valid move in the game $G\Psi_\gamma^{\theta_1}(A)$ such that $M_\alpha^*\supseteq\mathcal P(\kappa)\cap M_\alpha$. Every possible response $F_\alpha^*$ of the {\sf Judge} in the game $G\Psi_\gamma^{\theta_1}(A)$ induces a response $F_\alpha=F_\alpha^*\cap M_\alpha$ in the game $G\Psi_\gamma^{\theta_0}(A)$. We use this induced response together with the strategy $\sigma_0$ to obtain the next move of the {\sf Challenger} in the game $G\Psi_\gamma^{\theta_0}(A)$, and continue playing these two games in this way for $\gamma$-many steps. As the {\sf Challenger} is following a winning strategy in the game $G\Psi_\gamma^{\theta_0}(A)$, it follows that $F_\gamma$ is either not $M_\gamma$-normal or $\Psi(M_\gamma,F_\gamma)$ fails. But, using our assumptions, the same is the case for $M_\gamma^*=\bigcup_{\alpha<\gamma}M_\alpha^*$ and $F_\gamma^*=\bigcup_{\alpha<\gamma}F_\alpha^*$, showing that $\sigma_1$ is indeed a winning strategy.
\end{proof}

The following is now immediate from Lemma \ref{lemma:filtervsramsey} and Lemma \ref{lemma:filterequivalence}.

\begin{corollary}\label{corollary:RamseyLocal}
  Let $\kappa$ be an uncountable cardinal, let $\gamma\leq\kappa^+$, and let $\Psi(M,U)$ be a property that remains true under $\kappa$-restrictions. Then, an unbounded subset $A$ of $\kappa$ is $\Psi_\gamma^\forall$-Ramsey if and only if it is $\Psi_\gamma^\theta$-Ramsey for some regular cardinal $\theta>2^\kappa$.
\end{corollary}

The next result immediately yields Theorem \ref{theorem:SchemesSummary}.(\ref{item:schemeSmall:wR}) and (\ref{item:schemeSmall:nRobvious}), and also justifies the entries for $\triv_\omega^\kappa$-Ramsey, $\infi_\omega^\kappa$-Ramsey, $\delt_\omega^\kappa$-Ramsey, $\cc_\omega^\forall$-Ramsey, and the final entry for $\Delta_\omega^\forall$-Ramsey cardinals in Table \ref{table:schemeKappa}.  We will show in Section \ref{section:gnramsey} that the notions of $\Delta_\omega^\forall$-Ramseyness, $\infty_\omega^\forall$-Ramseyness and $\scc_\omega^\forall$-Ramseyness coincide.

\begin{theorem}\label{theorem:filtervsramsey}
  Let $\kappa$ be an uncountable cardinal, let $A$ be an unbounded subset of $\kappa$, let $\gamma\leq\kappa$ be a regular cardinal, and let $\Psi(v_0,v_1)$ be a first order $\IN$-formula that remains true under restrictions. 
 Then, the following statements are equivalent for all $\gamma\leq\lambda\leq\kappa$ with $\lambda^{{<}\gamma}=\lambda$: 
  \begin{enumerate} 
    \item $A$ is $\Psi_\gamma^\forall$-Ramsey. 

    \item For any regular cardinal $\theta>\kappa$ and many $(\lambda,\kappa)$-models $M\prec\HH{\theta}$ closed under ${<}\gamma$-sequences, there exists a uniform, $\kappa$-amenable, $M$-normal $M$-ultrafilter $U$ on $\kappa$ such that $A\in U$ and $\Psi(M,U)$ holds. 

    \item For any regular cardinal $\theta>\kappa$, and many $(\lambda,\kappa)$-models $M\prec\HH{\theta}$ closed under ${<}\gamma$-sequences, there exists a $\kappa$-powerset preserving $\kappa$-embedding $\map{j}{M}{\langle N,\IN_N\rangle}$ such that $\kappa^N\IN_N j(A)$ and $\Psi(M,U_j)$ holds. 
  \end{enumerate}
\end{theorem}

\begin{proof}
   The implication from (1) to (2) is trivial in case $\lambda=\kappa$. Given $\lambda<\kappa$, pick a $(\lambda,\kappa)$-model $\langle\bar M,\bar U\rangle\prec\langle M,U\rangle$ closed under ${<\gamma}$-sequences and containing $x$ and $A$ as elements. 
 Then, by elementarity, and since $\Psi$ remains true under restrictions, $\langle\bar M,\bar U\rangle$ is as desired. 
  The equivalence between (2) and (3) follows from Proposition \ref{proposition:correspondence1}.(5), Corollary \ref{corollary:correspondence2}, and Lemma \ref{lemma:general}. 
 Next, note that Lemma \ref{lemma:inaccessible2} shows that (3) implies that $\kappa$ is inaccessible. 
  The implication from (2) to (1) is again trivial in case $\lambda=\kappa$. For smaller $\lambda$, note that the size of $M$ did not matter in the proof of Lemma \ref{lemma:filtervsramsey}.(2), as long as $\gamma+1\subseteq M$. 
 This shows that (2) implies $A$ to have the $\Psi_\gamma^\forall$-filter property. Applying Lemma \ref{lemma:filtervsramsey}.(1) then yields $A$ to be $\Psi_\gamma^\forall$-Ramsey.
\end{proof}

Let us now introduce ideals that are canonically induced by our Ramsey-like cardinals.

\begin{definition}\label{definition:ramseyidealtheta}
 Let $\Psi(v_0,v_1)$ be a first order $\IN$-formula and let $\kappa$ be a $\Psi_\alpha^\theta$-Ramsey cardinal with $\theta\geq\kappa$ regular and $\alpha\leq\kappa$ regular and infinite. 
 We define the \emph{$\Psi_\alpha^\theta$-Ramsey ideal on $\kappa$} to be the set $$\II\Psi_\alpha^\theta(\kappa) ~ = ~ \Set{A\subseteq\kappa}{\textit{$A$ is not $\Psi_\alpha^\theta$-Ramsey}}.$$ 
\end{definition}

 If $\theta=\kappa$, the above ideals are particular instances of the ideals defined in Definition \ref{definition:Ideals}.(\ref{definition:Ideals-KappaSizedNonelem}): 
 Given an ordinal $\alpha$ and a property $\Psi(M,U)$ of models $M$ and $M$-ultrafilters $U$, if $\bar{\Psi}_\alpha(M,U)$ is the induced property defined in the discussion following Definition \ref{definition:ramseylike}, then $\II^\kappa_{\bar{\Psi}_\alpha}=\II\Psi^\kappa_\alpha(\kappa)$ holds for every $\Psi^\kappa_\alpha$-Ramsey cardinal $\kappa$. In particular, the discussion following Definition \ref{definition:Ideals} and Lemma \ref{lemma:AllIdealsNormal} show that these ideals are proper and normal.
 Similarly, if we further strengthen the property $\bar{\Psi}_\alpha(M,U)$ to obtain a property $\bar{\Psi}^+_\alpha(M,U)$ that also demands the model $M$ to be an elementary submodel of $\HH{\kappa^+}$, then $\II^\kappa_{\bar{\Psi}^+_\alpha}=\II\Psi^{\kappa^+}_\alpha(\kappa)$ holds for every $\Psi^{\kappa^+}_\alpha$-Ramsey cardinal $\kappa$, and hence these ideals are also proper and normal.

\begin{proposition}\label{proposition:compareideals}
  Let $\Psi(v_0,v_1)$ and $\Omega(v_0,v_1)$ be first order $\IN$-formulas that remain true under restrictions, such that $\Omega$ implies $\Psi$, let $\alpha\leq\beta\leq\kappa$ be regular infinite cardinals, let $\vartheta\geq\theta\ge\kappa$ be regular cardinals, and let $\kappa$ be an $\Omega_\beta^{\vartheta}$-Ramsey cardinal. 
 Then, $\II\Psi_\alpha^{\theta}(\kappa)\subseteq\II\Omega_\beta^{\vartheta}(\kappa)$. 
\end{proposition}

\begin{proof}
  Assume that $A\not\in\II\Omega_\beta^{\vartheta}(\kappa)$. Then, for any $x\in\HH{\theta}\cup\POT{\kappa}$, there is a weak $\kappa$-model $M$, elementary in $\HH{\vartheta}$ in case $\vartheta>\kappa$, and transitive in case $\vartheta=\kappa$, that is closed under ${<}\beta$-sequences, with $x\in M$, with $\theta\in M$ in case $\theta<\vartheta$, and with a uniform, $\kappa$-amenable, $M$-normal $M$-ultrafilter $U$ on $\kappa$ with $A\in U$, such that $\Omega(M,U)$ holds. 
 But then, using that $\Omega$ implies $\Psi$, which remains true under restrictions, either $M\cap \HH{\theta}$ (in case $\theta>\kappa$) or $M\cap\HH{\kappa^+}$ (in case $\theta=\kappa$) witnesses, together with $U$, that $A\not\in\II\Psi_\alpha^{\theta}(\kappa)$. 
\end{proof}

If $\kappa$ is a $\Psi_\alpha^\forall$-Ramsey cardinal, it follows by a trivial cardinality argument that the ideals $\II\Psi_\alpha^\theta$ on $\kappa$ stabilize for sufficiently large $\theta$.\footnote{In general, we do not know of any way to find a non-trivial bound on what a sufficiently large $\theta$ would be relative to $\kappa$.}  
 We can thus make the following definition, that corresponds to Definition \ref{definition:Ideals}.(\ref{definition:Ideals-KappaSizedElem}).

\begin{definition}\label{definition:ramseyideal}
  Let $\Psi(v_0,v_1)$ be a first order $\IN$-formula that remains true under restrictions and let $\kappa$ be a $\Psi_\alpha^\forall$-Ramsey cardinal with $\alpha\leq\kappa$ regular and infinite. We define the \emph{$\Psi_\alpha^\forall$-Ramsey ideal on $\kappa$} to be the set  $$\II\Psi_\alpha^\forall(\kappa) ~ = ~ \bigcup\Set{\II\Psi_\alpha^\theta(\kappa)}{\textit{$\theta>\kappa$ regular}}.$$
\end{definition}

 Given an ordinal $\alpha$ and a property $\Psi(M,U)$ of models $M$ and $M$-ultrafilters $U$, if $\bar{\Psi}_\alpha(M,U)$ is the induced property defined in the discussion following Definition \ref{definition:ramseylike}, then the above remarks directly show that $\II^\kappa_{{\prec}\bar{\Psi}_\alpha}=\II\Psi^\forall_\alpha(\kappa)$ holds for all $\Psi^\forall_\alpha$-Ramsey cardinals $\kappa$. In particular, these ideals are normal and proper. 
 %
%The above remarks then show that, in the situation of the previous definition, we have that $\II\Psi_\alpha^\forall(\kappa)=\II\Psi_\alpha^\theta(\kappa)$ for all sufficiently large regular cardinals $\theta$. In particular, this trivially yields that... \[inline]{Establish connection to ideals defined in the introduction. Perhaps say again that it is normal.}
 %
 In addition, Proposition \ref{proposition:compareideals} shows that, for properties $\Psi$ and $\Omega$ that remain true under restrictions such that $\Omega$ implies $\Psi$, and for regular infinite cardinals $\alpha\leq\beta\leq\kappa$, if $\kappa$ is $\Omega_\beta^\forall$-Ramsey, then $\II\Psi_\alpha^\forall(\kappa)\subseteq\II\Omega_\beta^\forall(\kappa)$.

 In the remainder of this section, we prove results concerning the relations of the ideals produced by Definition \ref{definition:ramseyidealtheta}  and Definition \ref{definition:ramseyideal}. 
The following sets will be central for this analysis. 
 Given regular cardinals $\alpha<\kappa$ and a first order $\IN$-formula $\Psi(v_0,v_1)$, we make the following definitions:
 \begin{itemize}
   \item $\NN\Psi^\kappa_\alpha(\kappa) =  \Set{\gamma\in(\alpha,\kappa)}{\textit{$\gamma$ is not a $\Psi_\alpha^{\gamma}$-Ramsey cardinal}}$. 

     \item $\NN\Psi^\kappa_\kappa(\kappa) =  \Set{\gamma<\kappa}{\textit{$\gamma$ is not a $\Psi_\gamma^{\gamma}$-Ramsey cardinal}}$. 

  \item $\NN\Psi^{\kappa^+}_\alpha(\kappa) ~ = ~ \Set{\gamma\in(\alpha,\kappa)}{\textit{$\gamma$ is not a $\Psi_\alpha^{\gamma^+}$-Ramsey cardinal}}$. 

 \item $\NN\Psi^{\kappa^+}_\kappa(\kappa) ~ = ~ \Set{\gamma<\kappa}{\textit{$\gamma$ is not a $\Psi_\gamma^{\gamma^+}$-Ramsey cardinal}}$.   
 \end{itemize}

The following lemmas now show that under mild assumptions on the formula $\Psi$, the $\Psi_\alpha^\kappa$-Ramsey, $\Psi_\alpha^{\kappa^+}$-Ramsey and $\Psi_\alpha^\forall$-Ramsey cardinals are strictly increasing in terms of consistency strength, thus strengthening and generalizing {\cite[Proposition 5.2 and Proposition 5.3]{MR3800756}} and {\cite[Theorem 3.14]{MR2830415}}. 
 They also show that if $\kappa$ is a $\Psi_\alpha^\forall$-Ramsey cardinal, then $\II\Psi_\alpha^\forall(\kappa)\supsetneq\II\Psi_\alpha^\kappa(\kappa)$.

\begin{lemma}\label{lemma:notinkapparamseylikeideal}
 Let $\Psi(v_0,v_1)$ be a first order $\IN$-formula that remains true under restrictions, let $\alpha$ be a regular cardinal, and let $\kappa\ge\alpha$ be a $\Psi_\alpha^{\kappa^+}$-Ramsey cardinal such that $\Psi$ is absolute between $\VV$ and $\HH{\kappa^+}$. Then, the following statements hold true.
  \begin{enumerate} 
      \item
$\NN\Psi^\kappa_\alpha(\kappa)\in\II\Psi_\alpha^{\kappa^+}(\kappa)$. 

  \item $\NN\Psi^{\kappa^+}_\alpha(\kappa)\notin\II\Psi_\alpha^{\kappa}(\kappa)$.
  \end{enumerate}
\end{lemma}

\begin{proof}
  (1) Assume that $A=\NN\Psi^\kappa_\alpha(\kappa)\notin\II\Psi_\alpha^{\kappa^+}(\kappa)$. 
 Then, there is a weak $\kappa$-model $M\prec\HH{\kappa^+}$ and a $\kappa$-powerset preserving $\kappa$-embedding $\map{j}{M}{\langle N,\IN_N\rangle}$ such that $A\in M$ and $\kappa^N\IN_N j(A)$. 
 First assume that $\alpha<\kappa$. 
 Then our assumptions on $\Psi$ imply that the set $A$ consists of all $\gamma$ in $(\alpha,\kappa)$ that are not $\Psi_\alpha^{\gamma}$-Ramsey cardinals in $M$. 
  Therefore, $\kappa^N$ is not a $\Psi_{j(\alpha)}^{\kappa^N}$-Ramsey cardinal in $\langle N,\IN_N\rangle$. 
 However, since $j$ is $\kappa$-powerset preserving and $M$ is an elementary submodel of $\HH{\kappa^+}$, we can use the isomorphism provided by Lemma \ref{lemma:kappapowersetpreservings}.(\ref{lemma:kappapowersetpreserving:2}) to conclude that $\kappa^N$ is $\Psi_{j(\alpha)}^{\kappa^N}$-Ramsey in $\langle N,\IN_N\rangle$, a contradiction. 
 In the other case, if $\alpha=\kappa$, then our assumptions ensure that $A$ consists of all $\gamma<\kappa$ that are not $\Psi_\gamma^\gamma$-Ramsey cardinals in $M$ and hence $\kappa^N$ is not a $\Psi_{\kappa^N}^{\kappa^N}$-Ramsey cardinal in $\langle N,\IN_N\rangle$.  
 As above, we can use Lemma \ref{lemma:kappapowersetpreservings}.(\ref{lemma:kappapowersetpreserving:2}) to derive a contradiction.

  (2) First, assume that $\kappa>\alpha$ and $\kappa$ is the least $\Psi_\alpha^{\gamma^+}$-Ramsey cardinal $\gamma>\alpha$ with the property that $\NN\Psi^{\gamma^+}_\alpha(\gamma)\in\II\Psi_\alpha^\gamma(\gamma)\subseteq\II\Psi_\alpha^{\gamma^+}(\gamma)$.  
  By Definition \ref{definition:ramseyidealtheta}, Proposition \ref{proposition:correspondence1}.(5) and Corollary \ref{corollary:correspondence2}, there is a weak $\kappa$-model $M\prec\HH{\kappa^+}$, and a $\kappa$-powerset preserving $\kappa$-embedding $\map{j}{M}{\langle N,\IN_N\rangle}$ with $\kappa^N\notin_N j(\NN\Psi^{\kappa^+}_\alpha(\kappa))$. 
 Therefore, $\kappa^N$ is a $\Psi_{j(\alpha)}^{(\kappa^N)^+}$-Ramsey cardinal below $j(\kappa)$ in $\langle N,\IN_N\rangle$, and hence, by minimality, it follows that $\NN\Psi^{(\kappa^N)^+}_{j(\alpha)}(\kappa^N)\notin\II\Psi_{j(\alpha)}^{\kappa^N}(\kappa^N)$ holds in this model. 
  By our assumptions on $\Psi$, the model $M$ computes both $\NN\Psi^{\kappa^+}_\alpha(\kappa)$ and $\II\Psi^\kappa_\alpha(\kappa)$ correctly. 
In this situation, Lemma \ref{lemma:kappapowersetpreservings} (\ref{lemma:kappapowersetpreserving:2}) shows that $\NN\Psi^{(\kappa^N)^+}_{j(\alpha)}(\kappa^N)\in\II\Psi_{j(\alpha)}^{\kappa^N}(\kappa^N)$ holds in $\langle N,\epsilon_N\rangle$, a contradiction. 
\end{proof}

The next result shows that, in many important cases, ideals of the form $\II\Psi^{\kappa^+}_\alpha(\kappa)$ are proper subsets of the corresponding ideals $\II\Psi^\forall_\alpha(\kappa)$.

\begin{lemma}\label{lemma:notinkappaplusramseylikeideal}
  Let $\Psi(v_0,v_1)$ be a first order $\IN$-formula that remains true under restrictions and is absolute between $\VV$ and $\HH{\theta}$ for sufficiently large regular cardinals $\theta$. 
  If $\kappa$ is a $\Psi_\alpha^\forall$-Ramsey cardinal for some regular $\alpha<\kappa$, then $\NN\Psi^{\kappa^+}_\alpha(\kappa)\in\II\Psi_\alpha^\forall(\kappa)$.  
\end{lemma}

\begin{proof}
Assume that $B=\NN\Psi^{\kappa^+}_\alpha(\kappa)\notin\II\Psi_\alpha^\forall(\kappa)$. 
 Let $\theta>(2^\kappa)^+$ be a sufficiently large regular cardinal. 
 Then, there is a weak $\kappa$-model $M\prec\HH{\theta}$ with a $\kappa$-powerset preserving $\kappa$-embedding $\map{j}{M}{\langle N,\IN_N\rangle}$ such that $\kappa^N\IN_N j(B)$. 
  Since our assumption on $\Psi$ imply that $M$ computes $\NN\Psi^{\kappa^+}_\alpha(\kappa)$ correctly, the model $\langle N,\IN_N\rangle$ thinks that $\kappa^N$ is not $\Psi_{j(\alpha)}^{(\kappa^N)^+}$-Ramsey. 
  However, by $\kappa$-powerset preservation and by Lemma \ref{lemma:kappapowersetpreservings} (\ref{lemma:kappapowersetpreserving:2}), the model $\langle N,\IN_N\rangle$ also thinks that $\kappa^N$ is $\Psi_{j(\alpha)}^{(\kappa^N)^+}$-Ramsey,  a contradiction.
\end{proof}

\begin{lemma}\label{lemma:IdealsNicelyProper}
 Let $\Psi(M,U)$ be a first order property such that whenever $\kappa$ is an infinite cardinal and $M_0$, $M_1$, $U_0$ and $U_1$ satisfy the  properties listed below, then $\Psi(M_1,U_1)$ holds.
 \begin{itemize} 
  \item $M_i$ is a transitive weak $\kappa$-model for all $i<2$.  
 
  \item $U_i$ is a uniform, $\kappa$-amenable and $M_i$-normal $M_i$-ultrafilter on $\kappa$ for all $i<2$. 
 
  \item $\Psi(M_0,U_0)$ holds and $M_1,U_1\in \HH{\kappa^+}^{M_0}$. 

  \item Some surjection $\map{s}{\kappa}{\VV_\kappa}$ is an element of $M_0$.

  \item $\Psi(j^*_{U_0}(M_1),j^*_{U_0}(U_1))$ holds in $\langle\Ult{M_0}{U_0},\IN_{U_0}\rangle$, where $j_{U_0}^*$ is the $\IN$-isomorphism induced by the ultrapower embedding $j_{U_0}\colon M_0\to\Ult{M_0}{U_0}$ and by $s$, as in Lemma \ref{lemma:kappapowersetpreservings}.(\ref{lemma:kappapowersetpreserving:2}). 
 \end{itemize}
 Then, if $\kappa$ is a $\Psi^\vartheta_\alpha$-Ramsey cardinal with $\alpha\leq\kappa$ and $\vartheta\in\{\kappa,\kappa^+\}$, then $\NN\Psi^\vartheta_\alpha(\kappa)\notin\II\Psi^\vartheta_\alpha(\kappa)$. 
\end{lemma}

\begin{proof}
 First, assume that there is an ordinal $\alpha$ and a $\Psi^\kappa_\alpha$-Ramsey cardinal $\kappa\geq\alpha$ with $\NN\Psi^\kappa_\alpha(\kappa)\in\II\Psi^\kappa_\alpha(\kappa)$. 
  Let $\kappa$ be minimal with this property and pick $x\subseteq\kappa$ witnessing that $\NN\Psi^\kappa_\alpha(\kappa)$ is not  $\Psi^\kappa_\alpha$-Ramsey. Pick a surjection $s\colon\kappa\to\VV_\kappa$. Since $\kappa$ is $\Psi^\kappa_\alpha$-Ramsey, there is a weak $\kappa$-model $M_0$ closed under ${<}\alpha$-sequences and a uniform, $\kappa$-amenable $M_0$-normal $M_0$-ultrafilter $U_0$ such that $x,s\in M_0$ and $\Psi(M_0,U_0)$ holds. 
  If $\alpha<\kappa$, then we set $\beta=j_{U_0}(\alpha)$. In the other case, if $\alpha=\kappa$, then we set $\beta=\kappa^{U_0}$. 
Then $\kappa^{U_0}$ is a $\Psi^{\kappa^{U_0}}_\beta$-Ramsey cardinal with $\NN\Psi^{\kappa^{U_0}}_\beta(\kappa^{U_0})\notin\II\Psi^{\kappa^{U_0}}_\beta(\kappa^{U_0})$ in $\Ult{M_0}{U_0}$. 
 Hence, in $\Ult{M_0}{U_0}$, there is a weak $\kappa^{U_0}$-model $\bar{M}$ closed under ${<}\beta$-sequences and a uniform, $\kappa^{U_0}$-amenable $\bar{M}$-normal $\bar{M}$-ultrafilter $\bar{U}$  such that $j_{U_0}^*(x)\in\bar{M}$, $\NN\Psi^{\kappa^{U_0}}_\beta(\kappa^{U_0})\in\bar{U}$ and $\Psi(\bar{M},\bar{U})$ holds. 
 Pick $M_1,U_1\in\HH{\kappa^+}^{M_0}$ with $j_{U_0}^*(M_1)=\bar{M}$ and $j_{U_0}^*(U_1)=\bar{U}$. 
 Then $M_1$ is a weak $\kappa$-model closed under ${<}\alpha$-sequences, $U_1$ is a uniform, $\kappa$-amenable and $M_1$-normal $M_1$-ultrafilter on $\kappa$ and our assumptions on $\Psi$ imply that $\Psi(M_1,U_1)$ holds. 
 Moreover, we have $x\in M_1$ and, since $j_{U_0}^*(\NN\Psi^\kappa_\alpha(\kappa))=(\NN\Psi^{\kappa^{U_0}}_\beta(\kappa^{U_0}))^\NN$, we know that $\NN\Psi^\kappa_\alpha(\kappa)\in U_1$. 
 But this shows that $M_1$ and $U_1$ witness that $\NN\Psi^\kappa_\alpha(\kappa)\notin\II\Psi^\kappa_\alpha(\kappa)$, a contradiction.

 The case $\vartheta=\kappa^+$ works analogously, using the observation that, if $M_0\prec\HH{\kappa^+}$ is a weak $\kappa$-model, $U_0$ is a uniform, $\kappa$-amenable and $M_0$-normal $M_0$-ultrafilter on $\kappa$, $\bar{M}\prec\HH{(\kappa^{U_0})^+}$ is a weak $\kappa^{U_0}$-model in $\Ult{M_0}{U_0}$ and $M_1\in\HH{\kappa^+}^{M_0}$ with $j_{U_0}^*(M_1)=\bar{M}$, then $M_1$ is a weak $\kappa$-model with $M_1\prec\HH{\kappa^+}$.  
\end{proof}

The above lemma directly yields the related parts  of Theorem \ref{theorem:IdealContain}.(\ref{item:cont:wr}), \ref{theorem:IdealContain},(\ref{item:IdealContain:stR}) and \ref{theorem:IdealContain},(\ref{item:IdealContain:suR}). 
 It also provides the corresponding statements for $\beta$-iterable, super weakly Ramsey and super Ramsey cardinals.

\begin{corollary}\label{corollary:RamseyIdealsNotContained}
 Let $\alpha\leq\kappa\leq\vartheta$ be cardinals with $\vartheta\in\{\kappa,\kappa^+\}$. 
 \begin{enumerate} 
  \item\label{item:RamseyIdealsNotContained:1}  If $\kappa$ is a $\triv^\vartheta_\alpha$-Ramsey cardinal, then $\NN\triv^\vartheta_\alpha(\kappa)\notin\II\triv^\vartheta_\alpha(\kappa)$. 

   \item\label{item:RamseyIdealsNotContained:2} If $\kappa$ is a $\wf^\vartheta_\alpha$-Ramsey cardinal, then $\NN\wf^\vartheta_\alpha(\kappa)\notin\II\wf^\vartheta_\alpha(\kappa)$. 

  \item If $\kappa$ is a $\wf\beta^\vartheta_\alpha$-Ramsey cardinal with $\beta\leq\omega_1$, then $\NN\wf\beta^\vartheta_\alpha(\kappa)\notin\II\wf\beta^\vartheta_\alpha(\kappa)$.  \qed 
 \end{enumerate}
\end{corollary}

%%%%%%%%%%%%%%%%%%%%%%%%%%%%%%%%%%%%%%%%%%
%%%%%%%%%%%%%%%%%%%%%%%%%%%%%%%%%%%%%%%%%%

\section{The bottom of the Ramsey-like hierarchy}\label{section:bottom}

The weakest principles that can be extracted from the general definitions of the previous section are the $\triv_\omega^\kappa$-Ramsey and the $\triv_\omega^{\kappa^+}$-Ramsey cardinals. 
 It already follows from Theorem \ref{theorem:weaklycompact} that if $\kappa$ is $\triv_\omega^\kappa$-Ramsey, then $\kappa$ is weakly compact. Moerover, it is trivial to check that whenever $\kappa$ is a $\triv_\omega^\kappa$-Ramsey cardinal, then $\II\triv_\omega^\kappa(\kappa)$, the smallest of our Ramsey-like ideals, is a superset of the ideal $\II_{\delta}^\kappa$.

\begin{lemma}\label{lemma:TwkRamsey}
  If $\kappa$ is a $\triv^\kappa_\omega$-Ramsey cardinal, then $\II_{wie}^\kappa\cup\{\NN^\kappa_{ie}\}\subseteq\II\triv_\omega^\kappa(\kappa)$, $\NN\triv^\kappa_\omega(\kappa)\notin\II\triv^\kappa_\omega(\kappa)$ and $\II_{ie}^\kappa\not\subseteq\II\triv_\omega^\kappa(\kappa)$.
\end{lemma}

\begin{proof}
 First, let $A\subseteq\kappa$ be $\triv_\omega^\kappa$-Ramsey and fix an $A$-list $\vec{d}=\seq{d_\alpha}{\alpha\in A}$. 
  Pick a weak $\kappa$-model $M$ with $\vec{d}\in M$ and a $\kappa$-amenable, $M$-normal $M$-ultrafilter $U$ on $\kappa$ with  $A\in U$. Set $N=\Ult{M}{U}$. Since $j_U$ is $\kappa$-powerset preserving, the set $D=\Set{\alpha<\kappa}{j_U(\alpha)\IN_N(j_U(\vec{d})_{\kappa^N})^N}$ is an element of $M$. 
   Then $\Set{\alpha\in A}{D\cap\alpha=d_\alpha}\in U$ and, since $U$ is uniform, we can conclude that $A$ is weakly ineffable.  
    These computations show that $\kappa$ is weakly ineffable with $\II_{wie}^\kappa\subseteq\II\triv_\omega^\kappa(\kappa)$. Moreover, Corollary \ref{corollary:RamseyIdealsNotContained} directly shows that $\NN\triv^\kappa_\omega\notin\II\triv^\kappa_\omega(\kappa)$. 

  Next, assume that $\NN^\kappa_{ie}\notin\II\triv_\omega^\kappa(\kappa)$. 
  Then, there is a transitive weak $\kappa$-model $M$ and a $\kappa$-amenable, $M$-normal $M$-ultrafilter $U$ on $\kappa$ such that $\NN^\kappa_{ie}\in U$. 
  Now, for every $\kappa$-size collection of subsets of $\kappa$ in $M$, we can use $U$ to find a normal ultrafilter on that collection in $M$. In particular,  $\kappa$ is ineffable in $M$ (see \cite[Corollary 1.3.1]{MR0460120} or Theorem \ref{theorem:ineffableideal}). 
  By the $\kappa$-powerset preservation of the embedding $j_U$, the fact that the ineffability of $\kappa$ is a property of $\VV_{\kappa+1}$ implies that $\kappa^U$ is ineffable in $\langle\Ult{M}{U},\IN_U\rangle$. 
   On the other hand, we have $\kappa^U\IN_U j_U(\NN^\kappa_{ie})$, yielding that $\kappa^U$ is not ineffable in $\Ult{M}{U}$, a contradiction. 

  Finally, if $\kappa$ is not ineffable, then the remarks following Definition \ref{definition:Ideals} show that $\kappa\in\II^\kappa_{ie}\setminus\II\triv^\kappa_\omega(\kappa)$. Hence, we may assume that $\kappa$ is ineffable. 
 Since $\triv^\kappa_\omega$-Ramseyness is a $\Pi^1_2$-property and the classical argument of Jensen and Kunen in \cite{jensennotes} proving the $\Pi^1_2$-indescribability of ineffable cardinals shows that, given an a $\Pi^1_2$-statement $\Omega$ that holds in $\VV_\kappa$, the set of all non-reflection points of $\Omega$ in $\kappa$ is not ineffable, we can use Theorem \ref{theorem:ineffableideal} to conclude that $\NN\triv^\kappa_\omega\in\II^\kappa_{ie}\setminus\II\triv^\kappa_\omega(\kappa)$. 
\end{proof}

\begin{proposition}\label{proposition:TwkplusRamsey}
  If $\kappa$ is $\triv_\omega^{\kappa^+}$-Ramsey, then $\II_{ie}^\kappa\subseteq\II\triv_\omega^{\kappa^+}(\kappa)$ and $\NN\triv^{\kappa^+}_\omega(\kappa)\notin\II\triv^{\kappa^+}_\omega(\kappa)$.
\end{proposition}

\begin{proof}
  The first statement is proven exactly as the related part of Lemma \ref{lemma:TwkRamsey}, additionally using that, by elementarity, every element of $U$ is stationary. 
%  It suffices to show that if $A\subseteq\kappa$ is $\mathbf T_\omega^{\kappa^+}$-Ramsey, then $A$ is ineffable. Given such $A$, and an $A$-list $\vec x=\langle x_\alpha\mid\alpha\in A\rangle$, let $M\prec\HH{\kappa^+}$ be a weak $\kappa$-model with $\vec x\in M$ and with a $\kappa$-amenable, $M$-normal $M$-ultrafilter $U$ on $\kappa$ such that $A\in U$. Let $j\colon M\to\langle N,\epsilon_N\rangle$ be the induced ultrapower embedding, and let $X=j(\vec x)_{\kappa^N}$. It follows that $\{\alpha\in A\mid j^{-1}[X]\cap\alpha=x_\alpha\}\in U$. Since by elementarity, every element of $U$ is stationary, we have shown that $A$ is ineffable.
 The second statement follows from Corollary \ref{corollary:RamseyIdealsNotContained}. 
\end{proof}

%%%%%%%%%%%%%%%%%%%%%%%%%%%%%%%%%%%%%%%%%%%%%%%
%%%%%%%%%%%%%%%%%%%%%%%%%%%%%%%%%%%%%%%%%%%%%%%

\section{Completely ineffable cardinals}\label{section:completelyineffable}

We start by recalling the definition of complete ineffability.

\begin{definition}\label{definition:completelyineffable}
 Let $\kappa$ be an uncountable regular cardinal. 
 \begin{enumerate} 
  \item A nonempty collection $\mathcal{S}\subseteq\POT{\kappa}$ is a \emph{stationary class} if the following statements hold: 
    \begin{enumerate}
      \item Every $A\in\mathcal S$ is a stationary subset of $\kappa$. 
      
      \item If $A\in\mathcal S$ and $A\subseteq B\subseteq \kappa$, then $B\in\mathcal S$.
    \end{enumerate}
    
   \item A subset $A$ of $\kappa$ is \emph{completely ineffable} if there is a stationary class $\mathcal S\subseteq\POT{\kappa}$ with $A\in\mathcal S$ and the property that for every $S\in\mathcal S$ and every function $\map{c}{[S]^2}{2}$, there is $H\in\mathcal S$ that is homogeneous for $f$. 
   
   \item The cardinal $\kappa$ is \emph{completely ineffable} if the set $\kappa$ is completely ineffable in the above sense.
 \end{enumerate}
\end{definition}

It is trivial to check that if there exists a stationary class $\mathcal S\subseteq\mathcal P(\kappa)$ witnessing the complete ineffability of $\kappa$, then the union of all such stationary classes is again a stationary class witnessing the complete ineffability of $\kappa$, and it is therefore the unique \emph{maximal stationary class} that does so. From \cite[Corollary 3]{MR853844} and its proof, and from the definition of the completely ineffable ideal in \cite{MR918427}, it is immediate that the completely ineffable ideal is the complement of this maximal stationary class.
 The following lemma is an easy adaption of Kunen's result that ineffability can be characterized either in terms of homogeneous sets for colourings or for lists (see \cite[Theorem 4]{jensennotes}).\footnote{\cite[Theorem 4]{jensennotes} also provides a characterization of ineffability in terms of regressive colourings. An analogous result would be possible for complete ineffability, however we do not need this in our paper, and hence omitted to present it.} It is probably a folklore result, and its substantial direction is implicit in the proof of \cite[Theorem 3.12]{nielsen-welch}.

\begin{lemma}\label{lemma:cewrtlists}
  A stationary class $\mathcal S\subseteq\POT{\kappa}$ with $0\not\in\bigcup\mathcal S$\footnote{Note that this is a harmless extra assumption, for if some stationary class witnesses $A\subseteq\kappa$ to be completely ineffable, then there is such a stationary class $\mathcal S$ which also satisfies $0\not\in\bigcup\mathcal S$.} witnesses that $A\subseteq\kappa$ is completely ineffable if and only if $\mathcal S$ witnesses $A$ to be \emph{completely ineffable with respect to lists}, in the sense that    $A\in\mathcal S$ and for every $S\in\mathcal S$, and every $S$-list $\vec{d}=\seq{d_\alpha}{\alpha\in S}$, there is $K\in\mathcal S$ with $d_\alpha=d_\beta\cap\alpha$ for all $\alpha,\beta\in K$ with $\alpha<\beta$. 
\end{lemma}

\begin{proof}
  First, assume  that the stationary class $\calS$ witnesses that $A\subseteq\kappa$ is completely ineffable.
  Pick $S\in\calS$, and an $S$-list $\vec{d}=\seq{d_\alpha}{\alpha\in S}$. 
    Order the bounded subsets of $\kappa$ by letting $a\prec b$ if there is an $\alpha<\kappa$ such that $a\cap\alpha=b\cap\alpha$ and $\alpha\in b\setminus a$.\footnote{Note that for bounded subsets $a$ and $b$ of $\kappa$, either $a\prec b$, or $b\prec a$,  or $a=b$ holds.}
    Define a colouring $\map{c}{[\kappa]^2}{2}$ by setting, for $\alpha<\beta$, $c(\{\alpha,\beta\})=1$ in case $d_\alpha\prec d_\beta$ or $d_\alpha=d_\beta$, and setting $c(\{\alpha,\beta\})=0$ otherwise. 
    Let $H\in\calS$ be homogeneous for $c$. Since we cannot have a descending $\kappa$-sequence in the ordering $\prec$, it follows that $c$ takes constant value $1$ on $H$. If the sequence $\seq{d_\alpha}{\alpha\in H}$ is eventually constant, then some final segment $K$ of $H$ is as desired.
    Otherwise, for every $\xi<\kappa$, consider the sequence $\seq{d_\alpha\cap\xi}{\alpha\in H, ~ \alpha>\xi}$. 
    Since this is a weakly $\prec$-increasing $\kappa$-sequence of subsets of $\xi$, we can define a function $\map{f}{\kappa}{\kappa}$ by letting $f(\xi)$ be the minimal $\eta\geq\xi$ such that $d_\alpha\cap\xi=d_\beta\cap\xi$ whenever $\alpha,\beta\in H$ with $\eta\leq\alpha<\beta$. 
    Then,      $f$ is a continuous, increasing function that maps $\kappa$ cofinally into $\kappa$. 
    Let $C$ be the closed unbounded subset of $\kappa$ of fixed points of $f$.  
    Whenever $\zeta\in C\cap H$, we have in particular that $d_\zeta=d_\alpha\cap\zeta$ for every $\alpha>\zeta$ in $H$. But it is also easy to see that $\calS$ is closed under intersections with closed unbounded subsets of $\kappa$. Thus, we have $K=C\cap H\in\calS$, and hence $\calS$ witnesses $A$ to be completely ineffable with respect to lists, as desired.

 In the other direction, let $\calS$ be a stationary class with $0\not\in\bigcup\calS$ that witnesses that $A$ is completely ineffable with respect to lists.
 Pick $S\in\calS$, and a colouring $\map{c}{[S]^2}{2}$. Define an $S$-list $\vec{d}=\seq{d_\alpha}{\alpha\in S}$ by setting $d_\alpha=\Set{\beta<\alpha}{c(\{\alpha,\beta\})=1}$. 
  By our assumption, we find $H\in\calS$ such that $d_\alpha=d_\beta\cap\alpha$ whenever $\alpha<\beta$ are both elements of $H$. 
  Let $\map{f}{H}{2}$ be defined by setting $f(\alpha)=1$ if and only if $\alpha\in d_\beta$ for some (equivalently, for all) $\beta>\alpha$ in $H$. Now, define an $H$-list $\vec{e}=\seq{e_\alpha}{\alpha\in H}$ by setting $e_\alpha=\alpha$ in case $f(\alpha)=1$, and setting $e_\alpha=\emptyset$ otherwise. 
  By our assumption, we find $K\in\calS$ such that $f$ is homogeneous on $K$. Assume that $f$ takes value $i\in\{0,1\}$ on $K$. Then, if $\alpha<\beta$ are both elements of $K$, our definitions yield that $c(\{\alpha,\beta\})=i$, i.e.\ $K\in\mathcal S$ is homogeneous for $c$, as desired. 
\end{proof}

The following result is the crucial link connecting completely ineffable and Ramsey-like cardinals, and in particular implies Theorem \ref{theorem:Ideals} (\ref{item:Ideals:CI}), by the above and by Lemma \ref{lemma:filtervsramsey}. In particular, it shows that a cardinal is completely ineffable if and only if it is $\triv_\omega^\forall$-Ramsey. Its proof is a generalization, adaption and simplification of \cite[Theorem 3.12]{nielsen-welch}.

\begin{theorem}\label{nielsen}
 Given an uncountable regular cardinal $\kappa$, a subset   $A$ of $\kappa$ is completely ineffable if and only if $\kappa=\kappa^{{<}\kappa}$ holds and $A$ has the $\triv_\omega^\forall$-filter property. %, that is, the $\mathbf T_\omega^\forall$-filter property. %\footnote{Note that the game $G_\omega^\theta(\kappa)$ is open and hence determined. Therefore, the $\omega$-filter property is easily seen to be the same as what is called coherent ${<}\omega$-Ramseyness in \cite{nielsen-welch}.}
\end{theorem}

\begin{proof}
  First, assume that $A$ has the $\triv_\omega^\forall$-filter property, and  $\kappa=\kappa^{{<}\kappa}$ holds. 
   Let $\theta>\kappa$ be regular. By Lemma \ref{lemma:filtervsramsey}, $A$ is $\triv_\omega^\theta$-Ramsey. 
   Let $\mathcal S$ denote the collection of all subsets of $\kappa$ which are $\triv_\omega^\forall$-Ramsey.
   Then, $A\in \mathcal S$ and $\mathcal S$ is a stationary class. 
   Pick $X\in\mathcal S$ and $\map{c}{[X]^2}{2}$. 
   Pick $M\prec\HH{\theta}$ with $c\in M$, and an $M$-normal, $\kappa$-amenable $M$-ultrafilter $U$ on $\kappa$ with $X\in U$. 
   Let $\vec{Y}=\seq{Y_\alpha}{\alpha\in X}$ be defined by setting $Y_\alpha=\Set{\beta>\alpha}{c(\{\alpha,\beta\})=0}$.
    Define $\vec Z=\langle Z_\alpha\mid\alpha\in X\rangle$ by setting $Z_\alpha=Y_\alpha$ in case $Y_\alpha\in U$, and let $Z_\alpha=X\setminus Y_\alpha\in U$ otherwise. Then, $\vec Z\subseteq U$ implies $\Delta\vec Z\in U$. Let $H\in U$ either be $\Delta\vec Z\cap\Set{\alpha\in X}{Y_\alpha\in U}$ or $\Delta\vec Z\cap\Set{\alpha\in X}{Y_\alpha\notin U}$. 
     Then, it is easy to check that $H\subseteq X$ is homogeneous for $c$. Moreover, we have $H\in U\subseteq\mathcal S$ and hence $\calS$ is a stationary class witnessing that $A$ is completely ineffable.

 For the reverse direction, assume that $A\subseteq\kappa$ is completely ineffable, as witnessed by the stationary class $\mathcal S$. 
  Let $\theta>\kappa$ be a regular cardinal. We describe a strategy for the {\sf Judge} in the game $G\triv_\omega^\theta(A)$. As required by the rules of this game, the {\sf Challenger} and the {\sf Judge} take turns playing $\kappa$-models $M_n$ and $M_n$-ultrafilters $U_n$. We let the {\sf Judge} also pick, in each step $n<\omega$, an enumeration $\vec{X}^n=\seq{X_\xi^n}{\xi<\kappa}$ of $\POT{\kappa}\cap M_n$ and a set $H_n\in\mathcal S$ such that the following hold (the first two items are required by the rules of the game $G\triv_\omega^\theta(A)$): 
\begin{itemize} 
  \item $A\in U_0$,
  \item If $n>0$, then $U_n\supseteq U_{n-1}$ and $H_n\subseteq H_{n-1}$,
  \item $X_\xi^n\in U_n$ if and only if  $\gamma\in X_\xi^n$ for all $\xi<\gamma\in H_n$ if and only if $\gamma\in X_\xi^n$ for some $\xi<\gamma\in H_n$.
\end{itemize}
Assume that we have done this for $m<n$. We want to define the above objects at stage $n$. For the sake of a uniform argument when $n=0$, let $H_{-1}=A$, and let $C_{-1}=\kappa$. 
 For $\alpha\in H_{n-1}$, define $r_\alpha\subseteq\alpha$ by letting $\xi\in r_\alpha$ if and only if  $\alpha\in X_\xi^n$. 
  By Lemma \ref{lemma:cewrtlists}, we find a stationary set $H_n\subseteq H_{n-1}$ in $\mathcal S$ on which the $r_\alpha$'s cohere, that is, for every $\alpha<\beta$ in $H_n$, $r_\alpha=r_\beta\cap\alpha$. 
  But this means that for $\alpha<\beta\in H_n$, and $\xi<\alpha$, $\alpha\in X_\xi^n$ if and only if $\beta\in X_\xi^n$. This shows that if we now define $U_n$ using $H_n$ as required above, it will satisfy the required equivalence. In particular, this implies that $H_n\subseteq\Delta U_n$, making $U_n$ a normal $M_n$-measure.

It remains to show that $U_{n-1}\subseteq U_n$ in case $n>0$. Thus, let $X\in U_{n-1}$ be given, say $X=X_\xi^n=X_\zeta^{n-1}$. By the definition of $U_{n-1}$, every $\zeta<\gamma\in H_{n-1}$ is an element of $X$. In particular, we find some such $\gamma>\xi$ in $H_n$, witnessing that $X\in U_n$, as desired.
\end{proof}

We are now ready to generalize Kleinberg's result from \cite{MR513844}. Given the above, the following is now an easy consequence of Theorem \ref{theorem:filtervsramsey}, and in particular implies Theorem \ref{theorem:SchemesSummary} (\ref{item:schemeSmall:CI}).

\begin{theorem}\label{theorem:completelyineffable}
  Given an uncountable cardinal $\kappa$, the following statements are equivalent for all regular $\theta>2^\kappa$  and all $\lambda\leq\kappa$: 
  \begin{enumerate} 
    \item The cardinal $\kappa$ is completely ineffable. 
    
    \item For many $(\lambda,\kappa)$-models $M\prec\HH{\theta}$, there exists a uniform, $\kappa$-amenable $M$-normal $M$-ultrafilter on $\kappa$. 
    
    \item For many $(\lambda,\kappa)$-models $M\prec\HH{\theta}$, there exists a $\kappa$-powerset preserving $\kappa$-embedding $\map{j}{M}{\langle N,\IN_N\rangle}$.
  \end{enumerate}
\end{theorem}

\begin{proof}
  That (1) implies (2) is immediate from Theorem \ref{theorem:filtervsramsey} and Theorem \ref{nielsen}. 
  The equivalence between (2) and (3) follows from Corollary \ref{corollary:correspondence1}. 
  Now, assume that (2) holds. Then Lemma \ref{lemma:inaccessible2} implies that $\kappa$ is inaccessible. 
  Moreover, observe that the proof of the implication from (2) to (1) in Theorem \ref{theorem:filtervsramsey} shows that $\kappa$ has the 
  %$(\omega,\kappa^+)$-filter property. 
  $\triv_\omega^{\kappa^+}$-filter property. By Lemma \ref{lemma:filterequivalence}, it follows that $\kappa$ is completely ineffable.
\end{proof}

In the remainder of this section, we verify Theorem \ref{theorem:IdealContain} (\ref{item:cont:ci}).

\begin{lemma}
  If $\kappa$ is completely ineffable, then $\II_{ie}^\kappa\cup\{N_{ie}^\kappa\}\subseteq\II_{{\prec}{cie}}^\kappa$ and $\NN_{cie}^\kappa\notin\II_{{\prec}{cie}}^\kappa$.
\end{lemma}

\begin{proof}
  The first statement is immediate from Lemma \ref{lemma:TwkRamsey}, Proposition \ref{proposition:TwkplusRamsey} and Proposition \ref{proposition:compareideals}.
%  Assume first that $A\notin\II_{cie}^\kappa$, and let $M$ be a weak $\kappa$-model with $A\in M$. Making use of our assumption, we find a weak $\kappa$-model $N\prec\HH{\theta}$ with a $\kappa$-amenable $M$-normal $M$-ultrafilter $U$, and with $M\in N$ and $A\in U$. Then, $\Delta(U\cap M)$ is stationary, yielding that $A\notin\II_{ie}^\kappa$, that is, we have shown that $\II_{ie}^\kappa\subseteq\II_{cie}^\kappa$.
%
%Now, assume for a contradiction that $\NN_{ie}^\kappa\notin\II_{cie}^\kappa$. Then, we find a weak $\kappa$-model $M\prec\HH{\theta}$ with a $\kappa$-amenable $M$-normal $M$-ultrafilter $U$ such that $\NN_{ie}^\kappa\in U$. 
% %
% Then $\kappa^U$ is not ineffable in $\Ult{M}{U}$. However, $\kappa$ is ineffable in $M$ by elementarity of $M$, and hence $\kappa^U$ is ineffable in $\Ult{M}{U}$ by the $\kappa$-powerset preservation of $j_U$, a contradiction.

Assume for a contradiction that $\NN_{cie}^\kappa\in\II_{{\prec}cie}^\kappa$, and assume that $\kappa$ is the least completely ineffable cardinal with this property. 
 Then there exists a regular cardinal $\theta>\kappa$, a weak $\kappa$-model $M\prec\HH{\theta}$ with $\HH{\kappa}\in M$ and a $\kappa$-amenable, $M$-normal $M$-ultrafilter $U$ with $\NN_{cie}^\kappa\notin U$. 
 It follows that $\kappa^U$ is completely ineffable in $\Ult{M}{U}$, however $\NN_{cie}^{\kappa^U}\notin\II_{{\prec}cie}^{\kappa^U}$ in $\Ult{M}{U}$ by elementarity of $j_U$ and by our minimality assumption on $\kappa$. 
  Let $\mathcal S$ be the maximal stationary class witnessing that $\kappa^U$ is completely ineffable in $\Ult{M}{U}$, that is, the complement of $\II_{{\prec}cie}^{\kappa^U}$ in $\Ult{M}{U}$. 
  The iterative construction of the maximal stationary class $\mathcal T$ witnessing that $\kappa$ is completely ineffable in $M$ (see \cite{MR513844}) together with the $\kappa$-powerset preservation of $j_U$ easily yields that the $j_U$-preimage of the collection of elements of $\mathcal S$ in $\Ult{M}{U}$ is contained in $\mathcal T$. However, this yields that $\NN_{cie}^\kappa\in\mathcal T$, and since $\mathcal T$ is the complement of $\II_{{\prec}cie}^\kappa$ in $M$, this is clearly a contradiction.
\end{proof}

%%%%%%%%%%%%%%%%%%%%%%%%%%%%%%%%%%%%%%%%%%%%%%%
%%%%%%%%%%%%%%%%%%%%%%%%%%%%%%%%%%%%%%%%%%%%%%%

\section{Weakly Ramsey cardinals, Ramsey cardinals and ineffably Ramsey cardinals}\label{section:ramsey}

We start this section by proving several statements from Theorem \ref{theorem:IdealContain} (\ref{item:cont:wr}).

\begin{lemma}\label{lemma:StructuralPropertiesWRIdeals}
  If $\kappa$ is a weakly Ramsey cardinal, then $\II^\kappa_{wie}\cup\{\NN^\kappa_{cie}\}\subseteq\II^\kappa_{\textup{w}R}$, $\NN^\kappa_{\textup{w}R}\notin\II^\kappa_{\textup{w}R}$ and $\II^\kappa_{ie}\nsubseteq\II^\kappa_{\textup{w}R}$. 
\end{lemma}

\begin{proof}
 First, note that, since the properties $\triv$ and $\wf$ remain true under restrictions, we can combine Proposition  \ref{proposition:compareideals} and Lemma \ref{lemma:TwkRamsey} to conclude that $\II^\kappa_{wie}\subseteq\II\triv^\kappa_\omega(\kappa)\subseteq\II\wf^\kappa_\omega(\kappa)=\II^\kappa_{\textup{w}R}$.  
 Moreover, the proof of {\cite[Theorem 3.7]{MR2830415}} directly shows that $\NN^\kappa_{cie}\in\II^\kappa_{\textup{w}R}$. 
 In addition, Corollary \ref{corollary:RamseyIdealsNotContained}.(\ref{item:RamseyIdealsNotContained:2}) directly implies that $$\NN^\kappa_{\textup{w}R} ~ = ~ \NN\wf^\kappa_\omega(\kappa) ~ \notin ~ \II\wf^\kappa_\omega(\kappa) ~ = ~ \II^\kappa_{\textup{w}R}.$$  
 Finally, since weak Ramseyness is a $\Pi^1_2$-property, the argument used in the last part of the proof of Lemma \ref{lemma:TwkRamsey} also shows that $\II^\kappa_{ie}\nsubseteq\II^\kappa_{\textup{w}R}$. 
\end{proof}

In \cite[Theorem 1.3]{MR2830415} and \cite[Theorem 5.1]{MR2817562}, isolating from folklore results (see for example \cite{MR534574}), Gitman, Sharpe and Welch have shown that a cardinal $\kappa$ is Ramsey if and only if (in our notation) it is $\cc_\omega^\kappa$-Ramsey. 
 In \cite{MR0540770},  Baumgartner introduced the notion of \emph{ineffably Ramsey} cardinals.% These are defined just like Ramsey cardinals, except that the homogeneous set is required to be a stationary subset of $\kappa$.

\begin{definition}
  A cardinal $\kappa$ is \emph{ineffably Ramsey} if and only if every function $\map{c}{[\kappa]^{{<}\omega}}{2}$ has a homogeneous set that is stationary in $\kappa$. 
\end{definition}

%We will show that Gitmans arguments easily adapt to show that $\kappa$ is ineffably Ramsey if and only if it is $sc_\omega^\kappa$-Ramsey -- this will be an immediate consequence of our below result on the ineffably Ramsey ideal. Let us first show that Gitmans arguments can also be adapted to provide a result on the Ramsey ideal.

 In \cite{MR0540770}, Baumgartner also introduced the \emph{Ramsey ideal} and the \emph{ineffably Ramsey ideal} at $\kappa$, which can be described as follows (see also \cite{MR1077260}). 
A subset $A\subseteq\kappa$ is \emph{Ramsey} if every regressive function function $\map{c}{[A]^{{<}\omega}}{\kappa}$ has a homogeneous set of size $\kappa$. 
The \emph{Ramsey ideal} on $\kappa$  is the collection of all subsets of $\kappa$ that are not Ramsey. 
Moreover, a subset  $A\subseteq\kappa$ is \emph{ineffably Ramsey} if every regressive function $\map{c}{[A]^{{<}\omega}}{\kappa}$ has a homogeneous set that is stationary in $\kappa$ and the  \emph{ineffably Ramsey ideal} on $\kappa$ is the collection of all subsets of $\kappa$ that are not ineffably Ramsey. 

\medskip

The same argument as for \cite[Theorem 2.10]{brent} yields Item (1) of the following. A completely analogous argument, replacing unboundedness by stationarity throughout, then verifies Item (2) below, showing in particular that $\kappa$ is ineffably Ramsey if and only if it is $\scc_\omega^\kappa$-Ramsey, yielding Theorem \ref{theorem:SchemesSummary}.(\ref{item:CharIneffRamsey}).  
 %
% For its proof, which we will omit, one simply needs to replace \lanf\emph{unbounded subset of $\kappa$}\ranf\ by \lanf\emph{stationary subset of $\kappa$}\ranf, and \lanf\emph{countably complete\emph}\ranf\ by \lanf\emph{stationary-complete}\ranf\ throughout.

\begin{proposition}
  \begin{enumerate}
    \item If $\kappa$ is a Ramsey cardinal, then $\II^\kappa_R=\II\cc_\omega^\kappa(\kappa)$ is the Ramsey ideal on $\kappa$.
    \item If $\kappa$ is ineffably Ramsey, then $\II^\kappa_{iR}=\II\scc_\omega^\kappa(\kappa)$ is the ineffably Ramsey ideal on $\kappa$.
  \end{enumerate}
\end{proposition}

%\begin{proof}
%  Let $\kappa$ be a Ramsey cardinal, and let $A\subseteq\kappa$ be such that $A\notin\II\cc_\omega^\kappa(\kappa)$. 
%  %
%  Adapting the proof of {\cite[Theorem 3.10]{MR2830415}}, we want to show that $A$ is Ramsey. Let $\map{c}{[A]^{{<}\omega}}{\kappa}$ be a regressive function, and let $M$ be a weak $\kappa$-model containing both $A$ and $c$ as elements, such that there exists a countably complete, $M$-normal and $\kappa$-amenable $M$-ultrafilter $U$ on $\kappa$. 
%   %
%  By \cite[Lemma 3.6]{MR2830415}, for every $n\in\omega$, there is $H_n\in U$ that is homogeneous for $f\restriction[A]^n$. Using that $U$ is countably complete, we obtain an unbounded subset of $\kappa$ witnessing that $A$ is Ramsey. 
%
%  For the reverse direction, assume that $A\subseteq\kappa$ is Ramsey. Adapting the arguments from {\cite[Section 4]{MR2830415}}, we want to show that $A\notin\II\cc_\omega^\kappa(\kappa)$. 
%  %
%  Pick $x\subseteq\kappa$. Using Lemma \ref{lemma:indiscerniblesramsey}, let $I\subseteq A$ be a set of good indiscernibles for the structure $\langle L_\kappa[A,x],A,x\rangle$. Now we construct a weak $\kappa$-model $M$ and an $M$-normal, countably complete ultrafilter $U$ on $\kappa$ exactly as in {\cite[Section 4]{MR2830415}}. 
%  %
%   By \cite[Lemma 4.2.12]{MR2830415}, it follows that $A\in M$, and, by {\cite[Lemma 4.2.10]{MR2830415}}, $A\in U$, showing that $A\notin\II\cc_\omega^\kappa(\kappa)$, as desired.
%\end{proof}

Finally, using results from \cite{MR1077260} and \cite{MR513844}, we verify several statements from Theorem \ref{theorem:IdealContain} (\ref{item:cont:r}) and (\ref{item:cont:ir}).

\begin{lemma}
 If $\kappa$ is a Ramsey cardinal, then $\NN^\kappa_R\notin\II^\kappa_R$ and $\II^\kappa_{ie}\nsubseteq\II^\kappa_R$. 
\end{lemma}

\begin{proof}
 The first statement follows directly from {\cite[Theorem 4.5]{MR1077260}}. 
  Since Ramseyness is a $\Pi^1_2$-property, the argument used in the last part of the proof of Lemma \ref{lemma:TwkRamsey} also shows that $\II^\kappa_{ie}\nsubseteq\II^\kappa_R$. 
\end{proof}

\begin{lemma}
  If $\kappa$ is an ineffably Ramsey cardinal, then $\NN^\kappa_{iR}\notin\II^\kappa_{iR}$ and $\II^\kappa_{{\prec}cie}\nsubseteq\II^\kappa_{iR}$. 
\end{lemma}

\begin{proof}
 The first statement again follows directly from {\cite[Theorem 4.5]{MR1077260}}. 
  For the second statement, we may assume that $\kappa$ is completely ineffable, because otherwise the remarks following Definition \ref{definition:Ideals} show that $\kappa\in\II^\kappa_{{\prec}cie}\setminus\II^\kappa_{iR}$. 
 Then the proof of {\cite[Theorem 4]{MR513844}} shows that, given %a completely ineffable cardinal $\kappa$ and 
 a $\Sigma^2_0$-statement $\Omega$ that holds in $\VV_\kappa$, the set of all non-reflection points of $\Omega$ in $\kappa$ is not completely ineffable. Since ineffable Ramseyness is $\Pi^1_3$-definable, the results of Section \ref{section:completelyineffable} now show that $\NN^\kappa_{iR}\in\II^\kappa_{{\prec}cie}\setminus\II^\kappa_{iR}$. 
\end{proof}

%%%%%%%%%%%%%%%%%%%%%%%%%%%%%%%%%%%%%%%%%
%%%%%%%%%%%%%%%%%%%%%%%%%%%%%%%%%%%%%%%%%

\section{$\Delta_\omega^\forall$-Ramsey cardinals}\label{section:gnramsey}

In this section we provide the short and easy proof that -- perhaps somewhat surprisingly -- the notions of $\scc_\omega^\forall$-Ramsey, $\infty_\omega^\forall$-Ramsey, and $\delt_\omega^\forall$-Ramsey cardinals are equivalent.\footnote{In particular, this contrasts the hierarchy of large cardinals treated in \cite[Section 3]{nielsen-welch}.} 
 Together with Theorem \ref{theorem:filtervsramsey}, this result shows why $\Delta_\omega^\forall$-Ramseyness appears three times in  Table \ref{table:schemeKappa}, yielding Theorem \ref{theorem:SchemesSummary} (\ref{item:schemeSmall:nR}), and in particular completes the tables presented in our introductory section.

\begin{proposition}\label{RamseysEquivalence}
  Let $\kappa$ be a cardinal. Then $\kappa$ is $\scc_\omega^\forall$-Ramsey if and only if $\kappa$ is $\infty_\omega^\forall$-Ramsey if and only if $\kappa$ is $\Delta_\omega^\forall$-Ramsey.
\end{proposition}

\begin{proof}
  Assume that $\kappa$ is $\scc_\omega^\forall$-Ramsey, and let $A\subseteq\kappa$. Pick a sufficiently large regular cardinal $\theta$, and let $M_0\prec\HH{\theta}$ with $A\in M_0$ be a weak $\kappa$-model. Consider a run of the game $G\scc_\omega^\theta(\kappa)$, in which the {\sf Challenger} starts by playing $M_0$. As the {\sf Challenger} has no winning strategy in this game, there is a run of this game which is won by the {\sf Judge}. Let $M=M_\omega$ and $F=F_\omega$ be the final model and filter produced by this run. This means that $M\prec\HH{\theta}$ is a weak $\kappa$-model with $A\in M$, and that $F$ is a $\kappa$-amenable, $M$-normal and stationary-complete $M$-ultrafilter. It is now easy to verify that the set $(\bigcap_{i<\omega}\Delta F_i)\setminus\Delta F$ is non-stationary set. Since  $\Delta F_i\in F$ for all $i<\omega$, it follows that $\Delta F$ is stationary, for $F$ is stationary-complete.
\end{proof}

We want to close this section by mentioning that $\omega_1$-Ramsey cardinals are limits of $\Delta_\omega^\forall$-Ramsey cardinals. This (and the slightly stronger statement that we will actually mention below) is shown exactly as in \cite[Theorem 5.10]{MR3800756}, using Corollary \ref{corollary:RamseyLocal}.

\begin{proposition}\label{proposition:omega1r}
  $\NN_{nR}^\kappa\in\II\triv_{\omega_1}^\forall(\kappa)$. \qed
\end{proposition}

%%%%%%%%%%%%%%%%%%%%%%%%%%%%%%%%%%%%%%%%%
%%%%%%%%%%%%%%%%%%%%%%%%%%%%%%%%%%%%%%%%%

\section{Strongly Ramsey and super Ramsey cardinals}

In this section, we prove several statements about strong and super Ramsey cardinals contained in Theorem \ref{theorem:IdealContain} (\ref{item:IdealContain:stR}) and (\ref{item:IdealContain:suR}). 
 We start by using ideals similar to the ones used in the proof of Lemma \ref{lemma:notinkapparamseylikeideal} to derive the following result.

\begin{proposition}\label{proposition:StronglyRamseysinUltrapower}
 If $\kappa$ is a strongly Ramsey cardinal, then $\NN\triv^{\kappa^+}_\alpha(\kappa)\in\II\triv^\kappa_\kappa(\kappa)$ for all regular $\alpha<\kappa$. 
\end{proposition}

\begin{proof}
 Pick a $\kappa$-model $M$ and a uniform $M$-ultrafilter $U$ on $\kappa$ that is $M$-normal and $\kappa$-amenable for $M$. Then $\Ult{M}{U}$ is well-founded and $\HH{\kappa^+}^M=\HH{\kappa^+}^{\Ult{M}{U}}\in\Ult{M}{U}$. 
 Fix $x\in\POT{\kappa}^M$. 
 Using the closure properties of $M$ and the fact that $j_U$ is $\kappa$-powerset preserving, we can construct a continuous sequence $\seq{M_i\in\HH{\kappa^+}^M}{i\leq\alpha^+}$ of elementary submodels of $\HH{\kappa^+}^M$ with $x\in M_0$ and ${}^{{<}\kappa}M_i\cup\{M_i\cap U\}\in M_{i+1}$ for all $i<\alpha^+$. 
 Set $M_*=M_{\alpha^+}$ and $U_*=U\cap M_*\in\HH{\kappa^+}^M\subseteq\Ult{M}{U}$. 
 Our construction then ensures that, in $\Ult{M}{U}$, we have $x\in M_*\prec\HH{\kappa^+}$ is a weak $\kappa$-model closed under $\alpha$-sequences and $U_*$ is a uniform $M_*$-ultrafilter that is $M_*$-normal and $\kappa$-amenable for $M_*$. 
 These computations show that $\kappa$ is a $\triv^{\kappa^+}_{\alpha}$-Ramsey cardinal in $\Ult{M}{U}$. 
\end{proof}

The next result yields several statements from Theorem \ref{theorem:IdealContain} (\ref{item:IdealContain:stR}).

\begin{lemma}
 If $\kappa$ is a strongly Ramsey cardinal, then $\II^\kappa_R\cup\{\NN^\kappa_{iR}\}\subseteq\II^\kappa_{stR}$, $\NN^\kappa_{stR}\notin\II^\kappa_{stR}$ and $\II^\kappa_{ie}\nsubseteq\II^\kappa_{stR}$.  
\end{lemma}

\begin{proof}
 By definition, we have $\II^\kappa_R=\II\cc^\kappa_\omega(\kappa)\subseteq\II\cc^\kappa_\kappa(\kappa)=\II\triv^\kappa_\kappa(\kappa)=\II^\kappa_{stR}$.   
 Corollary \ref{corollary:RamseyIdealsNotContained} (\ref{item:RamseyIdealsNotContained:1}) shows that $$\NN^\kappa_{stR} ~ = ~ \NN\triv^\kappa_\kappa(\kappa) ~ \notin ~ \II\triv^\kappa_\kappa(\kappa) ~ = ~ \II^\kappa_{stR}.$$ 
 
 Next, since $\triv^{\kappa^+}_{\omega_1}$-Ramsey cardinals $\kappa$ are ineffably Ramsey, we can use Proposition \ref{proposition:StronglyRamseysinUltrapower}  to show that $\NN^\kappa_{iR}\subseteq\NN\triv^{\kappa^+}_{\omega_1}(\kappa)\in\II^\kappa_{stR}$. 
  Finally, rephrasing the ${<}\kappa$-closure of a model $M$ in a careful way, it is easy to see that a cardinal $\kappa$ is strongly Ramsey if and only if for all $x\subseteq\kappa$ and all $P\subseteq\POT{\kappa}$, either $P$ is not equal to the set of all bounded subsets of $\kappa$ or there exists a transitive weak $\kappa$-model $M$ with $x\in M$, a surjection $\map{s}{\kappa}{M}$ and a uniform $M$-ultrafilter $U$ such that $s[p]\in M$ for every $p\in P$, and $U$ is $\kappa$-amenable for $M$ and $M$-normal. 
  Since this equivalence shows that strong Ramseyness is a $\Pi^1_2$-property, the argument used in the last part of the proof of Lemma \ref{lemma:TwkRamsey} can be modified to show that $\II^\kappa_{ie}\nsubseteq\II^\kappa_{stR}$.  
\end{proof}

The following result will be useful below.

\begin{proposition}\label{proposition:InRamseyContainedInSuperRamsey}
 If $\kappa$ is a $\triv^{\kappa^+}_{\omega_1}$-Ramsey cardinal, then $\II\scc^\kappa_\omega(\kappa)\subseteq\II\triv^{\kappa^+}_{\omega_1}(\kappa)$. 
\end{proposition}

\begin{proof}
 Let $M\prec\HH{\kappa^+}$ be a weak $\kappa$-model closed under countable sequences, let $U$ be a uniform $M$-ultrafilter that is $\kappa$-amenable for $M$ and $M$-normal, let $\seq{X_n}{n<\omega}$ be a sequence of elements of $U$ and set $X=\bigcap_{n<\omega}X_n$. Then $X\in U$, $X$ is stationary in $M$, and elementarity implies that $X$ is stationary in $\VV$.  
 This shows that that $\scc(M,U)$ holds. 
\end{proof}

The next lemma proves several statements from Theorem \ref{theorem:IdealContain} (\ref{item:IdealContain:suR}). 
 As mentioned earlier, the argument showing that $\II^\kappa_{{\prec}cie}$ is not a subset of $\II^\kappa_{suR}$ is due to Victoria Gitman. 

\begin{lemma}
 If $\kappa$ is a super Ramsey cardinal, then $\II^\kappa_{iR}\cup\II^\kappa_{stR}\cup\{\NN^\kappa_{stR}\}\subseteq\II^\kappa_{suR}$, $\NN^\kappa_{suR}\notin\II^\kappa_{suR}$, and $\II^\kappa_{{\prec}cie}\nsubseteq\II^\kappa_{suR}$. 
\end{lemma}

\begin{proof}
 Let $\kappa$ be a super Ramsey cardinal.
 By definition, we have $$\II^\kappa_{stR} ~ = ~ \II\triv^\kappa_\kappa(\kappa) ~ \subseteq ~ \II\triv^{\kappa^+}_\kappa(\kappa) ~ = ~ \II^\kappa_{suR}.$$ 
   
  Next, Proposition \ref{proposition:InRamseyContainedInSuperRamsey} allows us to show that $$\II^\kappa_{iR} ~= ~ \II\scc^\kappa_\omega(\kappa) ~ \subseteq ~ \II\triv^{\kappa^+}_{\omega_1}(\kappa) ~ \subseteq ~ \II\triv^{\kappa^+}_\kappa(\kappa) ~ = ~ \II^\kappa_{suR}.$$  
  Moreover, Lemma \ref{lemma:notinkapparamseylikeideal} directly shows that $\NN^\kappa_{stR}=\NN\triv^\kappa_\kappa(\kappa)\in\II\triv^{\kappa^+}_\kappa(\kappa)=\II^\kappa_{suR}$. 
  Corollary \ref{corollary:RamseyIdealsNotContained} (\ref{item:RamseyIdealsNotContained:1}) now allows us to conclude  that $$\NN^\kappa_{suR} ~ = ~ \NN\triv^{\kappa^+}_\kappa(\kappa) ~ \notin ~ \II\triv^{\kappa^+}_\kappa(\kappa) ~ = ~ \II^\kappa_{suR}.$$

  Finally, we show that $\II^\kappa_{{\prec}cie}\nsubseteq\II^\kappa_{suR}$. In the following, we may assume $\kappa$ to be completely ineffable, for otherwise $\kappa\in\II^\kappa_{{\prec}cie}\setminus\II^\kappa_{suR}$. 
  We want to show that $\NN^\kappa_{suR}\in\II^\kappa_{{\prec}cie}$ holds under this additional assumption. 
  Assume, towards  a contradiction, that this is not the case, and pick a sufficiently large regular cardinal $\theta$, a weak $\kappa$-model $M\prec\HH{\theta}$, and a $\kappa$-powerset preserving $\kappa$-embedding $\map{j}{M}{\langle N,\epsilon_N\rangle}$ such that $\kappa$ is not super Ramsey in $\langle N,\epsilon_N\rangle$. However, since the embedding $j$ is $\kappa$-powerset preserving, this implies that $\kappa$ is not super Ramsey in $M$, and thus by elementarity of $M$, $\kappa$ is not super Ramsey in $\VV$, which yields our desired contradiction.
\end{proof}

%%%%%%%%%%%%%%%%%%%%%%%%%%%%%%%%%
%%%%%%%%%%%%%%%%%%%%%%%%%%%%%%%%%

\section{Locally measurable cardinals}\label{section:LocallyMeasurable}

In this section, we prove a few results about locally measurable cardinals that allow us to compare these cardinals and their ideals to the ones studied above, yielding several statements from Theorem \ref{theorem:IdealContain} (\ref{item:IdealContain:LMS}).

\begin{proposition}
 If $\kappa$ is locally measurable, then $\NN^\kappa_{lms}\notin\II^\kappa_{ms}$ and $\II^\kappa_{ie}\nsubseteq\II^\kappa_{ms}$. 
\end{proposition}

\begin{proof}
 As noted in Section \ref{section:ramseylike}, if we set $\Psi\equiv\Psi_{ms}$, then local measurability coincides with $\Psi^\kappa_\omega$-Ramseyness and hence $\NN^\kappa_{lms}=\NN\Psi^\kappa_\omega(\kappa)$ as well as $\II^\kappa_{ms}=\II\Psi^\kappa_\omega(\kappa)$. 
 Since $\Psi$ satisfies the assumptions of Lemma \ref{lemma:IdealsNicelyProper}, we can use the lemma to conclude that $\NN\Psi^\kappa_\omega(\kappa)\notin\II\Psi^\kappa_\omega(\kappa)$. 
 Next, since local measurability is a $\Pi^1_2$-property, we can modify the proof of Lemma \ref{lemma:TwkRamsey} to show that $\II^\kappa_{ie}\nsubseteq\II^\kappa_{ms}$. 
\end{proof}

\begin{lemma}
 If $\kappa$ is a locally measurable cardinal, then $\kappa$ is strongly Ramsey and $\II^\kappa_{stR}\subseteq\II^\kappa_{ms}$.
%  If $\kappa$ is a locally measurable cardinal, then $\kappa$ is Ramsey and $\II^\kappa_R\subseteq\II^\kappa_{ms}$.
\end{lemma}

\begin{proof}
  Pick $A\in\POT{\kappa}\setminus\II^\kappa_{ms}$ and some $x\subseteq\kappa$. 
Then, there is a transitive weak $\kappa$-model $M$ and an $M$-normal $M$-ultrafilter $U$ such that $x,\HH{\kappa},U\in M$ and $A\in U$. 
 But then, $U\in M\models\ZFC^-$ implies that $\POT{\kappa}^M\in M$, and hence $\HH{\kappa^+}^M\in M$. 
 Then, $M$ contains a continuous sequence $\seq{M_i}{i\leq\kappa}$ of elementary submodels of $\HH{\kappa^+}^M$ with $\kappa,x,A\in M_0$ and $({}^{{<}\kappa}M_i)^M\cup\{M_i\cap U\}\subseteq M_{i+1}$ for all $i<\kappa$. 
 By construction, the model $M_\kappa$ has cardinality $\kappa$ in $M$ and, since $\HH{\kappa}\subseteq M$, we know that ${}^{{<}\kappa}M_\kappa\subseteq M$. 
 But this implies that $M_\kappa$ is a $\kappa$-model. 
 Since our construction also ensures that $U$ is $M_\kappa$-normal and $\kappa$-amenable for $M_\kappa$, we can conclude that $A\notin\II^\kappa_{stR}$. 
 % Pick $A\in\POT{\kappa}\setminus\II^\kappa_{ms}$ and a function $\map{f}{[A]^{{<}\omega}}{2}$. 
 % Then there is a weak $\kappa$-model $M$ with $A,f\in M$ and the property that there is $F\in M$ such that $A\in F$ and $F\cap M$ is an $M$-normal $M$-ultrafilter. 
 %Since $F$ is a normal ultrafilter on $\kappa$ in $M$, a routine modification of the proof of \emph{Rowbottom's Theorem} (see {\cite[Theorem 7.17]{MR1994835}}) yields a $(f\upharpoonright[A]^n)$-homogeneous subsets $H_n\subseteq A$ in $F$ for every $n<\omega$, and this construction can be done within $M$, uniformly in $n$.
 %  Thus we can find our desired homogeneous set $H=\bigcap_{n<\omega}H_n\subseteq A$ in $F$ and conclude that $A$ is Ramsey. 
\end{proof}

The following result shows that locally measurable cardinals are consistency-wise strictly above all the large cardinals mentioned in Table \ref{table:schemeKappa}, noting that $\Delta_\kappa^\forall$-Ramsey cardinals are implication-wise stronger than all these large cardinal notions.

\begin{theorem}
 If $\kappa$ is locally measurable, then $\NN\delt^\forall_\kappa\in\II^\kappa_{ms}$. 
\end{theorem}

\begin{proof}
Assume that $\kappa$ is locally measurable. 
 Let $M$ be a transitive weak $\kappa$-model with $\VV_\kappa\in M$, let $U$ be an $M$-normal $M$-ultrafilter with $U\in M$,  let $\theta>\kappa$ be a regular cardinal in $\Ult{M}{U}$ and let $x\in\HH{\theta}^{\Ult{M}{U}}$. 
 Using the fact that $U$ is a normal ultrafilter on $\kappa$ in $M$ and $\Ult{M}{U}$ can be identified with the ultrapower $\Ult{M}{U}^M$ of $M$ by $U$ constructed in $M$, we know that $\Ult{M}{U}$ is closed under $\kappa$-sequences in $M$, and we can find an increasing continuous sequence $\seq{M_i}{i\leq\kappa}$ of weak $\kappa$-models in $\Ult{M}{U}$ with the properties that $x\in M_0$, that $M_i\prec\HH{\theta}^{\Ult{M}{U}}$, and that $M_i^{{<}\kappa}\cup\{M_i\cap U\}\subseteq M_{i+1}$ for all $i<\kappa$. 
 Then, $x\in M_\kappa\prec\HH{\theta}^{\Ult{M}{U}}$, and our construction ensures that $U_\kappa=U\cap M_\kappa\in\Ult{M}{U}$ is a $\kappa$-amenable $M_\kappa$-ultrafilter in $\Ult{M}{U}$. 
Moreover,  since $U$ is a normal ultrafilter in $M$, the filter $U_\kappa$ is normal in $\Ult{M}{U}$. 
 These computations show that $\kappa$ is $\delt^\forall_\kappa$-Ramsey in $\Ult{M}{U}$, and therefore $\NN\delt^\forall_\kappa\notin U$. 
This allows us to conclude that the set $\NN\delt^\forall_\kappa$ is contained in $\II^\kappa_{ms}$. 
 %
 %Build  Then $U_\kappa$ is an $M_\kappa$-normal $M_\kappa$-ultrafilter that is also $\kappa$-amenable for $M_\kappa$. Using that $\langle M,U\rangle\models U$ is normal, and that the ultrapower embedding from $M$ to $N$ is $\kappa$-powerset preserving, it follows that $N$ thinks that $U_\kappa=U\cap M_\kappa$ is normal, showing that $\kappa$ is $\Delta_\kappa^\forall$-Ramsey in $N$.
 %
 %We show that $\kappa$ is $\Delta_\kappa^\forall$-Ramsey in $N$, hence by elementarity, $\kappa$ is a limit of $\Delta_\kappa^\forall$-Ramsey cardinals in $M$, and since this property is sufficiently local by Corollary \ref{corollary:RamseyLocal}, and we have $V_\kappa\subseteq M$, it follows that $\kappa$ is indeed a limit of $\Delta_\kappa^\forall$-Ramsey cardinals.
\end{proof}

 Note that $\delt^\forall_\kappa$-Ramsey cardinals are in particular super Ramsey, and therefore the above theorem provides a proof for the statement $\NN^\kappa_{suR}\in\II^\kappa_{ms}$ from Theorem \ref{theorem:IdealContain} (\ref{item:IdealContain:LMS}).

%%%%%%%%%%%%%%%%%%%%%%%%%%
%%%%%%%%%%%%%%%%%%%%%%%%%%
%%%%%%%%%%%%%%%%%%%%%%%%%%

\section{The measurable ideal}\label{section:MeasurableIdeal}

We close our paper with the investigation of the ideal induced by the property $\Psi_{ms}(M,U)$ with respect to Scheme \ref{schemeSmall} and Scheme \ref{schemeKappa}, and its relations with the ideals studied above. 
We start by verifying Theorem \ref{theorem:Ideals} (\ref{item:Ideals:measurable}) and Theorem \ref{theorem:IdealContain} (\ref{item:IdealContain:MS}), and then make some further observations concerning this ideal and its induced partial order $\POT{\kappa}/\II_{\prec{ms}}^\kappa$.

\begin{lemma}\label{lemma:measurableIdelaUnion}
  If $\kappa$ is a measurable cardinal, then $\II^{{<}\kappa}_{{ms}}=\II^\kappa_{\prec{ms}}$ and this ideal is equal to the complement of the union of all normal ultrafilters on $\kappa$. 
\end{lemma}

\begin{proof}
 First, if $A\subseteq\kappa$ with $A\notin\II^{{<}\kappa}_{ms}\cap\II^\kappa_{\prec{ms}}$, then there is a regular cardinal $\theta>(2^\kappa)^+$, an infinite cardinal $\lambda\leq\kappa$, a weak $(\lambda,\kappa)$-model $M\prec\HH{\theta}$ with $A\in M$ and an $M$-ultrafilter $U$ on $\kappa$ with $A\in U$ and $\Psi_{ms}(M,U)$. 
 Then $U$ is $M$-normal and $U=F\cap M$ for some $F$ in $M$. Therefore elementarity implies that $F$ is a normal ultrafilter on $\kappa$ with $A\in F$. 
 In the other direction, assume that $\lambda\leq\kappa$ is an infinite cardinal, $F$ is a normal ultrafilter on $\kappa$ and $A\in F$. If $\theta>(2^\kappa)^+$ is regular and $x\in\HH{\theta}$, then we can pick a weak $(\lambda,\kappa)$-model $M\prec\HH{\theta}$ with  $x,A,F\in M$. In this situation, it is easy to see that $\Psi_{ms}(M,F\cap M)$ holds and hence $A\notin\II^{{<}\kappa}_{{ms}}\cup\II^\kappa_{\prec{ms}}$. 
\end{proof}

\begin{lemma}\label{lemma:PropertiesOfMeasurableIdeal}
  If $\kappa$ is measurable, then $\II\delt^\forall_\kappa(\kappa)\cup\II^\kappa_{ms}\cup\{\NN^\kappa_{lms}\}\subseteq\II^\kappa_{\prec{ms}}$ 
 %$\II^\kappa_{\prec{nR}}\cup\{\NN^\kappa_{nR}\}\subseteq\II^\kappa_{\prec{ms}}$, 
and $\NN^\kappa_{ms}\notin\II^\kappa_{\prec{ms}}$. 
\end{lemma}

\begin{proof}  
  Assume that $A\subseteq\kappa$ is not in $\II^\kappa_{\prec{ms}}$. Using Lemma \ref{lemma:measurableIdelaUnion}, we may pick a normal ultrafilter $U$ on $\kappa$ such that $A\in U$. But then, $A\notin\II\delt^\forall_\kappa(\kappa)\cup\II^\kappa_{ms}$ is easily seen to be witnessed using suitable models $M$ that contain $U$ as an element together with the $M$-ultrafilter $U\cap M$. 
 Moreover, if $U$ is a normal ultrafilter on $\kappa$, then $\VV$ and $\Ult{\VV}{U}$ contain the same weak $\kappa$-models and hence $\kappa$ is locally measurable in $\Ult{\VV}{U}$. By Lemma \ref{lemma:measurableIdelaUnion}, this shows that $\NN^\kappa_{lms}\in\II^\kappa_{{\prec}ms}$. 
Finally, assume that there is a measurable cardinal $\kappa$ with $\NN^\kappa_{ms}\in\II^\kappa_{\prec{ms}}$, and let $\kappa$ be minimal with this property. 
 By Lemma \ref{lemma:measurableIdelaUnion}, we can pick a normal ultrafilter $U$ on $\kappa$ with $\kappa\setminus\NN^\kappa_{ms}\in U$. 
 Set $M=\Ult{\VV}{U}$. Then $\kappa$ is measurable in $M$. Moreover, since $\HH{\kappa^+}\subseteq M$, we have $\NN^\kappa_{ms}=(\NN^\kappa_{ms})^M$, and therefore, the minimality of $\kappa$ implies that $\NN^\kappa_{ms}\notin(\II^\kappa_{\prec{ms}})^M$. 
 Again, by Lemma \ref{lemma:measurableIdelaUnion}, this yields a normal ultrafilter $F$ on $\kappa$ in $M$ with $\NN^\kappa_{ms}\in F$. Since the closure properties of $M$ ensure that $F$ is a normal ultrafilter on $\kappa$ in $\VV$, another application of Lemma \ref{lemma:measurableIdelaUnion} shows that $\NN^\kappa_{ms}\notin\II^\kappa_{\prec{ms}}$, a contradiction. 
\end{proof}

Note that if $\kappa$ is measurable, then we have $$\II^\kappa_{{\prec}cie} ~ \subseteq ~ \II\delt^\forall_\omega(\kappa) ~ \subseteq ~ \II\delt^\forall_\kappa(\kappa) ~ \subseteq ~ \II^\kappa_{{\prec}ms}$$ and $$\II^\kappa_{suR} ~ \subseteq ~ \II\delt^{\kappa^+}_\kappa(\kappa) ~ \subseteq ~ \II\delt^\forall_\kappa(\kappa) ~ \subseteq ~ \II^\kappa_{{\prec}ms}.$$   
In particular, the above lemma implies the related statements in Theorem \ref{theorem:IdealContain} (\ref{item:IdealContain:MS}).

 Lemma \ref{lemma:measurableIdelaUnion} shows that the ideal $\II^\kappa_{\prec{ms}}$ can consistently be the complement of a normal ultrafilter on $\kappa$. 
For example, this holds when there is such a filter $U$ on $\kappa$ with the property that $\VV=\LL[U]$ holds (see {\cite[Corollary 20.11]{MR1994835}}).
 This shows that it is possible that the canonical partial order $\POT{\kappa}/\II^\kappa_{<{ms}}$ induced by this ideal is atomic. 
In contrast, Theorem \ref{theorem:Ideals}.(\ref{item:Ideals:ia}) directly implies that for every inaccessible cardinal $\kappa$, the corresponding partial order $\POT{\kappa}/\II^{{<}\kappa}_{ia}$ is not atomic. 
The following results will allow us to show that, for many  of the large cardinal properties characterized in this paper that are weaker than measurability, their  corresponding ideals do not induce atomic quotient partial orders.

\begin{lemma}\label{lemma:atoms}
 Let $\II$ be a normal ideal on an uncountable regular cardinal $\kappa$. 
 \begin{enumerate} 
  \item\label{item:LemmaTrivialForcing1} If $[A]_\II$ is an atom in the partial order $\POT{\kappa}/\II$, then $U_A=\Set{B\subseteq\kappa}{A\setminus B\in\II}$ is a normal ultrafilter on $\kappa$ containing $A$, with $\II^+\cap\POT{A}=U_A\cap\POT{A}$. 

  \item\label{item:LemmaTrivialForcing2} If the partial order $\POT{\kappa}/\II$ is atomic, then $\kappa$ is a measurable cardinal with $\II^\kappa_{\prec{ms}}\subseteq\II$, and the ideal $\II$ is precipitous. 
 \end{enumerate}
\end{lemma}

\begin{proof}
 (\ref{item:LemmaTrivialForcing1}) Assume that there is $B\in\POT{\kappa}\setminus U_A$ with $\kappa\setminus B\notin U_A$. Then $A\cap B,A\setminus B\in\II^+$, and this implies that $[A\cap B]_\II$ and $[A\setminus B]_\II$ are incompatible conditions in $\POT{\kappa}/\II$ below $[A]_\II$, a contradiction. Since $\II$ is a normal ideal on $\kappa$, this shows that $U_A$ is a normal ultrafilter on $\kappa$. Moreover, if $B\in U_A\cap\POT{A}$, then $A\setminus B\in\II$ and $A\in\II^+$ implies that $B\in\II^+$. Finally, if $B\in\POT{A}\setminus U_A$, then the above computations show that $A\setminus B\in U_A$, and hence, $B\in\II$.

  (\ref{item:LemmaTrivialForcing2}) By (\ref{item:LemmaTrivialForcing1}), the existence of an atom in $\POT{\kappa}/\II$ implies the measurability of $\kappa$. Next, if $A\in\II^+$, then our assumption yields a $B\in\II^+\cap\POT{A}$ with the property that $[B]_\II$ is an atom in $\POT{\kappa}/\II$ and, by (\ref{item:LemmaTrivialForcing1}), the filter $U_B$ witnesses that $A$ is an element of $(I^\kappa_{ms})^+$. This shows that $\II^+\subseteq(\II^\kappa_{ms})^+$ and therefore $\II^\kappa_{ms}\subseteq\II$. Finally, let $\sigma$ be a strategy for the Player {\sf Nonempty} in the precipitous game $\mathcal{G}_\II$ (see {\cite[Lemma 22.21]{MR1940513}}), with the property that whenever Player {\sf Empty} plays $A\in\II^+$ for their first move of the game, then {\sf Nonempty} replies by playing $B\in\II^+\cap\POT{A}$ so that $[B]_\II$ is an atom in $\POT{\kappa}/\II$. Now, if $\seq{A_n}{n<\omega}$ is a run of $\mathcal{G}_\II$ in which {\sf Nonempty} played according to $\sigma$, then the above arguments show that $U=U_{A_1}$ is a normal ultrafilter on $\kappa$ with $A_n\in U$ for all $n<\omega$. Hence, $\emptyset\neq\bigcap_{n<\omega}A_n\in U$. This shows that $\sigma$ is a winning strategy for {\sf Nonempty} in $\mathcal{G}_\II$, and therefore that $\II$ is precipitous. 
\end{proof}

Note that %in combination with Lemma \ref{lemma:AllIdealsNormal},
the above lemma allows us to derive the statements on the non-atomicity of the induced ideals in all items of Theorem \ref{theorem:IdealContain}: 
  For any large cardinal notion weaker than measurability that is mentioned in the theorem, the results of Theorem \ref{theorem:IdealContain} that we have already verified show that their corresponding ideals on a measurable cardinal $\kappa$ are strictly contained in the measurable ideal $\II_{\prec{ms}}^\kappa$ on $\kappa$. 
  Together with Lemma \ref{lemma:atoms}.(2), this shows that these ideals can never be atomic.

 In the remainder of this section, we consider the question whether the partial order induced by the measurable ideal has to be atomic. 
 The following lemma gives a useful criterion for the atomicity of these partial orders. 
Note that the assumption of the lemma is satisfied if the Mitchell order on the collection of normal ultrafilters on the given measurable cardinal is linear.
  As noted in \cite{MR3809587}, this statement holds in all known canonical inner models for large
cardinal hypotheses, and is expected to also be true in potential canonical inner models for supercompact cardinals.

\begin{lemma}
 If  $\kappa$ is a measurable cardinal with the property that any set of pairwise incomparable elements in the Mitchell ordering at $\kappa$ has size at most $\kappa$,  %for every collection $\mathcal{F}$ of normal ultrafilters on $\kappa$, the set of all elements of $\mathcal{F}$ that are minimal in $\mathcal{F}$ with respect to the Mitchell ordering has cardinality at most $\kappa$. 
then the partial order $\POT{\kappa}/I^\kappa_{\prec{ms}}$ is atomic. 
\end{lemma}

\begin{proof}
 Let $\II=\II^\kappa_{\prec{ms}}$, and fix $A\in\II^+$. Let $\mathcal{F}$ denote the collection of all normal ultrafilters on $\kappa$ that contain $A$, and let $\mathcal{F}_0$ denote the set of all elements of $\mathcal{F}$ that are minimal in $\mathcal{F}$ with respect to the Mitchell ordering. Note that any two elements of $\mathcal F_0$ are incomparable, hence $\mathcal F_0$ has size at most $\kappa$ by our assumption. Lemma \ref{lemma:measurableIdelaUnion} implies that $\mathcal F_0\ne\emptyset$. % there exists a $U\in\mathcal{F}_0$ with $A\in U$.
We may thus pick some $U\in\mathcal F_0$.

 \begin{claim*}
  There exists $B\in \POT{A}\cap U$ with the property that $U$ is the unique element of $\mathcal{F}_0$ that contains $B$. 
 \end{claim*}

 \begin{proof}[Proof of the Claim]
  We may assume that $\mathcal{F}_0\neq\{U\}$. Let $\map{u}{\kappa}{\mathcal{F}_0\setminus\{U\}}$ be a surjection. Given $\alpha<\kappa$, fix $B_\alpha\in U\setminus u(\alpha)$. Define $B=A\cap\triangle_{\alpha<\kappa}B_\alpha\in\POT{A}\cap U$. Then, $B\notin u(\alpha)$ for any $\alpha<\kappa$, for otherwise we would have $B\cap(\alpha,\kappa)\in u(\alpha)$, and hence $B_\alpha\in u(\alpha)$ for some $\alpha<\kappa$, contradicting our assumption on $B_\alpha$. 
 \end{proof}

\begin{claim*}
 There exists $C\in\POT{B}\cap U$ with the property that $U$ is the unique normal ultrafilter on $\kappa$ that contains $C$. 
\end{claim*}

\begin{proof}[Proof of the Claim]
 First, assume that $U$ has Mitchell rank $0$. Then $$C ~ = ~ \Set{\alpha\in B}{\textit{$\alpha$ is not measurable}} ~ \in ~ \POT{B}\cap U.$$ Let $U^\prime$ be a normal ultrafilter on $\kappa$ that contains $C$. Then $A,B\in U^\prime$ and $U^\prime$ has Mitchell rank $0$. This implies $U^\prime\in\mathcal{F}_0$ and, by the previous claim, we can conclude that $U=U^\prime$. 

 Now, assume that $U$ has Mitchell rank greater than $0$. Define $$C ~ = ~ \Set{\alpha\in B}{\textit{$\alpha$ is measurable and $B\cap\alpha\notin F$ for every normal ultrafilter $F$ on $\alpha$}}.$$ Then $U\in\mathcal{F}_0$ implies that $C\in\POT{B}\cap U$. Let $U^\prime$ be a normal ultrafilter on $\kappa$ that contains $C$. Then $A,B\in U^\prime$ and $U'\in\mathcal{F}_0$. By the above claim, we know that $U=U^\prime$ . 
\end{proof}

 \begin{claim*}
  The condition $[C]_\II$ is an atom in the partial order $\POT{\kappa}/\II$. 
 \end{claim*}

 \begin{proof}[Proof of the Claim]
  Pick $D,E\in\II^+$ with $[D]_\II,[E]_\II\leq_{\POT{\kappa}/\II}[C]_\II$. By Theorem \ref{theorem:Ideals}.(\ref{item:Ideals:measurable}), we can find normal ultrafilters $U_0$ and $U_1$ on $\kappa$ with $D\in U_0$ and $E\in U_1$. Then $C\in U_0\cap U_1$, $U=U_0=U_1$, $D\cap E\in U\subseteq\II^+$, and $[D\cap E]_\II\leq_{\POT{\kappa}/\II}[D]_\II,[E]_\II$. 
 \end{proof}

 This completes the proof of the lemma. 
\end{proof}

In contrast to the situation studied in the above lemma, the next result shows that it is possible to combine ideas from a classical construction of Kunen and Paris from \cite{MR0277381} with results of Hamkins from \cite{MR2063629} to obtain a measurable cardinal $\kappa$ with the property that the  ideal $\II^\kappa_{\prec{ms}}$ induces an atomless partial order.

\begin{theorem}
 Let $\kappa$ be a measurable cardinal. Then, in a generic extension of the ground model $\VV$, the cardinal $\kappa$ is measurable and  the partial order $\POT{\kappa}/\II^\kappa_{\prec{ms}}$ is atomless. Moreover, if the ideal $\II^\kappa_{\prec{ms}}$ is precipitous in $\VV$, then the ideal $\II^\kappa_{\prec{ms}}$ is precipitous in the given generic extension. 
\end{theorem}

\begin{proof}
  Let $\mathcal{F}_0$ denote the collection of all normal ultrafilters on $\kappa$ in $\VV$, let $c$ be $\Add{\omega}{1}$-generic over $\VV$, let $\mathcal{F}$ denote the collection of all normal ultrafilters on $\kappa$ in $\VV[c]$ and set $\lambda=\kappa^+=(\kappa^+)^{\VV[c]}$. 
 By the L{\'e}vy--Solovay Theorem (see {\cite[Proposition 10.13]{MR1994835}}) and results of Hamkins (see {\cite[Corollary 8 and Lemma 13]{MR2063629}}), there is a bijection $\map{b}{\mathcal{F}}{\mathcal{F}_0}$ with $b(U)=U\cap\POT{\kappa}^\VV$ and $$U ~ = ~ \Set{A\in\POT{\kappa}^{\VV[c]}}{\exists B\in b(U) ~ A\supseteq B}$$ for all $U\in\mathcal{F}$. Given $U\in\mathcal{F}$, let $\map{j_U}{\VV[c]}{M_U=\Ult{\VV[c]}{U}}$ denote the corresponding ultrapower embedding. 
 Work in $\VV[c]$, and let $\vec{\PPP}$ denote the Easton-support product of all partial orders of the form $\Add{\nu^+}{1}$ for some infinite cardinal $\nu<\kappa$. Then $\vec{\PPP}$ has cardinality $\kappa$ and satisfies the $\kappa$-chain condition. 
Let $f$ be the function with domain $\kappa$ and $f(\alpha)=\vec{\PPP}\restriction[\alpha,\kappa)$ for all $\alpha<\kappa$.  Given $U\in\mathcal{F}$, set $\RRR_U=[f]_U$ and define $\vec{\RRR}_U$ to be the ${<}\lambda$-support product of $\lambda$-many copies of $\RRR_U$. Then $\RRR_U$ is a ${<}\lambda$-closed partial order in $M_U$ and there is a canonical isomorphism between $j(\vec{\PPP})$ and $\vec{\PPP}\times\RRR_U$ in $M_U$. Since we have ${}^\kappa M_U\subseteq M_U$, the partial orders $\RRR_U$ and $\vec{\RRR}_U$ are also ${<}\lambda$-closed in $\VV[c]$.  
 Finally, let $\vec{\SSS}$ denote the ${<}\lambda$-support product of all partial orders of the form $\vec{\RRR}_U$ with $U\in\mathcal{F}$.

 Let $G\times H$ be $(\vec{\PPP}\times\vec{\SSS})$-generic over $\VV[c]$. 
Given $U\in\mathcal{F}$ and $\gamma<\lambda$, we let $H_{U,\gamma}$ denote the filter induced by $H$ on the $\gamma$-th factor of $\vec{\RRR}_U$, and we let $\map{j_{U,\gamma}}{V[c,G]}{M_U[G,H_{U,\gamma}]}$ denote the corresponding canonical lifting of $j_U$ (see {\cite[Proposition 9.1]{MR2768691}}). 
 Finally, we set $\II=(\II^\kappa_{\prec{ms}})^{\VV[c,G,H]}$.

 \begin{claim*}
  $\POT{\kappa}^{\VV[c,G,H]}\subseteq\VV[c,G]$. 
 \end{claim*}

 \begin{proof}[Proof of the Claim]
  Fix $A\in\POT{\kappa}^{\VV[c,G,H]}$. Then, there is a $\vec{\PPP}$-nice name $\dot{A}$ in $\VV[c,H]$ with $A=\dot{A}^G$. Since $\vec{\SSS}$ is ${<}\lambda$-closed in $\VV[c]$, the above remarks imply that $\vec{\PPP}$ satisfies the $\kappa$-chain condition and has cardinality $\kappa$ in $\VV[c,H]$. This shows that $\dot{A}$ is an element of $\VV[c]$ and we can conclude that $A=\dot{A}^G\in\VV[c,G]$.  
 \end{proof}

 The above claim directly implies that if $U\in\mathcal{F}$ and $\gamma<\lambda$, then $$U_\gamma ~ = ~ \Set{A\in\POT{\kappa}^{\VV[c,G]}}{\kappa\in j_{U,\gamma}(A)}$$ is a normal ultrafilter on $\kappa$ in $\VV[c,G,H]$.

 \begin{claim*}
  Given $U\in\mathcal{F}$, $\gamma<\lambda$ and a $\vec{\PPP}$-name $\dot{A}\in\VV[c]$ for a subset of $\kappa$, 
the set $\dot{A}^G$ is an element of $U_\gamma$ if and only if there is 
a condition $\vec{p}\in G$, 
an element $E$ of $U$ and 
a function $g\in\VV[c]$ with domain $\kappa$ and $[g]_U\in H_{U,\gamma}$ 
such that for all $\alpha\in E$, we have $\supp(\vec{p})\subseteq\alpha$, $g(\alpha)\in\vec{\PPP}\restriction[\alpha,\kappa)$ and $\vec{p}\cup g(\alpha)\Vdash^{\VV[c]}_{\vec{\PPP}}\anf{\check{\alpha}\in\dot{A}}$. 
 \end{claim*}
 
 \begin{proof}[Proof of the Claim]
  Let $K$ denote the filter on $j_U(\vec{\PPP})$ induces by $G\times H_{U,\gamma}$. 
  First, assume that $\dot{A}^G$ is an element of $U_\gamma$. Then $\kappa\in j_{U,\gamma}(\dot{A}^G)=j_U(\dot{A})^K$ and there exists a condition $\vec{q}$ in $K$ with the property that $\vec{q}\Vdash^{M_U}_{j_U(\vec{\PPP})}\anf{\check{\kappa}\in j_U(\dot{A})}$. 
  Set $\vec{p}=\vec{q}\restriction\kappa\in G$ and pick a function $g$ in $\VV[c]$ with domain $\kappa$ satisfying $[g]_U=\vec{q}\restriction[\kappa,j_U(\kappa))\in H_{U,\gamma}$. 
  Moreover, define $$E ~ = ~ \Set{\alpha<\kappa}{\supp(\vec{p})\subseteq\alpha, ~ g(\alpha)\in\vec{\PPP}\restriction[\alpha,\kappa), ~ \vec{p}\cup g(\alpha)\Vdash^{\VV[c]}_{\vec{\PPP}}\anf{\check{\alpha}\in\dot{A}}} ~ \in ~ \VV[c].$$ 
  By {\L}os' Theorem, our assumptions on $\vec{q}$ directly imply that $E$ is an element of $U$. 
  In the other direction, if $\vec{p}$, $g$ and $E$ satisfy the properties listed in the statement of the claim, then {\L}os' Theorem implies that $$\vec{p}\cup[g]_U\Vdash^{M_U}_{j_U(\vec{\PPP})}\anf{\check{\kappa}\in j_U(\dot{A})}$$ and hence $j_{U,\gamma}(\dot{A}^G)=j_U(\dot{A})^K\in U_\gamma$. 
 \end{proof}

 \begin{claim*}
  If $W$ is a normal ultrafilter on $\kappa$ in $\VV[c,G,H]$, then $W\cap\VV[c]\in\mathcal{F}$. 
 \end{claim*}
 
 \begin{proof}[Proof of the Claim]
  Since the partial order $\vec{\PPP}\times\vec{\SSS}$ is $\sigma$-closed in $\VV[c]$, the results of \cite{MR2063629} mentioned above yield an ultrafilter $U\in\mathcal{F}$ with $W\cap\VV=b(U)=U\cap\VV$. Since $$U ~ = ~ \Set{A\in\POT{\kappa}^{\VV[c]}}{\exists B\in W\cap\VV ~ A\subseteq B}$$ is an ultrafilter in $\VV[c]$, we can conclude that $U=W\cap\VV[c]$. 
 \end{proof}

 \begin{claim*}
  We have
  \begin{equation*}
   \begin{split}
    \II ~ & = ~ \POT{\kappa}^{\VV[c,G]}\setminus\bigcup\Set{U_\gamma}{U\in\mathcal{F}, ~ \gamma<\lambda} \\
    ~ & = ~ \Set{A\in\POT{\kappa}^{\VV[c,G]}}{\exists B\in(\II^\kappa_{ms})^{\VV[c]} ~ A\subseteq B}. 
   \end{split}
  \end{equation*} 
 \end{claim*}

 \begin{proof}[Proof of the Claim]
  Fix a $\vec{\PPP}$-name $\dot{A}\in\VV[c]$ for a subset of $\kappa$, and let $O$ denote the set of all $\vec{p}$ in $\vec{\PPP}$ with   $$D_{\vec{p}} ~ = ~ \Set{\alpha<\kappa}{\vec{p}\Vdash^{\VV[c]}_{\vec{\PPP}}\anf{\check{\alpha}\notin\dot{A}}} ~ \in ~ \bigcap\mathcal{F}.$$

  First, assume that there is a $\vec{p}\in G\cap O$. Then $\dot{A}^G\cap D_{\vec{p}}=\emptyset$ and, if $W$ is a normal ultrafilter on $\kappa$ in $\VV[c,G,H]$, then $D_{\vec{p}}\in W\cap\VV[c]\in\mathcal{F}$ and hence $\dot{A}^G\notin W$. By Theorem \ref{theorem:Ideals}.(\ref{item:Ideals:measurable}), this shows that $\dot{A}^G\in\II$. In particular, we have $\dot{A}^G\notin U_\gamma$ for all $U\in\mathcal{F}$ and $\gamma<\lambda$. Finally, these arguments also directly show that $\dot{A}^G\subseteq\kappa\setminus D_{\vec{p}}\in(I^\kappa_{ms})^{\VV[c]}$ holds.

 Now, assume that $G\cap O=\emptyset$. Since $O$ is an open subset of $\vec{\PPP}$ in $\VV[c]$, there is $\vec{p}_0\in G$ with the property that no condition below $\vec{p}_0$ in $\vec{\PPP}$ is an element of $O$. 
 Fix some condition $(\vec{p}_1,\vec{s}_1)$ below $(\vec{p}_0,\mathbbm{1}_{\vec{\SSS}})$ in $\vec{\PPP}\times\vec{\SSS}$. 
 Then, there is $U\in\mathcal{F}$ with $$E ~ = ~ \Set{\alpha<\kappa}{\textit{$\alpha$ is inaccessible with $\vec{p}_1\not\Vdash^{\VV[c]}_{\vec{\PPP}}\anf{\check{\alpha}\notin\dot{A}}$}} ~ \in ~ U.$$
  This allows us to find a sequence $\seq{\vec{q}_\alpha\in\vec{\PPP}}{\alpha<\kappa}$ in $\VV[c]$ with $\vec{q}_\alpha\leq_{\vec{\PPP}}\vec{p}_1$ and $\vec{q}_\alpha\Vdash^{\VV[c]}_{\vec{\PPP}}\anf{\check{\alpha}\in\dot{A}}$ for all $\alpha\in E$. 
 Then the set $\supp(\vec{q}_\alpha\restriction\alpha)$ is bounded in $\alpha$ all for $\alpha\in E$ and hence,  
 using the normality of $U$ and the inaccessibility of the elements of $E$, we find a condition $\vec{p}$ below $\vec{p}_1$ in $\vec{\PPP}$ and an element $F$ of $U$ with $F\subseteq E$ and $\vec{q}_\alpha\restriction\alpha=\vec{p}$ for all $\alpha\in F$. 
 Now, pick a function $g\in\VV[c]$ with domain $\kappa$ and $g(\alpha)=\vec{q}_\alpha\restriction[\alpha,\kappa)$ for all $\alpha\in F$. 
 Then $\vec{p}\cup g(\alpha)\Vdash^{\VV[c]}_{\vec{\PPP}}\anf{\check{\alpha}\in\dot{A}}$ for all $\alpha\in F$.
 % and hence $\vec{p}\cup[g]_U\Vdash^{M_U}_{j_U(\vec{\PPP})}\anf{\check{\kappa}\in j_U(\dot{A})}$. 
 %
 Finally, fix $\gamma\in\lambda\setminus\supp(\vec{s}_1(U))$. 
Then, there is a condition $\vec{s}$ below $\vec{s}_1$ in $\vec{\SSS}$ with the property that $\gamma\in\supp(\vec{s}(U))$ and $\vec{s}(U)(\gamma)=[g]_U$. 
By genericity, these computations allow us to find a condition $(\vec{p},\vec{s})$ in $G\times H$ with the property that there are $U\in\mathcal{F}$, $E\in U$, $\gamma<\lambda$ and a function $g\in\VV[c]$ with domain $\kappa$ and $[g]_U=\vec{s}(U)(\gamma)$, such that for all $\alpha\in E$, we have $\supp(\vec{p})\subseteq\alpha$, $g(\alpha)\in\vec{\PPP}\restriction[\alpha,\kappa)$ and $\vec{p}\cup g(\alpha)\Vdash^{\VV[c]}_{\vec{\PPP}}\anf{\check{\alpha}\in\dot{A}}$. 
 By a previous claim, this shows that $\dot{A}^G$ is an element of $U_\gamma$, and hence $\dot{A}^G\notin\II$. 
 Finally, if $B\in(\II^\kappa_{ms})^{\VV[c]}$, then $\kappa\setminus B\in U\subseteq U_\gamma$ and hence $\dot{A}^G\nsubseteq B$. 
 \end{proof}

 \begin{claim*}
  In $\VV[c,G,H]$, the partial order $\POT{\kappa}/\II$ is atomless. 
 \end{claim*}

 \begin{proof}[Proof of the Claim] 
  Pick a $\PPP$-name $\dot{A}$ in $\VV[c]$ with $\dot{A}^G\notin\II$. By the previous claim, there is $U\in\mathcal{F}$ and $\gamma<\lambda$ with $\dot{A}^G\in U_\gamma$. 
 In this situation, earlier remarks show that we can find $(\vec{p}_0,\vec{s}_0)\in G\times H$, $E\in U$ and a function $g\in\VV[c]$ with domain $\kappa$ and $[g]_U=\vec{s}_0(U)(\gamma)$ such that $\supp(\vec{p}_0)\subseteq\alpha$, $g(\alpha)\in\vec{\PPP}\restriction[\alpha,\kappa)$ and $\vec{p}_0\cup g(\alpha)\Vdash^{\VV[c]}_{\vec{\PPP}}\anf{\check{\alpha}\in\dot{A}}$ for all $\alpha\in E$. 
 Fix a condition $(\vec{p}_1,\vec{s}_1)$ below $(\vec{p}_0,\vec{s}_0)$ in $\vec{\PPP}\times\vec{\SSS}$ and $\delta\in\lambda\setminus\supp(\vec{s}_1(U))$. Then, we can find $F\in U$ and a function $h\in\VV[c]$ with domain $\kappa$ and $[h]_U=\vec{s}_1(U)(\gamma)$ such that $\supp(\vec{p}_1)\subseteq\alpha$ and $g(\alpha)\in\vec{\PPP}\restriction(\alpha,\kappa)$ holds for all $\alpha\in F$. Since partial orders of the form $\vec{\PPP}\restriction[\alpha,\kappa)$ with $\alpha<\kappa$ are atomless, we can find functions $h_\gamma,h_\delta\in\VV[c]$ with domain $\kappa$ with the property that $h_\gamma(\alpha)$ and $h_\delta(\alpha)$ are incompatible conditions below $h(\alpha)$ in $\vec{\PPP}\restriction[\alpha,\kappa)$ for all $\alpha\in F$.  
Then, there is a $\vec{\PPP}$-name $\dot{B}\in\VV[c]$ with the property that whenever $K$ is $\vec{\PPP}$-generic over $\VV[c]$, then $\dot{B}^K=\Set{\alpha\in\dot{A}}{h_\gamma(\alpha)\in K}$. Then $\vec{p}_1\cup h_\gamma(\alpha)\Vdash^{\VV[c]}_{\vec{\PPP}}\anf{\check{\alpha}\in\dot{B}\subseteq\dot{A}}$ and $\vec{p}_1\cup h_\delta(\alpha)\Vdash^{\VV[c]}_{\vec{\PPP}}\anf{\check{\alpha}\in\dot{A}\setminus\dot{B}}$ for all $\alpha\in E\cap F$. Moreover, there is a condition $(\vec{p},\vec{s})$ below $(\vec{p}_1,\vec{s}_1)$ in $\vec{\PPP}\times\vec{\SSS}$ with $\vec{s}(U)(\gamma)=[h_\gamma]_U$ and $\vec{s}(U)(\delta)=[h_\delta]_U$. A genericity argument now shows that there is $B\in U_\gamma\cap\POT{\dot{A}^G}$ with the property that $\dot{A}^G\setminus B\in U_\delta$ for some $\delta<\lambda$. In particular, the condition $[\dot{A}^G]_\II$ is not an atom in the partial order $\POT{\kappa}/\II$ in $\VV[c,G,H]$. 
 \end{proof}

 \begin{claim*}
  If the ideal $\II^\kappa_{\prec{ms}}$ is precipitous in $\VV$, then the  ideal $\II$ is precipitous in $\VV[c,G,H]$. 
 \end{claim*}

 \begin{proof}[Proof of the Claim]
  A result of Kakuda (see {\cite[Theorem 1]{MR613283}}) shows that the set $$\Set{A\in\POT{\kappa}^{\VV[c]}}{\exists B\in(\II^\kappa_{\prec{ms}})^\VV ~ A\subseteq B}$$ is a precipitous ideal on $\kappa$ in $\VV[c]$. As observed above, this ideal is equal to $(\II^\kappa_{\prec{ms}})^{\VV[c]}$. 
 Since the partial order $\vec{\SSS}$ is ${<}\lambda$-closed in $\VV[c]$, this shows that $(\II^\kappa_{\prec{ms}})^{\VV[c]}$ is also a precipitous ideal on $\kappa$ in $\VV[c,H]$. 
Since the partial order $\vec{\PPP}$ satisfies the $\kappa$-chain condition in $\VV[c,H]$, another application of Kakuda's result shows that the set $\Set{A\in\POT{\kappa}^{\VV[c,G]}}{\exists B\in(\II^\kappa_{\prec{ms}})^{\VV[c]} ~ A\subseteq B}$ is a precipitous ideal on $\kappa$ in $\VV[c,G,H]$. By an above claim, this collection is equal to the ideal $\II$. 
 \end{proof}

 This completes the proof of the theorem. 
\end{proof}

%%%%%%%%%%%%%%%%%%%%%%%%%%%%%%%%%

\section{Concluding remarks and open questions}\label{section:questions}

For many large cardinal properties corresponding to some property of models and filters, 
 we either were able to show or it was known that the collection of all smaller cardinals without the given property is not contained in the induced ideal. 
 Since each of these arguments has its individual proof that relies on the specific large cardinal property, it is natural to ask whether this conclusion holds true in general.

\begin{question}
 Assume that Scheme \ref{schemeSmall} (respectively, Scheme \ref{schemeNoel} or Scheme \ref{schemeKappa}) holds true for some large cardinal property $\Phi(\kappa)$ and some property $\Psi(M,U)$ of models and filters. 
 Does $\NN^\kappa_\Phi\notin\II^{{<}\kappa}_\Psi$ (respectively, $\NN^\kappa_\Phi\notin\II^\kappa_\Psi$ or $\NN^\kappa_\Phi\notin\II^\kappa_{{\prec}\Psi}$) hold for every cardinal $\kappa$ with $\Phi(\kappa)$? 
\end{question}

For some large cardinal properties that can be characterized through Scheme \ref{schemeKappa}, we were not able to show that $\NN^\kappa_\Phi\notin\II^\kappa_{{\prec}\Psi}$ always holds. 
 The difficulties in these arguments are mostly caused by the fact that elementary submodels of some large $\HH{\theta}^M$ cannot be transferred between some weak $\kappa$-model $M$ and its ultrapowers. 
 In particular, the following statements are left open:

\begin{question}
 \begin{enumerate}
  \item Does $\NN^\kappa_{\omega R}\notin\II^\kappa_{{\prec}\textup{w}R}$ hold for every $\omega$-Ramsey cardinal $\kappa$? 
  
  \item Does $\NN^\kappa_{nR}\notin\II^\kappa_{{\prec}nR}$ hold for every $\delt^\forall_\omega$-Ramsey cardinal $\kappa$? 
 \end{enumerate}
\end{question}

 For Ramsey-like cardinals characterized through the validity of Scheme \ref{schemeNoel}, Lemma \ref{lemma:IdealsNicelyProper} is our main tool to show that $\NN^\kappa_\Phi\notin\II^\kappa_\Psi$. 
  Since we considered several properties of models and ultrafilters that are not absolute between $\VV$ and the corresponding ultrapowers, we naturally arrive at the following question:

\begin{question}
 Let $\kappa$ be a $\Psi^\vartheta_\alpha$-Ramsey cardinal with $\alpha\leq\kappa$, $\vartheta\in\{\kappa,\kappa^+\}$ and $\Psi\in\{\cc,\infi,\delt\}$. Is it true that $\NN\Psi^\vartheta_\alpha(\kappa)\notin\II\Psi^\vartheta_\alpha(\kappa)$? 
\end{question}

 The individual results of our paper strongly support the idea that most natural large cardinal notions below measurability canonically induce large cardinal ideals in a way that the relationship between those ideals reflects the relationship between the corresponding large cardinal notions, as is examplified by the results listed in Theorem \ref{theorem:IdealContain}. 
Therefore, it is natural to ask whether this can be done more generally:

\begin{question}
 Given large cardinal properties $\Phi_0$ and $\Phi_1$ and properties $\Psi_0$ and $\Psi_1$ of models and filters that are each connected through one of our characterization schemes in a canonical way,\footnote{Note that, in the case of weak compactness, the results of Sections \ref{section:WCamenablecomplete} and \ref{section:weaklycompactideal} already show that these questions can have negative answers by showing that the various characterizations of weak compactness induce both the bounded and the weakly compact ideal. Therefore, it does not make sense to consider these questions for arbitrary instances of our characterization schemes.} \ldots  
 \begin{itemize} 
  \item \ldots is it true that $\Phi_0(\kappa)$ provably implies $\Phi_1(\kappa)$ for every cardinal $\kappa$ if and only if it can be proven that for every cardinal $\kappa$ satisfying $\Phi_0(\kappa)$, the ideal on $\kappa$ induced by $\Phi_1$ and $\Psi_1$  is contained in the ideal on $\kappa$ induced by $\Phi_0$ and $\Psi_0$?

  \item \ldots is it true that $\Phi_0$ has strictly larger consistency strength than $\Phi_1$ if and only if it can be proven that for every cardinal $\kappa$ satisfying $\Phi_0(\kappa)$, the set $\NN^\kappa_{\Phi_1}$ is an element of the ideal on $\kappa$ induced by $\Phi_0$ and $\Psi_0$? 
 \end{itemize}
\end{question}

In this paper, we introduced several new large cardinal concepts, whose relationships to each other are only partially established. 
For example, every $\Delta_\omega^\forall$-Ramsey cardinal is trivially $\cc_\omega^\forall$-Ramsey. However, we do not know if this implication can be reversed, and, in the light of Proposition \ref{RamseysEquivalence}, we ask the following:

\begin{question}
  Are $\cc_\omega^\forall$-Ramseyness and $\Delta_\omega^\forall$-Ramseyness distinct large cardinal notions? Are their consistency strengths distinct?
\end{question}

The following related question is also open:

\begin{question}
  Are ineffable Ramseyness, $\infi_\omega^\kappa$-Ramseyness, and $\delt_\omega^\kappa$-Ramseyness distinct large cardinal notions? Are their consistency strengths distinct? What is their relationship with $\cc_\omega^\forall$-Ramsey cardinals?
\end{question}

 The results of Section \ref{section:MeasurableIdeal} show that the atomicity of partial orders of the form $\POT{\kappa}/\II^\kappa_{{\prec}ms}$ depends on the ambient model of set theory. 
 In all models constructed in this section however, the ideal  $\II^\kappa_{{\prec}ms}$ is precipitous, which motivates the following question:\footnote{Let us mention the following two related results: Hellsten has shown (see \cite{MR2653962}) that for a weakly compact cardinal, the weakly compact ideal can consistently be precipitous. Johnson (see \cite{MR918427}) has shown that for a completely ineffable cardinal, the completely ineffable ideal is never precipitous.}

\begin{question}
 If $\kappa$ is a measurable cardinal, is the ideal $\II^\kappa_{{\prec}ms}$ precipitous? 
\end{question}

\bibliographystyle{plain}
\bibliography{references}

\end{document}